\documentclass{amsart}
\usepackage{epsfig}
\usepackage{graphicx}
\usepackage{color}
\def\R{\mathbb R}
\def\Z{\mathbb Z}
\def\N{\mathbb N}

\def\C{\mathbb C}

\def\romega{{r_{\phantom{a}\!\!\!\!_\Omega}}}
\def\rgamma{{r_{\phantom{a}\!\!\!\!_\Gamma}}}

\def\negquad{\!\!\!\!}

\def\eps{\varepsilon}

\DeclareMathOperator\essinf{ess\,inf}

\newcommand{\cal}[1]{{\mathcal #1}}
\numberwithin{equation}{section}

\newtheorem{thm}{Theorem}
\numberwithin{thm}{section}
\newtheorem{cor}{Corollary}
\numberwithin{cor}{section}
\newtheorem{lem}{Lemma}
\numberwithin{lem}{section}

\numberwithin{prop}{section}
\theoremstyle{definition}
\newtheorem{definition}{Definition}
\numberwithin{definition}{section}
\theoremstyle{remark}
\newtheorem{rem}{Remark}
\numberwithin{rem}{section}

\begin{document}
\title[On the wave equation with hyperbolic...]
{On the wave equation with hyperbolic dynamical  boundary
conditions, interior and boundary damping and source}
\author{Enzo Vitillaro}
\address[E.~Vitillaro]
       {Dipartimento di Matematica ed Informatica, Universit\`a di Perugia\\
       Via Vanvitelli,1 06123 Perugia ITALY}
\email{enzo.vitillaro@unipg.it}
\date{\today}



\thanks{ The author would like to convey his sincerest thanks to the anonymous
reviewers, whose comments helped him to improve the presentation of the paper.  Work done in the framework of the M.I.U.R. project
"Variational and perturbative aspects of nonlinear differential
problems" (Italy)}

\begin{abstract} The aim of this paper is to study  the problem
$$
\begin{cases} u_{tt}-\Delta u+P(x,u_t)=f(x,u) \qquad &\text{in
$(0,\infty)\times\Omega$,}\\
u=0 &\text{on $(0,\infty)\times \Gamma_0$,}\\
u_{tt}+\partial_\nu u-\Delta_\Gamma u+Q(x,u_t)=g(x,u)\qquad
&\text{on
$(0,\infty)\times \Gamma_1$,}\\
u(0,x)=u_0(x),\quad u_t(0,x)=u_1(x) &
 \text{in $\overline{\Omega}$,}
\end{cases}$$
where $\Omega$ is a open  bounded  subset of $\R^N$  with $C^1$ boundary  ($N\ge 2$),
$\Gamma=\partial\Omega$, $(\Gamma_0,\Gamma_1)$ is a measurable
partition of $\Gamma$,  $\Delta_\Gamma$ denotes the
Laplace--Beltrami operator on $\Gamma$, $\nu$ is the outward normal
to $\Omega$, and the terms $P$ and $Q$ represent nonlinear damping
terms, while $f$ and $g$ are nonlinear subcritical perturbations.

In the paper a local Hadamard well--posedness result for
initial data in the natural energy space associated to the problem
is given. Moreover, when $\Omega$ is  $C^2$  and
$\overline{\Gamma_0}\cap\overline{\Gamma_1}=\emptyset$,  the
regularity of solutions is studied.
 Next a blow--up theorem is given when
$P$ and $Q$ are linear and $f$, $g$ are superlinear sources. Finally a dynamical system is generated when
the source parts of $f$ and $g$ are at most linear at infinity, or they are dominated by the damping terms.

\end{abstract}

\maketitle
\section{Introduction and main result} \label{intro}
\noindent We deal with the evolution problem consisting of the
 wave equation posed in a bounded regular open subset  of
$\R^N$, supplied with a second order dynamical boundary condition of
hyperbolic type, in presence of interior and/or boundary damping
terms and sources. More precisely we consider the
initial--and--boundary value problem
\begin{equation}\label{1}
\begin{cases} u_{tt}-\Delta u+P(x,u_t)=f(x,u) \qquad &\text{in
$(0,\infty)\times\Omega$,}\\
u=0 &\text{on $(0,\infty)\times \Gamma_0$,}\\
u_{tt}+\partial_\nu u-\Delta_\Gamma u+Q(x,u_t)=g(x,u)\qquad
&\text{on
$(0,\infty)\times \Gamma_1$,}\\
u(0,x)=u_0(x),\quad u_t(0,x)=u_1(x) &
 \text{in $\overline{\Omega}$,}
\end{cases}
\end{equation}
where $\Omega$ is a  open  bounded subset
of $\R^N$  ($N\ge 2$)  with $C^1$ boundary (see \cite{grisvard}).  We denote  $\Gamma=\partial\Omega$ and we
assume $\Gamma=\Gamma_0\cup\Gamma_1$,
$\Gamma_0\cap\Gamma_1=\emptyset$,
$\Gamma_1$ being relatively open on $\Gamma$ (or equivalently  $\overline{\Gamma_0}=\Gamma_0$).  Moreover, denoting by $\sigma$ the standard Lebesgue hypersurface
measure on $\Gamma$,
we assume that $\sigma(\overline{\Gamma}_0\cap\overline{\Gamma}_1)=0$. These
properties of $\Omega$, $\Gamma_0$ and $\Gamma_1$ will be assumed,
without further comments, throughout the paper. In Section \ref{section5} we shall
restrict to  open bounded subsets with $C^2$ boundary  and to partitions such that $\overline{\Gamma_0}\cap
\overline{\Gamma_1}=\emptyset$.
Moreover $u=u(t,x)$, $t\ge 0$, $x\in\Omega$,
$\Delta=\Delta_x$ denotes the Laplace operator  with  respect to the
space variable, while $\Delta_\Gamma$ denotes the Laplace--Beltrami
operator on $\Gamma$ and $\nu$ is the outward normal to $\Omega$.
The terms $P$ and $Q$ represent nonlinear damping terms, i.e.
$P(x,v)v\ge 0$, $Q(x,v)v\ge 0$,  the cases $P\equiv 0$ and $Q\equiv
0$ being specifically allowed, while  $f$ and $g$  represent
nonlinear source, or sink, terms. The specific assumptions on them
will be introduce d  later on.

Problems with kinetic boundary conditions, that is boundary
conditions involving $u_{tt}$,  on $\Gamma$ or on a part of it,
naturally arise in several physical applications. A one dimensional
model was studied by several authors
to describe transversal small oscillations of an elastic rod with a
tip mass on one endpoint, while the other one is pinched.
See \cite{andrewskuttlershillor,conradmorgul,darmvanhorssen,guoxu,morgulraoconrad}.

 A two dimensional model introduced in
\cite{goldsteingiselle} deals with a vibrating membrane of surface
density $\mu$, subject to a tension $T$, both taken constant and
normalized here for simplicity.  If $u(t,x)$,
$x\in\Omega\subset\R^2$  denotes the vertical displacement from the
rest state, then (after a standard linear approximation) $u$
satisfies the wave equation $u_{tt}-\Delta u=0$,
$(t,x)\in\R\times\Omega$. Now suppose that a part $\Gamma_0$ of the
boundary  is pinched, while the other part $\Gamma_1$ carries a
constant linear mass density $m>0$ and  it is subject to a linear
tension $\tau$. A practical example of this situation is given by a
drumhead with a hole in the interior having  a thick border, as
common in bass drums. One linearly approximates the force exerted by
the membrane on the boundary with $-\partial_\nu u$. The boundary
condition thus reads as $mu_{tt}+\partial_\nu u-\tau
\Delta_{\Gamma_1}u=0$. In the quoted paper the case
$\Gamma_0=\emptyset$ and $\tau=0$ was studied, while here we
consider the more realistic case $\Gamma_0\not=\emptyset$ and
$\tau>0$, with $\tau$ and $m$ normalized for simplicity. We would
like to mention that this model belongs to a more general class of
models of Lagrangian type involving boundary energies, as introduced
for example in \cite{fagotinreyes}.

A three dimensional model involving kinetic dynamical boundary
conditions comes out from \cite{GGG}, where a gas undergoing small
irrotational perturbations from rest in a domain $\Omega\subset\R^3$
is considered. Normalizing the constant speed of propagation, the
velocity potential $\phi$ of the gas (i.e. $-\nabla \phi$ is the
particle velocity) satisfies the wave equation $\phi_{tt}-\Delta
\phi=0$ in $\R\times\Omega$. Each point $x\in\partial\Omega$ is
assumed to react to the excess pressure of the acoustic wave like a
resistive harmonic oscillator or spring, that is the boundary is
assumed to be locally reacting (see \cite[pp.
259--264]{morseingard}). The normal displacement $\delta$ of the
boundary into the domain then satisfies
$m\delta_{tt}+d\delta_t+k\delta+\rho\phi_t=0$, where $\rho>0$ is the
fluid density  and $m,d,k\in C(\partial\Omega)$, $m,k>0$, $d\ge 0$.
When the boundary is nonporous one has $\delta_t=\partial_{\nu}\phi$
on $\R\times\partial\Omega$, so the boundary condition reads as
$m\delta_{tt}+d\partial_\nu \phi+k\delta+\rho\phi_t=0$. In the
particular case $m=k$ and $d=\rho$  (see  \cite[Theorem~2]{GGG}) one
proves that $\phi_{|\Gamma}=\delta$, so the boundary condition reads
as  $m\phi_{tt}+d\partial_\nu \phi+k\phi+\rho\phi_t=0$, on
$\R\times\partial\Omega$.
 Now, if one consider s  the case in which the
boundary is not locally reacting, as in \cite{beale}, one has to had
a Laplace--Beltrami term  so getting an hyperbolic dynamical
boundary condition like the one in \eqref{1}.

Several papers in the literature deal with the wave equation  with
kinetic boundary conditions. This fact is even more evident  if one
takes into account that, plugging the equation in \eqref{1} into the
boundary condition, we can rewrite it as  $\Delta u +\partial_\nu
u-\Delta_\Gamma u+ Q(x,u_t)+P(x,u_t)=f(x,u)+g(x,u)$. Such a
condition is usually called a {\em generalized Wentzell boundary
condition}, at least when nonlinear perturbations are not present.
We refer to  \cite{mugnolo2011}, where abstract semigroup techniques are applied to dissipative wave equations, and to  \cite{doroninlarkinsouza,FGGGR,vazvitM3AS,xiaoliang2,xiaoliang1}.
All of them deal  either with the case $\tau=0$ or with linear
problems.

Here we shall consider this type of kinetic boundary condition in
connection with nonlinear boundary damping and source terms. These
terms have been considered by several authors, but mainly in
connection with first order dynamical boundary conditions. See
\cite{MR2674175,
MR2645989,bociuNAMA,bociulasiecka2,bociulasiecka1,CDCL,CDCM,chueshovellerlasiecka,lastat,rendiconti,global}.
The competition between interior damping and source terms is
methodologically related to the competition between boundary damping
and source and it possesses a large literature as well. See
\cite{MR2609953,georgiev,levserr,ps:private,radu1,STV,blowup}.

 Problem \eqref{1} has been recently
introduced by the author in \cite{AMS}, dealing with a preliminary
analysis of \eqref{1} in the particular the case $P=0$, $f=0$,
$Q=|u_t|^{\mu-2}u_t$, $g=|u|^{q-2}u$, $\mu>1$, $q\ge 2$. When
$\Omega$ is $C^2$, $\overline{\Gamma}_0\cap
\overline{\Gamma}_1=\emptyset$, so $\Gamma$ is disconnected, both
$Q$ and $g$ are subcritical with respect to the Sobolev embedding on
$\Gamma$,  and ${u_0}\in H^2(\Omega)$, ${u_0}_{|\Gamma_1}\in
H^2(\Gamma_1)$, $u_{1,\Omega}\in H^1(\Omega)$,
$u_{1,\Gamma_1}={u_{1,\Omega}}_{\Gamma_1}\in H^1(\Gamma_1)$, an
existence and uniqueness result is proved.  Moreover a linear
problem strongly related to \eqref{1} has also been recently studied
in \cite{graberlasiecka}, dealing with analiticity or  Gevrey
classification for the generated linear semigroup, and in \cite{Fourrier}, dealing with regularity and stability.

The aim of the present paper is to substantially generalize the
analysis made  in \cite{AMS} in several directions. At first we want
to treat in an unified framework  interior and/or internal source
and damping terms, each of which can vanish identically (the alternative being the study of several different problems).
 At second
we want to include supercritical boundary (as well as internal)
damping terms. Next we want to allow  $\Gamma$ to be
connected and just $C^1$ .  Moreover we want to
consider initial data in the natural energy space related to
\eqref{1} and thus weak solutions of it. Finally we plan to study
local Hadamard well-posedness. Several technical
problems, which were not present in \cite{AMS}, makes the analysis more involved. To best
illustrate our results we consider, in this section, the simplified
version of \eqref{1}
\begin{equation}\label{5}
\begin{cases} u_{tt}-\Delta u+\alpha(x)P_0(u_t)=f_0(u) \qquad &\text{in
$(0,\infty)\times\Omega$,}\\
u=0 &\text{on $(0,\infty)\times \Gamma_0$,}\\
u_{tt}+\partial_\nu u-\Delta_\Gamma u+\beta(x)Q_0(u_t)=g_0(u)\qquad
&\text{on
$(0,\infty)\times \Gamma_1$,}\\
u(0,x)=u_0(x),\quad u_t(0,x)=u_1(x) &
 \text{in $\overline{\Omega}$,}
\end{cases}
\end{equation}
where $\alpha\in L^\infty(\Omega)$, $\beta\in L^\infty(\Gamma_1)$,
$\alpha,\beta\ge0$,  conventionally taking $g_0\equiv 0$ when $\Gamma_1=\emptyset$,  and the following properties  are assumed:
\renewcommand{\labelenumi}{{(\Roman{enumi})}}
\begin{enumerate}
\item
$P_0$ and $Q_0$ are continuous and monotone increasing in $\R$,
$P_0(0)=Q_0(0)=0$, and there are $m,\mu >1$ such that
\begin{gather*}0<\liminf_{|v|\to \infty} \frac {|P_0(v)|}{|v|^{m-1}}  \le      \limsup_{|v|\to \infty} \frac {|P_0(v)|}{|v|^{m-1}}<\infty,\quad
\liminf_{|v|\to 0} \frac {|P_0(v)|}{|v|^{m-1}}>0,\\
0<\liminf_{|v|\to \infty} \frac {|Q_0(v)|}{|v|^{\mu-1}}  \le
\limsup_{|v|\to \infty} \frac {|Q_0(v)|}{|v|^{\mu-1}}<\infty,\quad
\liminf_{|v|\to 0} \frac {|Q_0(v)|}{|v|^{\mu-1}}>0;
\end{gather*}

\item $f_0,g_0\in C^{0,1}_{\text{loc}}(\R)$ and there are $p,q\ge 2$ such that
$|f_0'(u)|=O(|u|^{p-2})$ and  $|g_0'(u)|=O(|u|^{q-2})$ as
$|u|\to\infty$.
\end{enumerate}

Our model  nonlinearities satisfying (I--II) are given by
\begin{equation}\label{modelli}\left\{
\begin{alignedat}3 P_0(v)=&P_1(v):=a|v|^{\widetilde{m}-2}v+ |v|^{m-2}v,
 \quad&& 1<\widetilde{m}\le m,\quad &&a\ge 0,\\
Q_0(v)=&Q_1(v):=b|v|^{\widetilde{\mu}-2}v+ |v|^{\mu-2}v, \quad  &&
1<\widetilde{\mu}\le \mu, \quad &&b\ge 0,
\\
f_0(u)=&f_1(u):=\widetilde{\gamma}|u|^{\widetilde{p}-2}u+
\gamma|u|^{p-2}u+c_1,\quad &&2\le\widetilde{p}\le p,\quad
&&\widetilde{\gamma}, \gamma, c_1\in\R,
\\
g_0(u)=&g_1(u):=\widetilde{\delta}|u|^{\widetilde{q}-2}u+
\delta|u|^{q-2}u+c_2,\, \quad &&2\le\widetilde{q}\le q,\quad
&&\widetilde{\delta}, \delta, c_2\in\R.
\end{alignedat}
\right.\end{equation}
We introduce some basic notation. In the sequel we shall identify
$L^2(\Gamma_1)$ with its isometric image in $L^2(\Gamma)$, that is
\begin{equation}\label{ID}
L^2(\Gamma_1)=\{u\in L^2(\Gamma): u=0\,\,\text{a.e. on
}\,\,\Gamma_0\}.
\end{equation}
We set, for $\rho\in [1,\infty)$ and $\alpha\in L^\infty(\Omega)$,
$\beta\in L^\infty(\Gamma_1)$, $\alpha,\beta\ge 0$, the Banach
spaces
\begin{alignat*}2
&L^{2,\rho}_\alpha(\Omega)=\{u\in L^2(\Omega): \alpha^{1/\rho}u\in
L^\rho(\Omega)\}, \quad
&&\|\cdot\|_{2,\rho,\alpha}=\|\cdot\|_2+\|\alpha^{1/\rho}
\cdot\|_\rho,\\ &L^{2,\rho}_\beta(\Gamma_1)=\{u\in L^2(\Gamma_1):
\beta^{1/\rho}u\in L^\rho(\Gamma_1)\}, \quad
&&\|\cdot\|_{2,\rho,\beta}=\|\cdot\|_{2,\Gamma_1}+\|\beta^{1/\rho}
\cdot\|_{\rho,\Gamma_1},
\end{alignat*}
where
 $\|\cdot\|_\rho:=\|\cdot\|_{L^\rho(\Omega)}$ and
$\|\cdot\|_{\rho,\Gamma_1}:=\|\cdot\|_{L^\rho(\Gamma_1)}$
\begin{footnote}{it would appear simpler to set
$L^{2,\rho}_\alpha(\Omega)=L^2(\Omega)\cap
L^\rho(\Omega,\lambda_\alpha)$, but unfortunately when $\alpha$
vanishes in a set of positive measure that is wrong, since the
equivalence classes in the two intersecting spaces are different, as
it is clear in the extreme case $\alpha\equiv 0$.}\end{footnote}.

We denote by $u_{|\Gamma}$ the trace on $\Gamma$ of any $u\in
H^1(\Omega)$, and by $u_{|\Gamma_i}$ its restriction to $\Gamma_i$,
$i=0,1$. Moreover we introduce the Hilbert spaces $H^0 =
L^2(\Omega)\times L^2(\Gamma_1)$,
\begin{equation}\label{H1}
H^1 = \{(u,v)\in H^1(\Omega)\times H^1(\Gamma): v=u_{|\Gamma}, v=0
\,\,\ \text{on $\Gamma_0$}\},
\end{equation}
with the topology inherited from the products. For the sake of
simplicity we shall identify, when useful, $H^1$ with  its   isomorphic
counterpart $\{u\in H^1(\Omega): u_{|\Gamma}\in H^1(\Gamma)\cap
L^2(\Gamma_1)\}$, through the identification $(u,u_{|\Gamma})\mapsto
u$, so we shall write, without further mention, $u\in H^1$ for
functions defined on $\Omega$. Moreover we shall drop the notation
$u_{|\Gamma}$, when useful, so we shall write $\|u\|_{2,\Gamma}$,
$\int_\Gamma u$, and so on, for elements of $H^1$. We also
introduce, for $\alpha$ and $\beta$ as before and $\rho,\theta\in
[1,\infty]$, the Banach space
\begin{equation}\label{H1plus}
H^{1,\rho,\theta}_{\alpha,\beta}=H^1\cap
[L^{2,\rho}_\alpha(\Omega)\times L^{2,\theta}_\beta(\Gamma_1)],\quad
\|\cdot\|_{H^{1,\rho,\theta}_{\alpha,\beta}}=\|\cdot\|_{H^1}+\|\cdot\|_{L^{2,\rho}_\alpha(\Omega)\times
L^{2,\theta}_\beta(\Gamma_1)}.
\end{equation}
 Next, when $\Omega$ is $C^2$ and
$\rho\in [1,\infty]$, we denote
\begin{equation}\label{W2sigma}
W^{2,\rho}=[W^{2,\rho}(\Omega)\times W^{2,\rho}(\Gamma)]\cap
H^1,\qquad\text{and}\quad H^2=W^{2,2},
\end{equation}
endoweed with the norm inherited from the product.
 Finally we set $\romega$ and
$\rgamma$ to respectively be the critical exponents of the Sobolev
embeddings
\begin{footnote}{with the well--known exceptions for $\romega$ when $N=2$ and for $\rgamma$ when $N=3$. The embedding $H^1(\Gamma)\hookrightarrow
L^\rgamma(\Gamma)$ is standard in the $C^\infty$ setting, see for
example \cite[Theorem~2.6~p.32]{hebey}, and one easily sees that the
proof extends to $C^1$ manifolds without changes.}\end{footnote}
$H^1(\Omega)\hookrightarrow L^s(\Omega)$ and
$H^1(\Gamma)\hookrightarrow L^s(\Gamma)$, that is
$$\romega=
\begin{cases}
\dfrac {2N}{N-2} &\text{if $N \ge 3$,}\\ \infty &\text{if $N=2$},
\end{cases}
\qquad \rgamma=
\begin{cases}
\dfrac {2(N-1)}{N-3} &\text{if $N \ge 4$,}\\ \infty &\text{if $N=2,
3$}.
\end{cases}
$$
The first aim of the paper is to show that the problem \eqref{5} is
locally well-posed in the Hadamard sense in the phase space
$H^1\times H^0$ when $f_0$ and $g_0$ are subcritical in the sense of
semigroups.
\begin{thm}[\bf Local well--posedness in $\boldsymbol{H^1\times H^0}$] \label{theorem1}
If (I--II) hold and
\begin{equation}\label{6}2\le p\le 1+\romega / 2,\qquad  2\le q\le 1+\rgamma
/2,
\end{equation}
then the following conclusions hold.
\renewcommand{\labelenumi}{{(\roman{enumi})}}
\begin{enumerate}
\item
For any $(u_0,u_1)\in H^1\times H^0$  problem \eqref{5} has a unique
maximal weak solution $u$ in $[0,T_{\text{max}})$, that is
\begin{align}
&u=(u,u_{|\Gamma})\in
L^\infty_{\text{loc}}([0,T_{\text{max}});H^1)\cap
W^{1,\infty}_{\text{loc}}([0,T_{\text{max}});H^0), \\
\label{8} & u'=(u_t,(u_{|\Gamma})_t)\in
L^m_{\text{loc}}([0,T_{\text{max}});L^{2,m}_\alpha(\Omega))\times
L^\mu_{\text{loc}}([0,T_{\text{max}});L^{2,\mu}_\beta(\Gamma_1)),
\end{align}
which satisfies \eqref{5} in a distribution sense to be specified
later on;
\item $u$ enjoys the regularity
\begin{equation}\label{9}
u\in C([0,T_{\text{max}});H^1)\cap C^1([0,T_{\text{max}});H^0)
\end{equation}
and satisfies, for $0\le s\le t<T_{\text{max}}$, the energy identity
\begin{footnote}{$\nabla_\Gamma$
denotes the Riemannian gradient on $\Gamma$ and $|\cdot|_\Gamma$,
the norm associated to the Riemannian scalar product on the tangent
bundle of $\Gamma$. See Section \ref{section 2}.}\end{footnote}
\begin{equation*}\label{10}
\begin{split}
\frac 12 \left[\int_\Omega u_t^2(\tau)\! + \!\int_{\Gamma_1}
(u_{|\Gamma})_t^2(\tau) \!+ \!\int_\Omega |\nabla u(\tau)|^2 \!+\!
\int_{\Gamma_1} |\nabla_\Gamma  u(\tau)|_\Gamma ^2\right]_s^t\!\!
\!+ \!\int_s^t\!\!\!\int_\Omega \alpha
P_0(u_t)u_t\\+\int_s^t\int_{\Gamma_1} \beta
Q_0((u_{|\Gamma})_t))(u_{|\Gamma})_t=\int_s^t\int_{\Omega}
f_0(u)u_t+\int_s^t\int_{\Gamma_1} g_0(u)(u_{|\Gamma})_t;
\end{split}
\end{equation*}

\item
if  $T_{\text{max}}<\infty$ then
\begin{equation}\label{11} \lim_{t\to
T^-_{\text{max}}}\|u(t)\|_{H^1(\Omega)}+\|u(t)\|_{H^1(\Gamma
)}+\|u_t(t)\|_{L^2(\Omega)}+\|(u_{|\Gamma})_t(t)\|_{L^2(\Gamma_1)}=\infty;
\end{equation}

\item
if $u_{0n}\to u_0$ in $H^1$, $u_{1n}\to u_1$ in $H^0$ and we
respectively denote by $u_n\in C([0,T^n_{\text{max}});H^1)$ and
$u\in C([0,T_{\text{max}});H^1)$ the weak maximal solutions of
problem \eqref{5} corresponding to initial data $(u_{0n}, u_{1n})$
and $(u_0, u_1)$, we have $T_{\text{max}}\le\liminf_n
T^n_{\text{max}}$ and, for any $T\in (0,T_{\text{max}})$, $$u_n\to
u\quad\text{in $C([0,T];H^1)\cap C^1([0,T];H^0)$.}$$
\end{enumerate}
\end{thm}
\begin{figure}\label{Fig1}
\includegraphics[width=13cm]{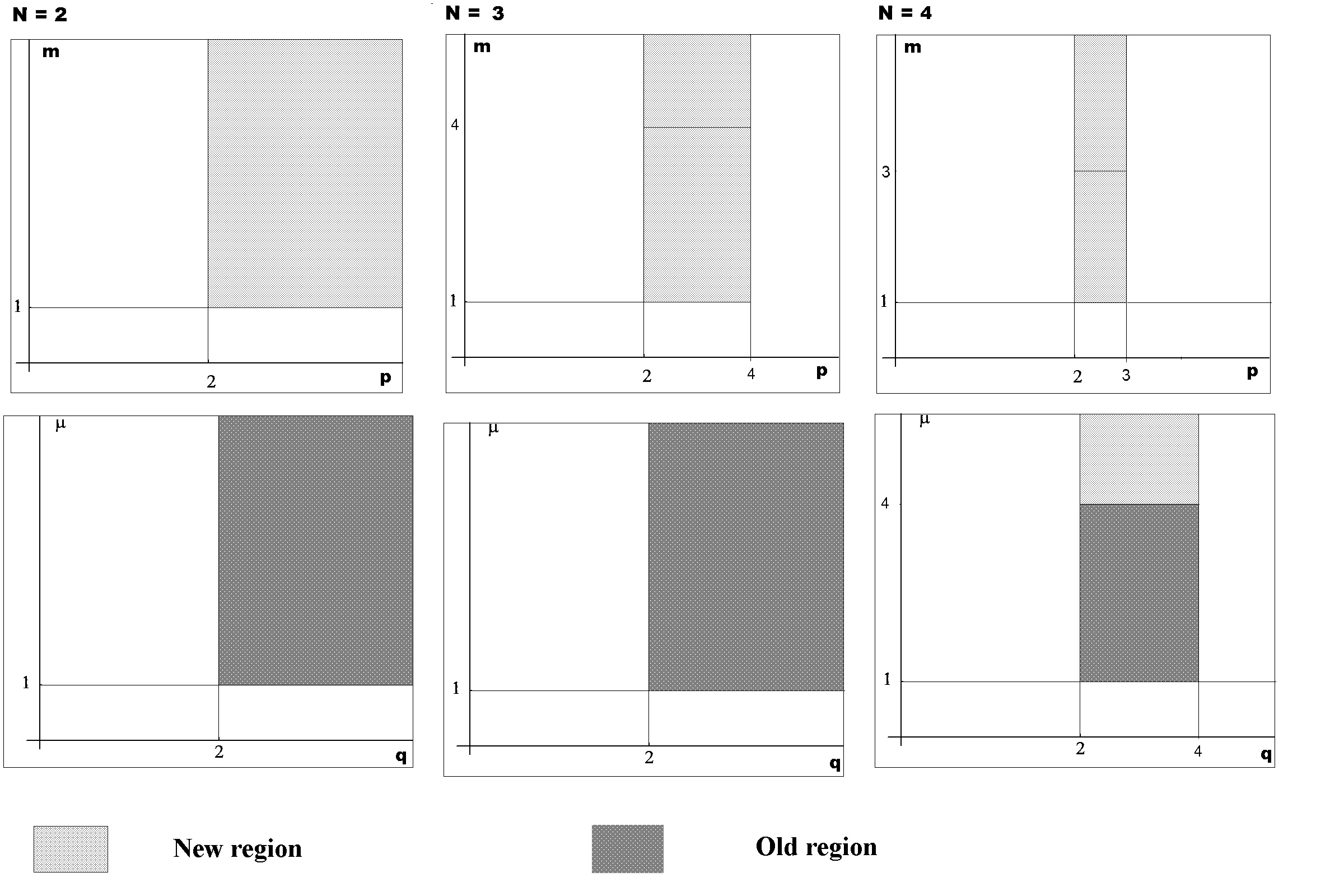}
\caption{The old region is the parameter range treated in
\cite{AMS}, while the new region is the range covered only by
Theorem~\ref{theorem1}.}
\end{figure}
\begin{rem} Figure~1 illustrates the parameter ranges covered by
Theorem~\ref{theorem1} and by \cite[Theorem~1]{AMS} in dimensions
$N=2,3,4$. Theorem~\ref{theorem1} is optimal when $N=2$, while when
$N=3$ assumption \eqref{6} imposes a severe restriction of the
growth of the internal source/sink. When $N=4$ the restriction concerns
both sources/sinks  and is even more severe. To relax assumption \eqref{6}
requires a specific analysis which is outside the aim of the present
paper. On the other hand there is no restriction on the growth of
the damping terms, improving the analysis in \cite{AMS}.
\end{rem}
As a simple byproduct of the arguments used to prove
Theorem~\ref{theorem1}  we  get a well--posedness result in a
stronger (when $m>\romega$ or $\mu>\rgamma$) topology provided $P_0$
and $Q_0$ satisfy the further assumption, trivially satisfied by $P_1$, $Q_1$ in \eqref{modelli},
\begin{enumerate}
\setcounter{enumi}2
\item $\liminf_{|v|\to \infty} \frac
{|P'_0(v)|}{|v|^{m-2}}>0$ if $m>\romega$, $\liminf_{|v|\to \infty} \frac {|Q'_0(v)|}{|v|^{\mu-2}}>0$ if $\mu>\rgamma$.
\end{enumerate}
\begin{thm}[\bf Local Hadamard well--posedness in $\boldsymbol{H^{1,\rho,\theta}_{\alpha,\beta}\times H^0}$] \label{theorem2}
If (I--III) and  \eqref{6} hold then, for any couple of exponents
\begin{equation}\label{special}
(\rho,\theta)\in [\romega,\max\{\romega,m\}]\times
[\rgamma,\max\{\rgamma,\mu \}]\end{equation}
 and any $(u_0,u_1)\in H^{1,\rho,\theta}_{\alpha,\beta}\times H^0$, the weak solution $u$ of
 problem \eqref{5} enjoys the further regularity
 \begin{equation}\label{9BIS}
u\in C([0,T_{\text{max}});H^{1,\rho,\theta}_{\alpha,\beta}).
\end{equation}
Moreover, if $u_{0n}\to u_0$ in $H^{1,\rho,\theta}_{\alpha,\beta}$,
$u_{1n}\to u_1$ in $H^0$, and $u_n$, $u$ are as in
Theorem~\ref{theorem1},  then
$$u_n\to u\text{ in $C([0,T];H^{1,\rho,\theta}_{\alpha,\beta})$ for any $T\in (0,T_{\text{max}})$.}$$
\end{thm}

Theorems~\ref{theorem1} and \ref{theorem2} can be easily extended to
more general second order uniformly elliptic linear operators, both
in $\Omega$ and $\Gamma$,  under suitable regularity assumptions on
the coefficients.  Here we prefer to deal with the Laplace and
Laplace--Beltrami operators for the sake of clearness. The proof
will rely on nonlinear semigroup theory (see \cite{barbu2010}
and \cite{showalter}), and in particular on \cite[Theorem~7.2,
Appendix]{chueshovellerlasiecka}, and on an easy consequence of
the approach used there,  which is outlined, for the reader's
convenience, in Appendix \ref{appendixA}.

The main difficulty faced in this approach consists in setting up,
and working with, the right pivot space which allows to get weak
solutions, i.e. solutions verifying the energy identity, when both
$\alpha$ and $\beta$ are allowed to vanish, identically or in  a  subset
of positive measure. Other approaches are possible, as for example
the use of a contraction argument, but  our approach  has
the advantage to set up a working framework useful for further
studies of the problem.

The first outcome of it is given by the
following regularity result, which proof constitutes the second aim
of the paper. Before stating it we introduce the exponents
$l=l(m,\mu, N)$ and $\lambda=\lambda(m,\mu,N)$ by
\begin{equation}\label{ls}
l=\min\left\{2,\tfrac{\max\{m,\romega\}}{m-1},\tfrac{\max\{\mu,\rgamma\}}{\mu-1}\right\},\,
\lambda=\begin{cases} \infty &\text{if $m\le \romega, \mu\le \rgamma$,}\\
\min\{m',\mu'\}&\text{otherwise.}
\end{cases}\end{equation}
\begin{thm}[\bf Regularity I] \label{theorem3}
Suppose that (I-II) and \eqref{6} hold true, that $\Omega$ is $C^2$
and $\overline{\Gamma}_0\cap\overline{\Gamma}_1=\emptyset$. Then, if
\begin{gather}\label{14}
(u_0,u_1)\in W^{2,l}\times H^{1,m,\mu}_{\alpha,\beta},\\
\label{15}
 -\Delta u_0+\alpha P_0(u_1)\in
 L^2(\Omega),\quad\partial_\nu
{u_0}_{|\Gamma_1}-\Delta_{\Gamma}{u_0}_{|\Gamma_1}+\beta
Q_0({u_1}_{|\Gamma})\in L^2(\Gamma_1),
\end{gather}
 then  the weak maximal solution $u$ of problem \eqref{5} found in
Theorem~\ref{theorem1} enjoys the further regularity
\begin{gather}\label{16}
u\in L^\lambda([0,T_{\text{max}}); W^{2,l})\cap
C^1_w([0,T_{\text{max}});H^1)\cap W^{2,\infty}_{\text{loc}}([0,T_{\text{max}});H^0),\\
u'\in C_w([0,T_{\text{max}});H^{1,m,\mu}_{\alpha,\beta}).
\end{gather}
Moreover, $u_{tt}-\Delta u+\alpha P_0(u_t)=f_0(u)$ in $L^l(\Omega)$,
a.e. in $(0,T_{\text{max}})$, and $(u_{|\Gamma})_{tt}+\partial_\nu
u-\Delta_\Gamma u_{|\Gamma}+\beta
Q_0((u_{|\Gamma})_t)=g_0(u_{|\Gamma})$ in $L^l(\Gamma_1)$, a.e. in
$(0,T_{\text{max}})$.

If $(u_0,u_1)\in [W^{2,l}\cap H^{1,m,\mu}_{\alpha,\beta}]\times
H^{1,m,\mu}_{\alpha,\beta}$ and \eqref{15} holds,  then  \eqref{16} becomes
$$u\in L^\lambda([0,T_{\text{max}}); W^{2,l}\cap H^{1,m,\mu}_{\alpha,\beta} )\cap
C^1_w([0,T_{\text{max}});H^{1,m,\mu}_{\alpha,\beta})\cap
W^{2,\infty}_{\text{loc}}([0,T_{\text{max}});H^0).$$
\end{thm}
The regularity \eqref{15} is improved, depending on the growth
of $P_0$, $Q_0$, as follows.
\begin{thm}[\bf Regularity II]\label{theorem4}
Suppose that (I-II) and \eqref{6} hold true, that $\Omega$ is $C^2$
and $\overline{\Gamma}_0\cap\overline{\Gamma}_1=\emptyset$. Moreover
suppose that
\begin{equation}\label{17}
1<m\le \romega,\qquad \text{and}\quad 1<\mu\le
\rgamma.\end{equation} Then for initial data satisfying
\eqref{14}--\eqref{15} the weak maximal solution $u$ of problem
\eqref{5} found in Theorem~\ref{theorem1} enjoys the regularity
\begin{equation}\label{18}
u\in C_w([0,T_{\text{max}}); W^{2,l})\cap
C^1_w([0,T_{\text{max}});H^1)\cap C^2_w([0,T_{\text{max}});H^0).
\end{equation}
In particular, when
\begin{equation}\label{19}
1<m\le 1+\romega/2,\qquad \text{and}\quad 1<\mu\le
1+\rgamma/2,\end{equation} for initial data $(u_0,u_1)\in H^2\times
H^1$ we have the optimal regularity
\begin{equation}\label{20}
u\in C_w([0,T_{\text{max}}); H^2)\cap
C^1_w([0,T_{\text{max}});H^1)\cap C^2_w([0,T_{\text{max}});H^0).
\end{equation}
\end{thm}
\begin{rem} In the particular case \eqref{19} Theorem~\ref{theorem4}
sharply extends \cite[Theorem 1]{AMS}, dealing with the case
$\alpha\equiv 0$, $P_0=f_0\equiv 0$, $\beta\equiv 1$,
$Q_0(v)=|v|^{\mu-2}v$, $g_0(u)=|u|^{q-2}u$.\end{rem}
The main difficulty in the proof of Theorems \ref{theorem3}--\ref{theorem4} consists in getting
the regularity with respect to the
space variable on $\Gamma_1$ expressed by \eqref{18}, especially
when \eqref{17} fails to hold and $\Omega$ is merely $C^2$.

 The  third aim of the paper  is to show that, under suitable assumptions on the nonlinearities involved beside (I--II) and \eqref{6}, the semi--flow generated by problem \eqref{5} is a dynamical system in the phase space $H^1\times H^0$ and, when also (III) holds, in  $H^{1,\rho,\theta}_{\alpha,\beta}\times H^0$ for $(\rho,\theta)$ verifying \eqref{special}. By Theorems \ref{theorem1}--\ref{theorem2} these assertions hold true if and only if  $T_{\max}=\infty$ for all $(u_0,u_1)\in H^1\times H^0$.

To motivate the need of additional assumptions we shall preliminarily show that assumptions (I--II) and \eqref{6} do not guarantee by themselves that all solutions are global in time, since  in some cases they blow--up in finite time.
To shortly prove this assertion  we temporarily restrict to linear damping terms, that is
we replace assumption (I) with the following one:
\renewcommand{\labelenumi}{{(I)$'$}}
\begin{enumerate}
\item $P_0(v)=Q_0(v)=v$ for all $v\in\R$.
\end{enumerate}

To state our blow--up result we introduce
\begin{equation}\label{23}
\mathfrak{F}_0(u)=\int_0^u f_0(s)\,ds,\qquad \mathfrak{G}_0(u)=\int_0^u g_0(s)\,ds\qquad\text{for all $u\in\R$,}
\end{equation}
and we make the following specific blow--up assumption:
\renewcommand{\labelenumi}{{(\Roman{enumi})}}
\begin{enumerate}
\setcounter{enumi}3
\item $(f_0,g_0)\not\equiv 0$ and there are $\overline{p}, \overline{q}>2$ such that
\begin{equation}\label{blowupinequalities}
f_0(u)u\ge \overline{p}\,\mathfrak{F}_0(u)\ge 0\quad\text{and}\quad
g_0(u)u\ge \overline{q}\,\mathfrak{G}_0(u)\ge 0\quad\text{for all $u\in\R$.}
\end{equation}
\end{enumerate}
\begin{rem}\label{remark1.4}
Clearly  $f_1$ and $g_1$ in \eqref{modelli} satisfy (IV) if and only if
\begin{equation}\label{blowupparametri}
c_1=c_2=0,\qquad \gamma, \widetilde{\gamma}, \delta,\widetilde{\delta}\ge 0,\qquad \gamma+\widetilde{\gamma}+\delta+\widetilde{\delta}>0,\qquad \text{and}\quad\widetilde{p},\widetilde{q}>2.
\end{equation}
\end{rem}
We also introduce the energy functional $\cal{E}_0\in C^1(H^1\times H^0)$ defined by
\begin{equation}\label{Eintroduction}
\cal{E}_0(u_0,u_1)=\frac 12 \|u_1\|_{H^0}^2+\tfrac 12\int_\Omega |\nabla u_0|^2+\tfrac 12\int_{\Gamma_1}\!\!|\nabla_{\Gamma} u_0|_{\Gamma}^2-\int_\Omega\!\!\mathfrak{F}_0(u_0)-\int_{\Gamma_1}\!\!\mathfrak{G}_0(u_0).
 \end{equation}
\begin{thm}[\bf Blow--up]\label{theorem1.5} Let (I)$'$, (II), (IV) and \eqref{6} hold. Then
\renewcommand{\labelenumi}{{(\roman{enumi})}}
\begin{enumerate}
\item
$N_0:=\{(u_0,u_1)\in H^1\times H^0: \,\,\cal{E}_0(u_0,u_1)<0\}\not=\emptyset$, and
\item for any  $(u_0,u_1)\in N_0$  the unique maximal weak solution $u$ of \eqref{1} blows--up in finite time, that is $T_{\max}<\infty$, and
\begin{equation}
\label{new6.4}
\lim\limits_{t\to T^-_{\text{max}}}
\|u(t)\|_{H^1}+\|u'(t)\|_{H^0}=\lim\limits_{t\to T_{\text{max}}^-}\|u(t)\|_p^p+\|u(t)\|_{q,\Gamma_1}^q=\infty.
\end{equation}
\end{enumerate}
 \end{thm}
\begin{rem} When $(f_0,g_0)\equiv 0$ the set $N_0$ is trivially empty, and all solutions are global in time, as it will be clear from
Theorem~\ref{theorem1.6}. The two cases
$f_0\not\equiv 0$, $g_0\equiv 0$ and $f_0\equiv 0$, $g_0\not\equiv 0$,
are of particular interest, since they show that just one source, internal or at the boundary, forces solutions to  blow--up.
\end{rem}
The proof of Theorem~\ref{theorem1.5} is based on Theorem~\ref{1} and on the classical concavity method of H. Levine.
In this way we give a first application of Theorem~\ref{1}.

Theorem \ref{theorem1.5} implicitly suggests that all solutions of \eqref{5} can be global in time when
$f_0$ and $g_0$ are sinks, that is $f_0(u)u,g_0(u)u\le 0$ in $\R$, or
they are sources, that is $f_0(u)u,g_0(u)u\ge 0$ in $\R$, with at  most linear growth at infinity.
It is reasonable to extend this conjecture to sums of terms of these types.  Moreover nonlocalized damping terms, whose growths at infinity
dominate those of the sources (when sources are superlinear), may also prevent solutions to blow--up in finite time.
To treat all these cases in a unified framework we shall make, beside (I--II) and \eqref{6}, the following specific global existence assumption:
\begin{enumerate}
\setcounter{enumi}4
\item there are $p_1$ and $q_1$ satisfying
\begin{equation}\label{V.2}2\le p_1\le \min\{p,\,\max\{2,m\}\}\quad\text{and}\quad  2\le q_1\le \min\{q,\,\max\{2,\mu\}\}
\end{equation}
and such that
\begin{equation}\label{V.1}
\varlimsup_{|u|\to\infty}\mathfrak{F}_0(u)/|u|^{p_1}<\infty\quad\text{and}\quad \varlimsup_{|u|\to\infty}\mathfrak{G}_0(u)/|u|^{q_1}<\infty.
\end{equation}
We also suppose that
\begin{equation}\label{V.3}
 \text{$\essinf_\Omega >0$ if $p_1>2$ and $\essinf_{\Gamma_1}\beta>0$ if $q_1>2$.}
\end{equation}
\end{enumerate}
Since $\mathfrak{F}_0(u)=\int_0^1 f_0(su)u\,ds$ (and similarly $\frak{G}_0$),
 (V) is a weak version
 \begin{footnote}{ Actually (V)$'$ is more general than (V). Indeed, when
$f_0(u)=(m+1)|u|^{m-1}u\cos |u|^{m+1}$ and $g_0(u)=(\mu+1)|u|^{\mu-1}u\cos |u|^{\mu+1}$,
\eqref{Vstrong} holds only for $p_1\ge m+1$, $q_1\ge \mu+1$, while \eqref{V.1} does with $p_1=q_1=2$.
}\end{footnote}
 of the following assumption, which is adequate for most purposes and easier to verify:
\renewcommand{\labelenumi}{{(V)$'$}}
\begin{enumerate}
\item there are $p_1$ and $q_1$ such that  \eqref{V.2} holds with \eqref{V.3} and
\begin{equation}\label{Vstrong}
\varlimsup_{|u|\to\infty}f_0(u)u/|u|^{p_1}<\infty\qquad\text{and}\quad \varlimsup_{|u|\to\infty}g_0(u)u/|u|^{q_1}<\infty.
\end{equation}
\end{enumerate}
\begin{rem}
Assumptions (II) and (V)$'$ hold true
when $f_0$ (respectively $g_0$) belongs to one among the following classes:
\renewcommand{\labelenumi}{{(\arabic{enumi})}}\begin{enumerate}
\setcounter{enumi}{-1}
\item $f_0$ (respectively $g_0$) is constant;
\item $f_0$ (respectively $g_0$) satisfies (II) with $p\le \max\{2,m\}$ and
$\essinf_\Omega \alpha>0$ if $p>2$ (respectively $q\le \max\{2,\mu\}$ and  $\essinf_{\Gamma_1}\beta>0$ if $q>2$);
\item $f_0$ (respectively $g_0$) satisfies (II) and it is a sink.
\end{enumerate}
More generally (II) and (V)$'$ hold when
\begin{equation}\label{centerdot}
f_0=f_0^0+f_0^1+f_0^2,\qquad\text{and}\quad g_0=g_0^0+g_0^1+g_0^2,
\end{equation}
where $f_0^i$ and $g_0^i$ are of class (i) for $i=0,1,2$.
\begin{footnote}{
Actually {\em all} functions verifying (II) and (V)$'$ are of the form \eqref{centerdot},
where $f_0^1$ are $g_0^1$ are sources. See Remark \ref{remark6.3ff}.}\end{footnote}
\end{rem}
\begin{rem}\label{remark1.5}
One easily checks that $f_1$ in \eqref{modelli} satisfies (II) and  (V) if and only if  one among the following cases (the analogous cases apply to  $g_1$) occurs:
\renewcommand{\labelenumi}{{(\roman{enumi})}}
\begin{enumerate}
\item $\gamma>0$, $p\le \max\{2,m\}$ and $\essinf_\Omega\alpha >0$ if $p>2$;
\item $\gamma\le 0$, $\widetilde{\gamma}>0$,  $\widetilde{p}\le \max\{2,m\}$ and
$\essinf_\Omega \alpha>0$ if $\widetilde{p}>2$;
\item $\gamma,\widetilde{\gamma}\le0$.
\end{enumerate}
\end{rem}
Our global existence result is the following one.
\begin{thm}[\bf Global existence]\label{theorem1.6} Let (I--II), (V) and \eqref{6} hold.
 Then for any  $(u_0,u_1)\in H^1\times H^0$  the unique maximal weak solution $u$ of \eqref{1} is global in time, that is $T_{\max}=\infty$.
Consequently the semi--flow generated by problem \eqref{5} is a dynamical system in $H^1\times H^0$ and, when also (III) holds, in  $H^{1,\rho,\theta}_{\alpha,\beta}\times H^0$ for $(\rho,\theta)$ verifying \eqref{special}.
\end{thm}
\begin{rem}
Comparing Remarks \ref{remark1.4} and \ref{remark1.5} it is clear that, when $P_1$ and $Q_1$ are linear (hence $m=\mu=2$),
$f_1$ and $g_1$ in \eqref{modelli}  cannot satisfy  (IV) {\em and} (V).
By integrating \eqref{blowupinequalities} one easily sees that the same fact holds for $f_0$ and $g_0$,
so the assumptions sets of Theorems~\ref{theorem1.5} and \ref{theorem1.6} have empty intersection, as expected.
Moreover, since  when $m=\mu=2$ and $\widetilde{\gamma}=\widetilde{\delta}=c_1=c_2=0$ then (V) holds if and only if $p=2$ when $\gamma>0$ and $q=2$ when $\delta>0$,  comparing with \eqref{blowupparametri} and remembering that $\widetilde{p}\le p$ and $\widetilde{q}\le q$, we see that
Theorem~\ref{theorem1.6} is  sharp when damping terms are linear and sources are pure powers.
The author is convinced that Theorem~\ref{theorem1.6} is sharp also when sources are not pure powers and damping terms are nonlinear.
He intends to study this topic in a forthcoming paper.
\end{rem}

The proof of Theorem~\ref{theorem1.6} relies on Theorems~\ref{theorem1} and \ref{theorem2} (so we are giving another application of them), on a suitable auxiliary functional inspired by  \cite{georgiev}
and on suitable estimates.

The paper is organized as follows: in
Section~\ref{section 2} we introduce some background material, we
set up the functional setting used and we  prove a couple of
preliminary results, one of which concerning weak solutions for a
linear version of \eqref{1}.  In Section~\ref{section3} we state our
main local well--posedness result for \eqref{1} and a slightly more
general (and abstract) version of it, which contains
Theorems~\ref{theorem1}--\ref{theorem2} as particular cases. They are proved in
Section~\ref{section4}.  Section~\ref{section5} is devoted to regularity results for problem
\eqref{1} and the proofs of Theorems~\ref{theorem3}--\ref{theorem4}.
 In Section~\ref{section6} we give our blow--up and global existence results for problem
\eqref{1} and the proofs of Theorems~\ref{theorem1.5}--\ref{theorem1.6}.
Finally, Appendix~\ref{appendixA} deals with abstract Cauchy problems for locally Lipschitz perturbations of maximal monotone operators,
while in Appendix~\ref{appendixB} we give the proof of the isomorphism property of the operator $-\Delta_M+I$ on a  $C^2$ manifold $M$.

\section{\bf Background and preliminary results}\label{section 2}
\subsection{Notation.} We shall adopt the standard notation for
functions spaces in $\Omega$ such as the Lebesgue and Sobolev spaces
of integer order, for which we refer to \cite{adams}. All the
function spaces considered in the paper will be spaces of real
valued functions, but for Appendix~\ref{appendixB} where
complex--valued functions will be considered.  Given a Banach space
$X$ and its dual $X'$ we shall denote by $\langle
\cdot,\cdot\rangle_X$ the duality product between them. Finally, we
shall use the standard notation for vector valued Lebesgue and
Sobolev spaces in a real interval.

 Given $\alpha\in
L^\infty(\Omega)$, $\beta\in L^\infty(\Gamma)$, $\alpha,\beta\ge 0$
and $\rho\in [1,\infty]$ we shall respectively denote by
$(L^\rho(\Omega),\|\cdot\|_\rho)$, $(L^\rho(\Gamma),\|\cdot\|_{\rho,\Gamma})$, $(L^\rho(\Gamma_1),\|\cdot\|_{\rho,\Gamma_1})$,
$(L^\rho(\Omega;\lambda_\alpha),\|\cdot\|_{\rho,\alpha})$, $(L^\rho(\Gamma;\lambda_\beta),\|\cdot\|_{\rho,\beta,\Gamma})$
and $(L^\rho(\Gamma_1;\lambda_\beta),\|\cdot\|_{\rho,\beta,\Gamma_1})$
the  Lebesgue spaces (and norms) with respect to the following
measures: the standard Lebesgue one in $\Omega$, the hypersurface
measure $\sigma$ on $\Gamma$ and $\Gamma_1$, $\lambda_\alpha$ in
$\Omega$ defined (for Lebesgue measurable sets) by
$d\lambda_\alpha=\lambda_\alpha\,dx$,
 $\lambda_\beta$ on $\Gamma$ and $\Gamma_1$ defined (for $\sigma$ measurable sets) by
$d\lambda_\beta=\lambda_\beta\,d\sigma$. The equivalence classes with respect
to the measures $\lambda_\alpha$ and $\lambda_\beta$ will be
respectively denoted by $[\cdot]_\alpha$ and $[\cdot]_\beta$.  As
usual  $\rho'$ is the H\"{o}lder conjugate of $\rho$, i.e.
$1/\rho+1/\rho'=1$.

 Finally
$W^{\tau,\rho}(\Omega)$, $\tau\ge 0$ will denotes, when $\tau\not\in
\N_0$, the fractional Sobolev (Sobolev--Slobodeckij)
space in $\Omega$ and
$H^{\tau}(\Omega)=W^{\tau,2}(\Omega)$. See \cite{adams,grisvard} or
\cite{triebel}.

\subsection{Sobolev spaces and Riemannian gradient on
$\Gamma$.}\label{subsec2.2} The Sobolev spaces on $\Gamma$ are
treated in many textbooks when $\Gamma$ is the boundary of a smooth
open bounded set $\Omega\subset\R^N$  or more generally  a smooth
compact manifold, the non--optimality of the smoothness assumption being
often asserted. See for example
\cite{hebey,lionsmagenesIII,lionsmagenes1,triebel,triebefunctionspacesII}.
An exception is given by  \cite{grisvard}, so referring to it, when
$\Gamma=\partial\Omega$ and $\Omega$ is $C^k$, $k\in\N$, $\Gamma'$ is a relatively open subset of $\Gamma$,
$\rho\in(1,\infty)$ and $\theta\in [-k,k]$, we shall
denote by
$W^{\theta,\rho}(\Gamma')$ ($H^\theta=W^{\theta,2}$) the space
defined through local charts, making the reader aware that
distributions in $\Gamma'$ are elements of $[C^k_c(\Gamma')]'$. One
sees by elementary considerations that this approach is equivalent
to the one used in \cite{lionsmagenesIII} in the smooth case, with
local charts and partitions of the unity, and moreover both extend
to $C^k$ compact manifolds.

We also recall (see \cite[Theorem~1.5.1.2 and 1.5.1.3 p.
37]{grisvard}) the trace operator $u\mapsto u_{|\Gamma}$ which is
 linear and bounded from $W^{1,\rho}(\Omega)$ to
$W^{1-\frac 1\rho,\rho}(\Gamma)$, and has a right inverse
$\mathbb{D}\in \cal{L}(W^{1-\frac 1\rho,\rho}(\Gamma),
W^{1,\rho}(\Omega))$, i.e. $(\mathbb{D}u)_{|\Gamma}=u$ for all $u$,
independently on $\rho$.

We recall here, for the reader's convenience, some well--known preliminaries
on the Riemannian gradient, where only the fact that $\Gamma$ is a $C^1$  compact
manifold  endowed  with a $C^0$ Riemannian metric is used. In  Appendix \ref{appendixB}
these preliminaries will be used for abstract compact manifolds $M$.
We refer to  \cite{taylor} for more details and proofs,
given there for smooth manifolds, and to \cite{sternberg} for a
general background on differential geometry on $C^k$  manifolds when
$k\in\N$.

We start by fixing some notation. We  denote  by
$(\cdot,\cdot)_\Gamma$ the metric inherited from $\R^N$, with
$|\cdot|_\Gamma^2=(\cdot,\cdot)_\Gamma$,  given in local coordinates
$(y_1,\ldots,y_{N-1})$ by $(g_{ij})_{i,j=1,\ldots,N-1}$. We denote
by $d\sigma$ the natural volume element on $\Gamma$, given by $\sqrt{\widetilde{g}}
\,\,dy_1\wedge\ldots\wedge dy_{N-1}$, where $\widetilde{g}=\operatorname{det}
(g_{ij})$.  We denote by $(\cdot|\cdot)_\Gamma$ the Riemannian
(real) inner product on $1$-forms on $\Gamma$ associated to the
metric, given in local coordinates by $(g^{ij})=(g_{ij})^{-1}$.
Trivially, as $\Gamma$ is compact, there are $c_i=c_i(\Gamma)>0$,
$i=1,2,3$, such that \begin{footnote}{here and in the sequel the
summation convention is used}\end{footnote}
\begin{equation}\label{metricalimitata}
c_1\le \widetilde{g}\le c_2, \qquad\text{and}\quad g^{ij}\xi_i\xi_j\ge c_3
|\xi|^2\qquad\text{on $\Gamma$, for all
$\xi\in\R^{N-1}$.}\end{equation} We also denote by $d_\Gamma$ the
total differential on $\Gamma$ and by $\nabla_\Gamma$ the Riemannian
gradient, given in local coordinates by $\nabla_\Gamma
u=g^{ij}\,\partial_j u\,\,\partial_i$ for any $u\in H^1(\Gamma)$. It
is then clear that $(d_\Gamma u|d_\Gamma v)_\Gamma=(\nabla_\Gamma u,
\nabla_\Gamma v)_\Gamma$ for $u,v\in H^1(\Gamma)$, so the use of
vectors or forms in the sequel is optional. It is well--known (see
\cite{taylor} in the smooth setting, and the recent
paper \cite{quarteroni} in the $C^1$ setting) that $H^1(\Gamma)$ can be
equipped with the inner product and norm, equivalent to the standard
one, given by
\begin{equation}\label{complessifica}
(u,v)_{H^1(\Gamma)}=\int_\Gamma uv\,d\sigma +\int_\Gamma
(\nabla_\Gamma u, \nabla_\Gamma v)_\Gamma \,d\sigma,\quad
\|u\|_{H^1(\Gamma)}^2=\|u\|_{2,\Gamma}^2+\|\nabla_\Gamma
u\|_{2,\Gamma}^2
\end{equation}
 for $u,v\in H^1(\Gamma)$,  where $\|\nabla_\Gamma
u\|_{2,\Gamma}^2:=\int_\Gamma |\nabla_\Gamma u|_\Gamma^2$.

In the sequel, the notation $d\sigma$ will be dropped from the
boundary integrals; we hope that the reader will be able to put in
the appropriate integration elements in all formulas.

\subsection{Functional setting.}
We start by giving some details on $L^{2,\rho}_\alpha(\Omega)$ and
$L^{2,\rho}_\beta(\Gamma_1)$, which  definition can be extended to $\rho=\infty$ by
setting, for any $\rho\in [1,\infty]$,
\begin{alignat*}2
&L^{2,\rho}_\alpha(\Omega)=\{u\in L^2(\Omega): [u]_\alpha\in
L^\rho(\Omega,\lambda_\alpha)\}, \,
&&\|\cdot\|_{2,\rho,\alpha}=\|\cdot\|_2+\|[
\cdot]_\alpha\|_{\rho,\alpha},\\\label{weight2bis}
&L^{2,\rho}_\beta(\Gamma_1)=\{u\in L^2(\Gamma_1): [u]_\beta\in
L^\rho(\Gamma_1,\lambda_\beta)\}, \,
&&\|\cdot\|_{2,\rho,\beta}=\|\cdot\|_{2,\Gamma_1}+\|[
\cdot]_\beta\|_{\rho,\beta,\Gamma_1}.
\end{alignat*}
 Since the case $1\le \rho<2$ reduces to $\rho=2$, we shall consider $2\le \rho\le
\infty$ in the sequel.
 As $u\mapsto(u,[u]_\alpha)$ isometrically embeds $L^{2,\rho}(\Omega)$
 into $L^2(\Omega)\times L^\rho(\Omega, \lambda_\alpha)$ (the same argument applying to $L^{2,\rho}(\Gamma_1)$),
 they are reflexive spaces provided $\rho<\infty$.

  Moreover we have the
 trivial embeddings $L^\rho(\Omega)\hookrightarrow L^{2,\rho}_\alpha
(\Omega)$ and $L^\rho(\Gamma_1)\hookrightarrow L^{2,\rho}_\alpha
(\Gamma_1)$, which are dense by \cite[Theorem~1.17, p.
15]{rudinrealcomplex} and Lebesgue Dominated Convergence Theorem in
abstract measure spaces. For the same reason $[\cdot]_\alpha\in
\cal{L}(L^{2,\rho}_\alpha (\Omega), L^\rho(\Omega,\lambda_\alpha))$
and $[\cdot]_\beta\in \cal{L}(L^{2,\rho}_\beta (\Gamma_1),
L^\rho(\Gamma_1,\lambda_\beta))$ have dense ranges. Hence by
\cite[Corollary~2.18, p.45]{brezis2} their Banach adjoints are
injective. Consequently, making the standard identifications
\begin{equation}\label{dualiLebesgue}
[L^\rho(\Omega)]'\simeq L^{\rho'}(\Omega),\qquad\text{and}\quad
[L^\rho(\Gamma_1)]'\simeq L^{\rho'}(\Gamma_1),
\end{equation}
when $\rho\in [2,\infty)$ we have the two chains of embeddings
\begin{footnote}{$[L^\rho(\Omega,\lambda_\alpha)]'$ and
$[L^\rho(\Gamma_1,\lambda_\beta)]'$ cannot be
identified with $L^{\rho'}(\Omega,\lambda_\alpha)$ and
$L^{\rho'}(\Gamma_1,\lambda_\beta)$, since these identifications would
be incoherent with \eqref{dualiLebesgue}
}\end{footnote}
\begin{equation}\label{3.2}
[L^\rho(\Omega,\lambda_\alpha)]' \hookrightarrow
[L^{2,\rho}_\alpha (\Omega)]' \hookrightarrow
L^{\rho'}(\Omega),\, [L^\rho(\Gamma_1,\lambda_\beta)]' \hookrightarrow
[L^{2,\rho}_\beta (\Gamma_1)]' \hookrightarrow
L^{\rho'}(\Gamma_1).
\end{equation}
 Next, given $\rho,\theta\in [2,\infty)$ and
$-\infty\le a<b\le\infty$ we introduce the Banach space
\begin{equation}\label{3.3}
L^{2,\rho,\theta}_{\alpha,\beta}(a,b)=L^\rho(a,b\, ;
L^{2,\rho}_\alpha (\Omega))\times L^\theta(a,b\, ;
L^{2,\theta}_\beta (\Gamma_1))
\end{equation}
with the standard norm of the product. Clearly it is reflexive and
\begin{equation}\label{3.4}
[L^{2,\rho,\theta}_{\alpha,\beta}(a,b)]'\simeq L^{\rho'}(a,b\, ;
[L^{2,\rho}_\alpha (\Omega))]'\times L^{\theta'}(a,b\, ;
[L^{2,\theta}_\beta (\Gamma_1)]').
\end{equation}
Consequently, by \eqref{3.2}--\eqref{3.4} we have the embedding
\begin{equation}\label{3.4bis}
L^{\rho'}(a,b\, ; [L^\rho(\Omega,\lambda_\alpha)]')\times
L^{\theta'}(a,b\, ;
[L^\theta(\Gamma_1,\lambda_\beta)]')\hookrightarrow
[L^{2,\rho,\theta}_{\alpha,\beta}(a,b)]'.
\end{equation}
We now give some details on $H^0$ and $H^1$ introduced in
Section~\ref{intro}. The standard scalar product of $H^0$ will be
denoted by $(\cdot,\cdot)_{H^0}$, The space $H^1$ introduced in
\eqref{H1} will be endowed with a scalar product which induces a
norm equivalent to the one inherited by the product. We
recall \cite[Lemma 1]{vazvitHLB}  that the space
$$H^1(\Omega;\Gamma)=\{(u,v)\in H^1(\Omega)\times H^1(\Gamma):
v=u_{|\Gamma}\}$$ can be equipped with the scalar product
\begin{footnote}{the proof does not depends on the $C^\infty$ regularity of $\Omega$ asserted there}\end{footnote}
\begin{equation}\label{3.5}
(u,v)_{H^1(\Omega;\Gamma)}=\int_\Omega \nabla u\nabla v+\int_\Gamma
(\nabla_\Gamma u,\nabla_\Gamma v)_\Gamma+\int_\Gamma uv,\qquad
u,v\in H^1(\Omega;\Gamma).
\end{equation}
Since $H^1$ is actually a closed subspace of $H^1(\Omega;\Gamma)$,
$\nabla_\Gamma u=0$ a.e. on  the relative interior of $\Gamma_0$, that is on $\Gamma\setminus\overline{\Gamma_1}$,  and
$\sigma(\overline{\Gamma}_0\cap\overline{\Gamma}_1)=0$, we can equip
$H^1$ with the scalar product
\begin{equation}\label{3.6}
(u,v)_{H^1}=\int_\Omega \nabla u\nabla v+\int_{\Gamma_1}
(\nabla_\Gamma u,\nabla_\Gamma v)_\Gamma+\int_{\Gamma_1} uv,\qquad
u,v\in H^1.
\end{equation}
The related norm will be denoted by
$\|\cdot\|_{H^1}$. Finally the definition of $H^{1,\rho,\theta}_{\alpha,\beta}$  given in \eqref{H1plus}
can  be  extended also when $\rho=\infty$ and $\theta=\infty$, and it is a reflexive space
provided $\rho,\theta\in [2,\infty)$.

The relations among the spaces introduced above when $\rho,\theta\in
[2,\infty]$, are given by the following two chains of trivial
embeddings
\begin{equation}\label{3.10}
H^{1,\rho,\theta}_{\alpha,\beta}\hookrightarrow H^1 \hookrightarrow
H^0,\qquad\text{and}\quad
H^{1,\rho,\theta}_{\alpha,\beta}\hookrightarrow
L^{2,\rho}_\alpha(\Omega)\times
L^{2,\theta}_\beta(\Gamma_1)\hookrightarrow H^0.
\end{equation}
At a first glance they are all trivially dense when $\rho,\theta\in
[2,\infty)$ since $C^\infty(\overline{\Omega})$ is dense in
$H^1(\Omega)$  and in $L^2(\Omega)$, while $C^1(\Gamma)$ is dense in
$H^1(\Gamma)$ and in $L^2(\Gamma)$. A more careful check of this
argument would convince the reader that, even in the simpler case
$\Gamma_0=\emptyset$, the density of $C^\infty(\overline{\Omega})$
in $H^1(\Omega)$ and of $C^1(\Gamma)$ in $H^1(\Gamma)$ {\em does not
trivially implies} the density of $\{(u,v)\in
C^1(\overline{\Omega})\times C^1(\Gamma): v=u_{|\Gamma}\}$ in
$H^1(\Omega;\Gamma)$. See also \cite[Remark~2.6]{quarteroni} and
\cite[Lecture~11]{lionstata}. For this reason we give the following result.
\begin{lem}\label{lemma3.1}
Let $\alpha\in L^\infty(\Omega)$, $\beta\in L^\infty(\Gamma_1)$,
$\alpha,\beta\ge 0$, and $\rho,\theta\in [2,\infty)$. Then
$H^{1,\infty,\infty}_{1,1}$ is dense in both
$L^{2,\rho}_\alpha(\Omega)\times L^{2,\theta}_\beta(\Gamma_1)$ and
$H^{1,\rho,\theta}_{\alpha,\beta}$. Then  all the embeddings in
\eqref{3.10} are dense.
\end{lem}
\begin{proof} We first claim that $H^{1,\infty,\infty}_{1,1}$ is dense in
$L^{2,\rho}_\alpha(\Omega)\times L^{2,\theta}_\beta(\Gamma_1)$. By \eqref{3.2} our claim
follows once we prove that, given $(\varphi,\psi)\in
L^{\rho'}(\Omega)\times L^{\theta'}(\Gamma_1)$ such that
\begin{equation}\label{3.11}
\int_\Omega \varphi u+\int_{\Gamma_1}\psi u=0\qquad\text{for all
$u\in H^{1,\infty,\infty}_{1,1}$,}
\end{equation}
then $(\varphi,\psi)=0$. Taking $u\in C^\infty_c(\Omega)$ in
\eqref{3.11} we immediately get that $\varphi=0$, hence \eqref{3.11}
reduces to $\int_{\Gamma_1}\psi u=0$ for all $u\in
H^{1,\infty,\infty}_{1,1}$. Next, given any  $v\in
C^1_c(\Gamma\setminus\overline{\Gamma_1})$,  trivially extended to $v\in
C^1(\Gamma)$, we have $v\in H^1(\Gamma)\cap W^{1-\frac 1{2N},
2N}(\Gamma)$, hence $\mathbb{D}v\in W^{1,2N}(\Omega)$, so
$\mathbb{D}v\in H^1(\Omega)\cap C(\overline{\Omega})$ by Morrey's
theorem, hence $\mathbb{D}v\in H^{1,\infty,\infty}_{1,1}$. It
follows that $\int_{\Gamma_1}\psi v=0$ for all  $v\in
C^1_c(\Gamma\setminus\overline{\Gamma_1})$,  so  $\psi=0$ a.e. in
 $\Gamma\setminus\overline{\Gamma_1}$,  from which, as
$\sigma(\overline{\Gamma}_0\cap\overline{\Gamma}_1)=0$, we get $\psi=0$,  concluding
the proof of our claim.

To prove that $H^{1,\infty,\infty}_{1,1}$ is dense in
$H^{1,\rho,\theta}_{\alpha,\beta}$  we use a classical truncation
argument. Given $u\in
H^1$ and $k\in\N$  we respectively denote by $u^k$ and
$(u_{|\Gamma})^k$ the truncated of $u$ and $u_{|\Gamma}$ given by
\begin{equation}\label{3.13}
u^k=
\begin{cases}
u &\text{if $|u|\le k$,}\\
ku/|u| &\text{if $|u|>k$,}
\end{cases}\qquad
(u_{|\Gamma})^k=
\begin{cases}
u_{|\Gamma} &\text{if $|u_{|\Gamma}|\le k$,}\\
ku_{|\Gamma}/|u_{|\Gamma}| &\text{if $|u_{|\Gamma}|>k$,}
\end{cases}
\end{equation}
or  $u^k=k-[2k-(u+k)^+]^+$ and
$(u_{|\Gamma})^k=k-[2k-(u_{|\Gamma}+k)^+]^+$. Trivially $u^k\in
L^\infty(\Omega)$ and $(u_{|\Gamma})^k\in L^\infty(\Gamma)$.
Moreover using \cite[Lemma~7.6]{gilbarg} first in
$\Omega$ and then in coordinate neighborhoods on $\Gamma$, we get
that  $u^k\in H^1(\Omega)$, $(u_{|\Gamma})^k\in H^1(\Gamma)$,
\begin{equation}\label{3.14}
\nabla u^k=\begin{cases}
\nabla u &\text{if $|u|\le k$,}\\
0 &\text{if $|u|>k$,}
\end{cases}
\qquad\text{and}\quad \nabla_\Gamma (u_{|\Gamma})^k=\begin{cases}
\nabla_\Gamma u_{|\Gamma} &\text{if $|u_{|\Gamma}|\le k$,}\\
0 &\text{if $|u_{|\Gamma}|>k$.}
\end{cases}
\end{equation}
Now we note that
\begin{equation}\label{3.15}
(u_{|\Gamma})^k={u^k}_{|\Gamma},
\end{equation}
which is trivial when $u\in H^1\cap
C(\overline{\Omega})$, while in the general case $u\in H^1$ it
follows by approximating $u$ by a sequence $(u_n)_n$   in
$C^\infty(\overline{\Omega})$ such that $u_n\to u$ in $H^1(\Omega)$
(see \cite[Corollary~9.8, p.~277]{brezis2}).
 By \eqref{3.15} then we
have $u^k\in H^{1,\infty,\infty}_{1,1}$ for all $k\in\N$. We now
note that, by \eqref{3.13}--\eqref{3.14} and several applications of
the Lebesgue Dominated Convergence Theorem, we get $u^k\to u$ in
$H^{1,\rho,\theta}_{\alpha,\beta}$ as $k\to\infty$. Hence
$H^{1,\infty,\infty}_{1,1}$ is dense in
$H^{1,\rho,\theta}_{\alpha,\beta}$.

Finally we note that the density of the embeddings in \eqref{3.10}
follows from  previous statements since
$H^{1,\infty,\infty}_{1,1}\subseteq
H^{1,\rho,\theta}_{\alpha,\beta}$, $H^1=H^{1,2,2}_{1,1}$ and
$H^0=L^{2,2}_1(\Omega)\times L^{2,2}_1(\Gamma_1)$.
\end{proof}
Using \eqref{3.10} and Lemma~\ref{lemma3.1} and making the
identification $(H^0)'\simeq H^0$, which is coherent with
\eqref{dualiLebesgue}, we have  the two chains of dense embeddings
\begin{equation}
\label{3.17}
H^{1,\rho,\theta}_{\alpha,\beta}\hookrightarrow H^1
\hookrightarrow H^0\simeq (H^0)'\hookrightarrow
(H^1)'\hookrightarrow
(H^{1,\rho,\theta}_{\alpha,\beta})',
\end{equation}
$$
H^{1,\rho,\theta}_{\alpha,\beta}\hookrightarrow
L^{2,\rho}_\alpha(\Omega)\times
L^{2,\theta}_\beta(\Gamma_1)\hookrightarrow H^0 \simeq
(H^0)'\hookrightarrow [L^{2,\rho}_\alpha(\Omega)]'\times
[L^{2,\theta}_\beta(\Gamma_1)]'\hookrightarrow
(H^{1,\rho,\theta}_{\alpha,\beta})'
$$
and, by \eqref{3.2},
\begin{equation}\label{3.18}
[L^\rho(\Omega,\lambda_\alpha)]'\times
[L^\rho(\Gamma_1,\lambda_\beta)]' \hookrightarrow [L^{2,\rho}_\alpha
(\Omega)]'\times [L^{2,\rho}_\beta (\Gamma_1)]'\hookrightarrow
(H^{1,\rho,\theta}_{\alpha,\beta})'.
\end{equation}

\subsection{Weak solutions for the linear version of problem (\ref{1}).}
We now consider the linear
evolution boundary value problem
\begin{equation}\label{L}
\begin{cases} u_{tt}-\Delta u=\xi \qquad &\text{in
$(0,T)\times\Omega$,}\\
u=0 &\text{on $(0,T)\times \Gamma_0$,}\\
u_{tt}+\partial_\nu u-\Delta_\Gamma u=\eta\qquad &\text{on
$(0,T)\times \Gamma_1$,}
\end{cases}
\end{equation}
where $0<T<\infty$ and $\xi=\xi(t,x)$, $\eta=\eta(t,x)$ are given
forcing terms of the form
\begin{equation}\label{3.19}
\left\{
\begin{alignedat}3
&\xi=\xi_1+\alpha \xi_2,\qquad && \xi_1\in L^1(0,T;L^2(\Omega)), \quad && \xi_2\in L^{\rho'}(0,T;L^{\rho'}(\Omega, \lambda_\alpha)), \\
&\eta=\eta_1+\beta \eta_2,\qquad && \eta_1\in
L^1(0,T;L^2(\Gamma_1)), \quad && \eta_2\in
L^{\theta'}(0,T;L^{\theta'}(\Gamma_1, \lambda_\beta)),
\end{alignedat}\right.
\end{equation}
where $\alpha\in L^\infty(\Omega)$, $\beta\in L^\infty(\Gamma_1)$,
$\alpha, \beta\ge 0$ and $\rho,\theta\in [2,\infty)$.

To write \eqref{L} in a more abstract form  we set
$A\in\cal{L}(H^1,(H^1)')$ by
\begin{equation}\label{3.22}
\langle Au,v\rangle_{H^1}=\int_\Omega \nabla u\nabla
v+\int_{\Gamma_1} (\nabla_\Gamma u,\nabla_\Gamma v)_\Gamma,\qquad
\text{for all $u,v\in H^1$}.
\end{equation}
Moreover we denote $\Xi_1=(\xi_1,\eta_1)\in L^1(0,T;H^0)$ and we
define $\Xi_2\in [L^{2,\rho,\theta}_{\alpha,\beta}(0,T)]'$ by
setting $\langle
\Xi_2(t),\Phi\rangle_{L^{2,\rho}_\alpha(\Omega)\times
L^{2,\theta}_\beta(\Gamma_1)}=\int_\Omega \alpha
\xi_2(t)\phi+\int_{\Gamma_1} \beta \eta_2(t)\psi$ for almost all
$t\in (0,T)$ and all $\Phi=(\phi,\psi)\in
L^{2,\rho}_\alpha(\Omega)\times L^{2,\theta}_\beta(\Gamma_1)$.
By \eqref{3.17} we set $\Xi=\Xi_1+\Xi_2\in L^1(0,T;
[L^{2,\rho}_\alpha(\Omega)]'\times [L^{2,\theta}_\beta(\Gamma_1)]')$.

The following result characterizes solutions of $u$ in the sense of distributions.
\begin{lem}\label{lemma3.2}
Suppose that \eqref{3.19} holds and let
\begin{equation}\label{3.20}
u\in L^\infty(0,T;H^1)\cap W^{1,\infty}(0,T;H^0),\qquad u'\in
L^{2,\rho,\theta}_{\alpha,\beta}(0,T).
\end{equation}
Then the following facts  are equivalent:
\renewcommand{\labelenumi}{{(\roman{enumi})}}
\begin{enumerate}
\item the distribution identity
\begin{equation}\label{3.21}
\int_0^T\left[-(u',\phi')_{H^0}+\int_\Omega \nabla u\nabla
\phi+\int_{\Gamma_1} (\nabla_\Gamma
u,\nabla_\Gamma\phi)_\Gamma-\int_\Omega
\xi\phi-\int_{\Gamma_1}\eta\phi\right]=0
\end{equation}
holds for all $\phi\in C_c((0,T);H^1)\cap C^1_c((0,T);H^0)\cap
L^{2,\rho,\theta}_{\alpha,\beta}(0,T)$;
\item $u'\in W^{1,1}(0,T; [H^{1,\rho,\theta}_{\alpha,\beta}]')$ and
\begin{equation}\label{3.26}
u''(t)+Au(t)=\Xi(t)\qquad\text{in
$[H^{1,\rho,\theta}_{\alpha,\beta}]'$ for almost all $t\in (0,T)$;}
\end{equation}
\item the alternative distribution identity
\begin{equation}\label{3.27}
(u',\phi)_{H^0}\Big|_0^T+ \int_0^T\left[-(u',\phi')_{H^0}+\langle
Au,\phi\rangle_{H^1}-\langle
\Xi,\phi\rangle_{L^{2,\rho}_\alpha(\Omega)\times
L^{2,\theta}_\beta(\Gamma_1)}\right]=0
\end{equation}
holds for all $\phi\in C([0,T];H^1)\cap C^1([0,T];H^0)\cap
L^{2,\rho,\theta}_{\alpha,\beta}(0,T)$.
\end{enumerate}
Moreover any $u$ satisfying \eqref{3.20} and (ii) enjoys the further regularity
\begin{equation}\label{3.24}
u\in C([0,T];H^1)\cap C^1([0,T];H^0)
\end{equation}
and satisfies the energy identity
\begin{equation}\label{3.25}
\frac 12\|u'\|_{H^0}^2+\frac 12 \langle Au, u\rangle
_{H^1}\Big|_s^t=\int_s^t\langle
\Xi,u'\rangle_{L^{2,\rho}_\alpha(\Omega)\times
L^{2,\theta}_\beta(\Gamma_1)}\,d\tau
\end{equation}
for $0\le s\le t\le T$.
\end{lem}

\begin{proof}
 Let us denote  $\widetilde{X}=H^{1,\rho,\theta}_{\alpha,\beta}$. We first  claim that \eqref{3.24}--\eqref{3.25} hold for any  u satisfying (ii).  To prove our
claim we   apply \cite[Theorems~4.1~and~4.2]{strauss}. Referring to the
notation of the quoted paper  (but adding a $\widetilde{}$ on the notation of the spaces)
we set

$$\widetilde{V}=H^1,\quad \widetilde{H}=H^0,\quad \widetilde{W}=L^{2,\rho}_\alpha(\Omega)\times
L^{2,\theta}_\beta(\Gamma_1),\quad
\widetilde{Z}=L^{2,\rho,\theta}_{\alpha,\beta}(0,T),$$
while $A(t)=A$ is defined
by \eqref{3.22}. To check the structural assumptions of
\cite[p.545]{strauss} we note that  $\widetilde{V}$  and  $\widetilde{W}$  are both contained in
$\widetilde{H}$  and  $\widetilde{X}=\widetilde{V}\cap \widetilde{W}$   is dense in both $\widetilde{V}$ and  $\widetilde{W}$  by
Lemma~\ref{lemma3.1}. Moreover, by \eqref{3.3}--\eqref{3.4} we have
 $\widetilde{Z}\subset L^1(0,T; \widetilde{W})$ and  $\widetilde{Z}'\subset L^1(0,T; \widetilde{W}')$  as dense
subsets with continuous inclusions. Trivially for any  $w\in \widetilde{Z}$  and
 $v\in \widetilde{Z}'$  we have   $\langle v,w,\rangle_{\widetilde{Z}}=\int_0^T \langle
v(t),w(t)\rangle_{\widetilde{W}}\,dt$  so \cite[(3.1)]{strauss} holds. Next
multiplications by step functions trivially maps  $\widetilde{Z}$   into itself
and translations in $t$ are continuous in the strong operator
topology of  $\widetilde{Z}$   thanks to the extension of
\cite[Lemma~4.3,~p.114]{brezis2} for vector--valued functions. The
specific assumptions for \cite[Theorems~4.1~and~4.2]{strauss} are
satisfied since  $\widetilde{V}$   is dense in  $\widetilde{H}$   by Lemma~\ref{lemma3.1},  $\widetilde{W}$   is
contained in  $\widetilde{H}$   and \cite[(3.5)]{strauss} holds by \eqref{3.6} and
\eqref{3.22}.  Since \cite[(4.1)]{strauss} in this case
reduces to \eqref{3.26},  the proof of our first claim is completed.

Next we claim  that (i) and (ii) are equivalent each other and with
 the distribution identity
\begin{equation}\label{3.28}
\langle u',\phi\rangle_X\Big|_0^T+
\int_0^T\left[-(u',\phi')_{H^0}+\langle Au,\phi\rangle_{H^1}-\langle
\Xi,\phi\rangle_{L^{2,\rho}_\alpha(\Omega)\times
L^{2,\theta}_\beta(\Gamma_1)}\right]=0
\end{equation}
for all  $\phi\in C^1([0,T];\widetilde{X})$. Indeed if  $u$ satisfies (i) then  by  taking test functions $\phi$ in the
separate form $\phi(t)=\psi(t)w$, $\psi\in C^\infty_c(0,T)$,  $w\in
\widetilde{X}$, from \eqref{3.21} we immediately get that
$\int_0^T-uì\psi'+(Au-\Xi)\psi=0$ in  $\widetilde{X}'$ , from which (ii) follows.
Conversely, if (ii) holds then, by a standard density argument in
 $W^{1,1}(0,T;\widetilde{X}')$  we get that for any  $\phi\in C^1([0,T];\widetilde{X})$  we
have $\langle u',\phi\rangle_X\in W^{1,1}(0,T)$ and

\begin{equation}\label{3.29}
\langle u',\phi\rangle_{\widetilde{X}'}=\langle u'',\phi\rangle_{\widetilde{X}}+\langle
u',\phi'\rangle_{\widetilde{X}}\qquad\text{a.e. in $(0,T)$}.
\end{equation}

Then, evaluating \eqref{3.26} with  $\phi\in C^1([0,T];\widetilde{X})$,
integrating in $[0,T]$ and using \eqref{3.29} we get \eqref{3.28}.
By \eqref{3.28} we immediately get that that \eqref{3.21} holds true
for any  $\phi\in C^1_c((0,T);\widetilde{X})$  and then, by standard time
regularization, for any $\phi\in C_c((0,T);H^1)\cap
C^1_c((0,T);H^0)\cap L^{2,\rho,\theta}_{\alpha,\beta}(0,T)$, so
 completing the proof of our claim.

Since (iii) trivially implies (i),  the proof is completed (thanks to
our  second  claim) provided  we prove that  if \eqref{3.28}
holds  for all  $\phi\in C^1([0,T];\widetilde{X})$  then \eqref{3.27} holds for
all $\phi\in C([0,T];H^1)\cap C^1([0,T];H^0)\cap
L^{2,\rho,\theta}_{\alpha,\beta}(0,T)$. Since by \eqref{3.24} the
identity \eqref{3.28} can be rewritten as \eqref{3.27}, we just have
to prove that we can take less regular test functions in it. By
standard time regularization one easily get that \eqref{3.27} holds
for any $\phi\in C(\R;H^1)\cap C^1(\R;H^0)\cap
L^{2,\rho,\theta}_{\alpha,\beta}(-\infty,\infty)$, so our claim follows
since any $\phi\in C([0,T];H^1)\cap
C^1([0,T];H^0)\cap L^{2,\rho,\theta}_{\alpha,\beta}(0,T)$ can be
extended to $\widetilde{\phi}\in C(\R;H^1)\cap C^1(\R;H^0)\cap
L^{2,\rho,\theta}_{\alpha,\beta}(-\infty,\infty)$
\begin{footnote}{One first
defines $\phi_1\in C([-T/2,3T/2];H^1)\cap C^1([-T/2,3T/2];H^0)\cap
L^{2,\rho,\theta}_{\alpha,\beta}(-T/2,3T/2)$ as
$$\phi_1(t)=\begin{cases}
\phi(t)\qquad &\text{if $t\in [0,T]$,}\\
3\phi(-t)-2\phi(-2t)\qquad &\text{if $t\in [-T/2,0]$,}\\
3\phi(2T-t)-2\phi(3T-2t)\qquad &\text{if $t\in [T,3T/2]$,}
\end{cases}.$$
Then one sets $\widetilde{\phi}\in C_c((-T/2,3T/2);H^1)\cap
C^1_c((-T/2,3T/2);H^0)\cap
L^{2,\rho,\theta}_{\alpha,\beta}(-T/2,3T/2)$ as
$\tilde{\phi}=\psi_0\phi_1$ where $\psi_0\in C^\infty_c(-T/2,3T/2)$
and  $\psi_0=1$ in $[0,T]$.}
\end{footnote}.
\end{proof}
\section{\bf Well--posedness in $\boldsymbol{H^1\times H^0}$ and in $\boldsymbol{H^{1,\rho,\theta}_{\alpha,\beta}\times H^0}$: statements}
\label{section3}  In this section we state our main well--posedness
result for  problem \eqref{1} and a slightly more general (and
abstract) version of it. With reference to
\eqref{1} we now introduce our main assumptions on the
nonlinearities in it, starting from $P$ and $Q$.
\renewcommand{\labelenumi}{{(PQ\arabic{enumi})}}
\begin{enumerate}
\item $P$ and $Q$ are Carath\'eodory functions, respectively in
$\Omega\times\R$ and $\Gamma_1\times\R$, and  there are $\alpha\in
L^\infty(\Omega)$, $\beta\in L^\infty(\Gamma_1)$, $\alpha,\beta\ge
0$, and constants $m,\mu>1$, $c_m,c_\mu> 0$  such that
\begin{alignat}2\label{Pgrowth}
|&P(x,v)|\le c_m\alpha(x) (1+|v|^{m-1})\,\,&&\text{for almost all
$x\in\Omega$, all $v\in\R$;}\\\label{Qgrowth} |&Q(x,v)|\le
c_\mu\beta(x) (1+|v|^{\mu-1})\,\,&&\text{for almost all
$x\in\Gamma_1$, all $v\in\R$.}
\end{alignat}
\item $P$ and $Q$ are monotone increasing in the second variable for almost all values of
the first one;
\item $P$ and $Q$ are coercive, that is there are constants $c'_m,c'_\mu>
0$ such that
\begin{alignat}2
&P(x,v)v\ge c'_m\alpha(x)|v|^m \qquad&&\text{for almost all $x\in\Omega$, all $v\in\R$;}\\
&Q(x,v)v\ge c'_\mu\beta(x)|v|^\mu\qquad&&\text{for almost all
$x\in\Gamma_1$, all $v\in\R$.}
\end{alignat}
\end{enumerate}
\begin{rem}\label{remark4.1} Trivially (PQ1--3) yield
$P(\cdot,0)\equiv 0$ and $Q(\cdot,0)\equiv0$. Moreover in the
separate variable case considered in problem \eqref{5}, that is
$P(x,v)=\alpha(x)P_0(v)$ and  $Q(x,v)=\beta(x)Q_0(v)$ with
$\alpha\in L^\infty(\Omega)$, $\beta\in L^\infty(\Gamma_1)$,
$\alpha,\beta\ge 0$, (PQ1--3) reduce to assumption (I).
\end{rem}
Referring to (PQ1) we fix the notation
\begin{equation}\label{4.1}
\overline{m}=\max\{2,m\},\,\, \overline{\mu}=\max\{2,\mu,\},\,\,
W=L^{2,\overline{m}}_\alpha(\Omega)\times
L^{2,\overline{\mu}}_\beta(\Gamma_1),\,\,
X=H^{1,\overline{m},\overline{\mu}}_{\alpha,\beta},
\end{equation}
so \eqref{3.17} and the subsequent embedding read as
\begin{equation}\label{3.17ripetuta}
\begin{alignedat}8
&X\hookrightarrow &&H^1 &&\hookrightarrow &&H^0\simeq
&&(H^0)'\hookrightarrow &&(H^1)'&&\hookrightarrow
&&X',\\
&X\hookrightarrow &&W &&\hookrightarrow &&H^0 \simeq
&&(H^0)'\hookrightarrow &&W'&&\hookrightarrow &&X'.
\end{alignedat}
\end{equation}
Moreover, for $-\infty\le a<b\le\infty$, we denote
$$Z(a,b)=L^{2,\overline{m},\overline{\mu}}_{\alpha,\beta}(a,b),\qquad\text{and}\quad
Z'(a,b)=[Z(a,b)]'.$$
By (PQ1) the Nemitskii operators $\widehat{P}$
and $\widehat{Q}$ (respectively) associated to $P$ and $Q$ are
continuous from $L^{\overline{m}}(\Omega)$ to
$L^{\overline{m}'}(\Omega))\simeq [L^{\overline{m}}(\Omega))]'$ and
from $L^{\overline{\mu}}(\Gamma_1)$ to
$L^{\overline{\mu}'}(\Gamma_1))\simeq
[L^{\overline{\mu}}(\Gamma_1))]'$. By (PQ1) they can be uniquely extended to
\begin{equation}\label{4.2}
\widehat{P}: L^{2,\overline{m}}_\alpha (\Omega)\to
[L^{\overline{m}}(\Omega,\lambda_\alpha)]'\quad\text{and}\quad
\widehat{Q}: L^{2,\overline{\mu}}_\beta (\Gamma_1)\to
[L^{\overline{\mu}}(\Gamma_1,\lambda_\beta)]'
\end{equation}
given, for $u\in L^{2,\overline{m}}_\alpha (\Omega)$, $v\in
L^{\overline{m}}(\Omega,\lambda_\alpha)$, $w\in
L^{2,\overline{\mu}}_\beta (\Gamma_1)$ and $z\in
L^{\overline{\mu}}(\Gamma_1,\lambda_\beta)$, by
\begin{equation}\label{4.3}
\langle\widehat{P}(u),v\rangle_{L^{\overline{m}}(\Omega,\lambda_\alpha
)}=\int_\Omega P(\cdot,u)v\quad\text{and}\quad
\langle\widehat{Q}(w),z\rangle_{L^{\overline{\mu}}(\Gamma_1,\lambda_\beta
)}=\int_{\Gamma_1}Q(\cdot,w)z.
\end{equation}
We denote
\begin{equation}\label{4.4}
B=(\widehat{P},\widehat{Q}):W\to
[L^{\overline{m}}(\Omega,\lambda_\alpha)]'\times
[L^{\overline{\mu}}(\Gamma_1,\lambda_\beta)]',
\end{equation}
and we point out some relevant properties of $B$ we shall use in the
sequel.
\begin{lem}\label{lemma4.1}
Let (PQ1---2) hold and $(a,b)\subset\R$ is bounded. Then
\renewcommand{\labelenumi}{{(\roman{enumi})}}
\begin{enumerate}
\item B is continuous and bounded from $W$ to $[L^{\overline{m}}(\Omega,\lambda_\alpha)]'\times
[L^{\overline{\mu}}(\Gamma_1,\lambda_\beta)]'$ and hence, by
\eqref{3.18}, to $W'$;
\item $B$ acts boundedly and continuously from $Z(a,b)$ to $L^{\overline{m}'}(a,b\, ; [L^{\overline{m}}(\Omega,\lambda_\alpha)]')\times
L^{\overline{\mu}'}(a,b\, ;
[L^{\overline{\mu}}(\Gamma_1,\lambda_\beta)]')$ and hence, by
\eqref{3.4bis}, to $Z'(a,b)$;
\item $B$ is monotone in $W$ and in $Z(a,b)$.
\end{enumerate}
\end{lem}
\begin{proof}
We shall prove the properties listed above only for $\widehat{P}$, since
the same arguments apply, {\em mutatis mutandis}, to $\widehat{Q}$.
As to (i) we note that the fact that $\widehat{P}$ is well--defined
and bounded follows from \eqref{Pgrowth} and H\"{o}lder inequality.
Moreover, since the classical result on the continuity of Nemitskii
operators (see \cite[Theorem~2.2,~p.16]{ambrosettiprodi}) trivially
extends to abstract measure spaces, the Nemitskii operator
associated to $P_\alpha=P/\alpha$ (which is defined $\lambda_\alpha$
-- a.e. in $\Omega$) is continuous from
$L^{\overline{m}}(\Omega,\lambda_\alpha)$ to
$L^{\overline{m}'}(\Omega,\lambda_\alpha)$. By the form of the Riesz
isomorphism between $L^{\overline{m}'}(\Omega,\lambda_\alpha)$ and
$[L^{\overline{m}}(\Omega,\lambda_\alpha)]'$, since
$[\cdot]_\alpha\in \cal{L}(L^{2,\overline{m}}_\alpha (\Omega),
L^{\overline{m}}(\Omega,\lambda_\alpha))$, we get (i).  To prove (ii) we note that the
boundedness of $B$, almost everywhere defined in $(a,b)$ by
\eqref{4.4}, is a trivial consequence of (PQ1) and H\"{o}lder
inequality once again. To prove the asserted continuity we note
that, by repeating  previous argument, the Nemitskii operator
$\widehat{P_\alpha}$ associated to $P_\alpha=P/\alpha$ is continuous
from $L^{\overline{m}}((a,b)\times\Omega,\lambda'_\alpha )$
($\lambda'_\alpha$ denoting the product of the 1 - dimensional
Lebesgue measure and $\lambda_\alpha$) to
$L^{\overline{m}'}((a,b)\times\Omega,\lambda'_\alpha)$. Since for
any $\rho\in [1,\infty)$ one can prove as in the standard case, by
the density of $C_c((a,b)\times\Omega)$ in
$L^{\overline{m}'}((a,b)\times\Omega,\lambda'_\alpha )$ (cfr.
\cite[Theorem~1.36~p.~27~and~Theorem~3.14~p.~68]{rudinrealcomplex}),
that
\begin{equation}\label{newiso}
L^{\rho}((a,b)\times\Omega,\lambda'_\alpha )\simeq
L^{\rho}(a,b;L^{\rho}(\Omega,\lambda_\alpha)),
\end{equation}
 we then get that
$\widehat{P}$ is continuous from
$L^{\overline{m}}(a,b;L^{\overline{m}}(\Omega,\lambda_\alpha))$ to
$L^{\overline{m}ì}(a,b;[L^{\overline{m}}(\Omega,\lambda_\alpha)]')$
and then by \eqref{3.2} we get (ii).
Finally (iii) trivially follows from (PQ2).
\end{proof}

We introduce the following assumption, which will be used only in the
last part of the proof of Theorem \ref{proposition4.1} below:
\begin{enumerate}
\setcounter{enumi}3
\item if $m>\romega$ there are constants $c_m'', M_m>0$  such that
\begin{equation}\label{P4}
P_v(x,v)\ge c_m''\, \alpha(x)|v|^{m-2}\quad\text{for almost all
$(x,v)\in \Omega\times(\R\setminus (-M_m,M_m))$,}
\end{equation}
and if $\mu>\rgamma$ there are constants $c_\mu'', M_\mu>0$  such
that
\begin{equation}\label{Q4}
Q_v(x,v)\ge c_\mu''\, \beta(x)|v|^{\mu-2}\quad\text{for almost all
$(x,v)\in \Gamma_1\times(\R\setminus (-M_\mu,M_\mu))$.}
\end{equation}
\end{enumerate}
\begin{rem}\label{remark4.2bis}
Since by (PQ1--2)  the partial derivatives $P_v$ and $Q_v$ exist
almost everywhere (see \cite{DiBenedettorealanalysis}) and are
nonnegative, \eqref{P4}--\eqref{Q4} always hold if one allows $c_m''$
and $c_\mu''$ to vanish, and the assumption (PQ4) reduces to ask
that if $m>\romega$ then there is $M_m>0$ such that one can take
$c_m''>0$ in \eqref{P4} and if $\mu>\rgamma$ then there is $M_\mu>0$
such that one can take $c_\mu''>0$ in \eqref{Q4}. Moreover, in the
separate variable case considered in problem \eqref{5}, that is
$P(x,v)=\alpha(x)P_0(v)$ and  $Q(x,v)=\beta(x)Q_0(v)$ with
$\alpha\in L^\infty(\Omega)$, $\beta\in L^\infty(\Gamma_1)$,
$\alpha,\beta\ge 0$, (PQ4) reduces to (III).
\end{rem}

\begin{rem}\label{remark4.2ter}
We remark here, for a future use, some trivial consequences of
assumptions (PQ1--4). Setting $c_m'=0$ when $m\le \romega$ and
$c_\mu'=0$ when $\mu\le \rgamma$, since $P_v,Q_v\ge 0$ a.e., from
(PQ4) we have
\begin{alignat}2\label{4.20bis}
P_v(x,v)\ge
&\alpha(x)\left[c_m''|v|^{m-2}-c_m'''\right]\quad&&\text{for almost all $(x,v)\in \Omega\times\R$,}\\
\label{4.21bis} Q_v(x,v)\ge
&\beta(x)\left[c_\mu''|v|^{\mu-2}-c_\mu'''\right]\quad&&\text{for
almost all $(x,v)\in \Gamma_1\times\R$,}
\end{alignat}
where $c_m'''=c_m''M_m^{m-2}$, $c_\mu'''=c_\mu''M_\mu^{\mu-2}$. By
(PQ2) then, integrating \eqref{4.20bis} we get, for almost all
$x\in\Omega$ and all $v<w$,
\begin{equation}\label{4.22bis}
P(x,w)-P(x,v)\ge  \alpha(x)\left[\frac
{c_m''}{m-1}\left(|w|^{m-2}w-|v|^{m-2}v\right)-c_m'''(w-v)\right].\end{equation}
Consequently, using when $m>\romega$ the elementary inequality
$$\left(|w|^{m-2}w-|v|^{m-2}v\right)(w-v)\ge \widetilde{c_m}|w-v|^m\qquad\text{for all $v,w\in\R$,}$$
where $\widetilde{c_m}$ is a positive constant, setting
$\widetilde{c_m}''=c_m''\widetilde{c_m}/(m-1)$, from \eqref{4.22bis}
we get
\begin{equation}\label{4.23bis}
\widetilde{c_m}''\alpha(x)|v-w|^{\overline{m}}\le c_m'''
\alpha(x)|v-w|^2+ (P(x,w)-P(x,v))(w-v)\end{equation} for almost all
$x\in\Omega$ and all $v,w\in\R$, with $\widetilde{c_m}''>0$ when
$m>\romega$. Using the same arguments we get the existence of
$\widetilde{c_\mu}''\ge0$ such that
\begin{equation}\label{4.24bis}
\widetilde{c_\mu}''\beta(x)|v-w|^{\overline{\mu}}\le c_\mu'''
\beta(x)|v-w|^2+ (Q(x,w)-Q(x,v))(w-v)\end{equation} for almost all
$x\in\Gamma_1$ and all $v,w\in\R$, with $\widetilde{c_\mu}''>0$ when
$\mu>\rgamma$.
\end{rem}

Our main assumptions on $f$ and $g$, are the following ones:
\renewcommand{\labelenumi}{{(FG\arabic{enumi})}}
\begin{enumerate}
\item $f$ and $g$ are  Carath\'eodory functions, respectively in
$\Omega\times\R$ and $\Gamma_1\times\R$, and  there are constants
$p,q\ge 2$ and  $c_p,c_q\ge 0$  such that
\begin{alignat}2\label{F1}
&|f(x,u)|\le c_p(1+|u|^{p-1}),\qquad&&\text{for almost all
$x\in\Omega$, all $u\in\R$, and }\\
\label{G1} &|g(x,u)|\le c_q(1+|u|^{q-1}) \qquad&&\text{for almost
all $x\in\Gamma_1$, all $u\in\R$;}
\end{alignat}
\item there are constants $c'_p,c'_q\ge 0$  such that
\begin{alignat}2\label{F2}
&|f(x,u)-f(x,v)|\le &&c'_p |u-v|(1+|u|^{p-2}+|v|^{p-2})\\
\intertext{for almost all $x\in\Omega$,  all $u,v\in\R$,
and}\label{G2} &|g(x,u)-g(x,v)|\le &&c'_q
|u-v|(1+|u|^{q-2}+|v|^{q-2})
\end{alignat}
for almost all $x\in\Gamma_1$, all $u,v\in\R$.
\end{enumerate}
\begin{rem}\label{remark4.3}
 Assumptions (FG1--2)  can be equivalently formulated as follows:
\renewcommand{\labelenumi}{{(FG\arabic{enumi})$'$}}
\begin{enumerate}
\item $f$ and $g$ are  Carath\'eodory functions, respectively in
$\Omega\times\R$ and $\Gamma_1\times\R$, $f(x,\cdot)\in C^{0,1}_{\text{loc}}(\R)$
for almost all $x\in\Omega$ and $g(x,\cdot)\in C^{0,1}_{\text{loc}}(\R)$ for
almost all $x\in\Gamma_1$;
\item  $f(\cdot,0)\in
L^\infty(\Omega)$ and $g(\cdot,0)\in L^\infty(\Gamma_1)$;
\item there are constants $p,q\ge 2$,
$\widetilde{c_p},\widetilde{c_q}\ge 0$ such that
\begin{alignat*}2
&|f_u(x,u)|\le \widetilde{c_p}(1+|u|^{p-2}),\qquad&&\text{for almost
all $(x,u)\in\Omega\times\R$, and }\\
&|g_u(x,u)|\le \widetilde{c_q}(1+|v|^{q-2})\qquad&&\text{for almost
all $(x,u)\in\Gamma_1\times\R$.}
\end{alignat*}
\end{enumerate}
Indeed by (FG1--2)  we immediately get (FG1--2)$'$. Moreover, by (FG1)$'$
$f_u$ and $g_u$ exist almost everywhere
\begin{footnote}{the fact that measurable functions in an open set, which are locally absolutely continuous with respect to a variable,
  possess  almost everywhere partial derivative with respect to that variable is classical, as stated for example in \cite[p.297]{mm}. However
 the sceptical reader can prove it by repeating \cite[Proof~of~ Proposition~2.1~p.~173]{DiBenedettorealanalysis}
 for Carath\'eodory functions, so getting the measurability of the four Dini derivatives. Hence the set where the
 derivative does not exist is measurable and finally  it has zero measure by
 Fubini's theorem.}\end{footnote} so (FG3)$'$ follows. Conversely one gets  (FG1--2)  by
simply integrating (FG3)$'$ with respect to the second variable in the
convenient interval.

In the case considered in problem \eqref{5}, i.e.
$f(x,u)=f_0(u)$ and  $g(x,u)=g_0(u)$, assumptions  (FG1--2)  then reduce to
(II). Other relevant examples of functions $f$ and $g$ satisfying
 (FG1--2)  are given by

\begin{equation}\label{modellisupplementari}
\begin{alignedat}3
f_2(x,u)&=\gamma_1(x)|u|^{\widetilde{p}-2}u+
\gamma_2(x)|u|^{p-2}u+\gamma_3(x),\quad &&2\le\widetilde{p}\le
p,\,\, && \gamma_i\in L^\infty(\Omega),
\\
g_2(x,u)&=\delta_1(x)|u|^{\widetilde{q}-2}u+
\delta_2(x)|u|^{q-2}u+\delta_3(x),\, \quad &&2\le\widetilde{q}\le
q,\,\, &&\delta_i\in L^\infty(\Gamma_1),
\end{alignedat}
\end{equation}
and by
\begin{equation}\label{f3g3}
f_3(x,u)=\gamma(x) f_0(u),\quad g_3(x,u)=\delta(x) g_0(u),\quad
\gamma\in L^\infty(\Omega),\quad \delta\in L^\infty(\Gamma_1),
\end{equation}
where $f_0$ and $g_0$ satisfy (II).
\end{rem}
In line with Sobolev embedding of $H^1(\Omega)$ the source $f$  can
be classified (see \cite{bociulasiecka1}) as subcritical (or
critical) when $2\le p\le 1+\romega/2$, supercritical  when
$1+\romega/2< p\le\romega $ and supersupercritical when
$p>\romega$. The source $g$ can be classified in the same way
referring to $\rgamma$. This paper is devoted to the case when both
sources are subcritical (or critical), that is \eqref{6} holds. In
this case is easy to see, using H\"{o}lder inequality and Sobolev
embedding, that the Nemitskii operators  $\widehat{f}:H^1(\Omega)\to
L^2(\Omega)$ and $\widehat{g}:H^1(\Gamma)\cap L^2(\Gamma_1)\to L^2(\Gamma_1)$ respectively
associated to $f$ and $g$  are locally
Lipschitz.

 To precise the meaning of weak solutions of problem
\eqref{1} we first note that, by  (FG1--2) , for any $u$ satisfying \eqref{3.20}, denoting $u'=(u_t,(u_{|\Gamma})_t)$,  we
have $\widehat{f}(u)\in L^1(0,T;L^2(\Omega))$ and $\widehat{g}(u_{|\Gamma})\in
L^1(0,T;L^2(\Gamma_1))$. Moreover, by Lemma~\ref{lemma4.1},
$\widehat{P}(u_t)\in L^{\overline{m}'}(0,T;[L^{\overline{m}}(\Omega,
\lambda_\alpha)]')$ and $\widehat{Q}(u_{|\Gamma})_t)\in
L^{\overline{\mu}'}(0,T;[L^{\overline{\mu}}(\Gamma_1,
\lambda_\alpha)]')$. Then, by \eqref{newiso}, $\widehat{P}(u_t)=\alpha
\,\,\xi_2$ and $\widehat{Q}(u_{|\Gamma})_t)=\beta \,\,\eta_2$, with $\xi_2\in
L^{\overline{m}'}(0,T;L^{\overline{m}'}(\Omega, \lambda_\alpha))$
and $\eta_2\in L^{\overline{\mu}'}(0,T;L^{\overline{\mu}'}(\Gamma_1,
\lambda_\beta))$. By previous remarks and Lemma~\ref{lemma3.2} the following definition makes sense.
\begin{definition}\label{definition4.1} Let (PQ1--3), (FG1--2)
hold and $u_0\in H^1$, $u_1\in H^0$. A weak solution of problem
\eqref{1} in $[0,T]$, $0<T<\infty$,  is  $u$ verifying \eqref{3.20}--
\eqref{3.21} with
\begin{equation}\label{xieta}
\xi=\widehat{f}(u)-\widehat{P}(u_t),\quad
\eta=\widehat{g}(u_{|\Gamma})-\widehat{Q}((u_{|\Gamma})_t), \quad
\rho=\overline{m}\quad\text{and}\quad \theta=\overline{\mu},
\end{equation} such that $u(0)=u_0$ and $u'(0)=u_1$.
A weak solution of \eqref{1} in $[0,T)$, $0<T\le\infty$, is
$u\in L^\infty_{\text{loc}}([0,T);H^1)$  which is a weak solution
of \eqref{1} in $[0,T']$ for any $T'\in (0,T)$. Such a solution is
called maximal if it has no proper extensions.
\end{definition}
 \begin{rem} The introduction of Definition~\ref{definition4.1} is justified by the fact that, when $\Gamma$ is $C^2$, any $u\in C^2([0,T]\times\overline{\Omega})$ is a weak solution of \eqref{1} if and only if it is a classical one.\end{rem}
\begin{rem}\label{remark4.4}It follows by Lemma~\ref{lemma3.2} that
any weak solution of \eqref{1} in $\text{dom }u=[0,T]$ or $\text{dom
}u=[0,T)$ enjoys the further regularity
\begin{equation}\label{4.9}
u\in C(\text{dom }u;H^1)\cap C^1(\text{dom }u;H^0),
\end{equation}
satisfies the energy identity
\begin{multline}\label{enidd}
\tfrac 12 \left[\int_\Omega u_t^2+\int_{\Gamma_1}
(u_{|\Gamma})_t^2+\int_\Omega |\nabla
u|^2+\int_{\Gamma_1}|\nabla_\Gamma
u|_\Gamma^2\right]_s^t+\int_s^t\int_\Omega
P(\cdot,u_t)u_t\\
+\int_s^t\left[\int_{\Gamma_1}Q(\cdot,(u_{|\Gamma})_t)(u_{|\Gamma})_t
-\int_\Omega f(\cdot,u)u_t-\int_{\Gamma_1}g(\cdot,
u)(u_{|\Gamma})_t\right]=0
\end{multline}
for all $s,t\in \text{dom }u$, and the distribution identity
\begin{multline*}
\left[\int_\Omega u_t\phi+\int_{\Gamma_1}(u_{|\Gamma})_t\phi\right]_0^{T'}
+\int_0^{T'}\left[-\int_\Omega
u_t\phi_t-\int_{\Gamma_1}(u_{|\Gamma})_t(\phi_{|\Gamma})_t+\int_\Omega
\nabla u\nabla \phi\right.\\
+\left.\int_{\Gamma_1} (\nabla_\Gamma
u,\nabla_\Gamma\phi)_\Gamma+\int_\Omega
P(\cdot,u_t)\phi+\int_{\Gamma_1}
Q(\cdot,(u_{|\Gamma})_t)\phi-\!\int_\Omega
f(\cdot,u)\phi-\!\!\int_{\Gamma_1}  g(\cdot,u)\phi\right]=0
\end{multline*}
for all $T'\in \text{dom }u$ and $\phi\in C([0,T'];H^1)\cap
C^1([0,T'];H^0)\cap Z(0,T')$.

Finally we remark that when $u_0\in
H^{1,\rho,\theta}_{\alpha,\beta}$ for some finite $\rho$, $\theta$
satisfying \eqref{special} then, as
$u'\in
L^1(0,T'; H^{1,\rho,\theta}_{\alpha,\beta})
$ for all $T'\in \text{dom }u$, one easily gets that $u\in
W^{1,1}(0,T'; H^{1,\rho,\theta}_{\alpha,\beta})$, so
\begin{equation}\label{4.11bis}
u\in C(\text{dom }u; H^{1,\rho,\theta}_{\alpha,\beta}).
\end{equation}
\end{rem}
We can now state our main local well--posedness result for problem
\eqref{1}.
\begin{thm}\label{theorem4.1}Let (PQ1--3), (FG1--2), \eqref{6}
hold, $u_0\in H^1$ and $u_1\in H^0$. Then problem \eqref{1} has a
unique maximal weak solution $u=u(u_0,u_1)$ in $[0,T_{\text{max}})$,
$T_{\text{max}}=T_{\text{max}}(u_0,u_1)>0$.  Moreover
$$\lim\limits_{t\to T^-_{\text{max}}}
\|u(t)\|_{H^1}+\|u'(t)\|_{H^0}=\infty$$ provided
$T_{\text{max}}<\infty$. Next, if $(u_{0n},u_{1n})\to (u_0,u_1)$ in $H^1\times
H^0$, denoting $u_n=u(u_{0n},u_{1n})$ and
$T_{\text{max}}^n=T_{\text{max}}(u_{0n},u_{1n})$, we have
\renewcommand{\labelenumi}{{(\roman{enumi})}}
\begin{enumerate}
\item $T_{\text{max}}\le \liminf_n T_{\text{max}}^n$, and
\item $u_n\to u$ in $C([0,T^*];H^1)\cap C^1([0,T^*];H^0)$ for
all $T^*\in(0,T_{\text{max}})$.
\end{enumerate}
Finally, if also (PQ4) holds and $(u_{0n},u_{1n})\to (u_0,u_1)$ in
$H^1\times H^0$ we also have $u_n'\to u'$ in $Z(0,T^*)$ for all
$T^*\in(0,T_{\text{max}})$ and consequently, if $(u_{0n},u_{1n})\to
(u_0,u_1)$ in $H^{1,\rho,\theta}_{\alpha,\beta}\times H^0$ for some
$\rho,\theta$ satisfying \eqref{special}, we also have $u_n\to u$ in
$C([0,T^*];H^{1,\rho,\theta}_{\alpha,\beta})$ for all
$T^*\in(0,T_{\text{max}})$.
\end{thm}
Theorem~\ref{theorem4.1} is a particular case of an analogous result
concerning a slightly more general and abstract version of problem
\eqref{1}, that is the abstract Cauchy problem
\begin{equation}\label{C}
\begin{cases} u''+Au+B(u')=F(u)\qquad\text{in $X'$,}\\
u(0)=u_0,\quad u'(0)=u_1,
\end{cases}
\end{equation}
where $A$ and $B$ are the operators respectively defined in
\eqref{3.22} and \eqref{4.4}, and $F:H^1\to H^0$ is a locally
Lipschitz map, that is for any $R>0$ there is $L(R)\ge
0$ such that
\begin{equation}\label{4.12}
\|F(u)-F(v)\|_{H^0}\le L(R)\,\|u-v\|_{H^1}\qquad\text{provided
$\|u\|_{H^1},\|v\|_{H^1}\le R$.}
\end{equation}
When  (FG1--2)  and \eqref{6} hold,  $F=(\widehat{f},\widehat{g})$ satisfies
\eqref{4.12}.

We first precise the meaning of strong,  generalized  and
weak  solutions of
\begin{equation}\label{C1}
u''+Au+B(u')=F(u)\qquad\text{in $X'$.}
\end{equation}
\begin{definition}\label{definition4.2}
Let (PQ1--3) and \eqref{4.12} hold, and $0<T<\infty$.
\renewcommand{\labelenumi}{{(\roman{enumi})}}
\begin{enumerate}
\item By a {\em strong} solution of \eqref{C1} in $[0,T]$ we mean $u\in
W^{1,\infty}(0,T;H^1)\cap W^{2,\infty}(0,T;H^0)$ such that
$Au(t)+B(u'(t))\in H^0$ and $u'(t)\in X$ for all $t\in [0,T]$ and
\eqref{C1} holds in $H^0$ almost everywhere in $(0,T)$.
\item By a {\em generalized} solution of \eqref{C1} in $[0,T]$ we mean the
limit of a sequence of strong solutions of \eqref{C1} in
$C([0,T];H^1)\cap C^1([0,T];H^0)$.
\item By a {\em weak} solution of \eqref{C1} in $[0,T]$ we mean
$u$ satisfying \eqref{3.20} with $\rho=\overline{m}$, $\theta=\overline{\mu}$ and
the distribution identity
\begin{equation}\label{4.13}
\int_0^T-(u',\phi')_{H^0}+\langle Au,\phi\rangle_{H^1}+\langle
B(u'),\phi\rangle_W=\int_0^T (F(u),\phi)_{H^0}
\end{equation}
for all $\phi\in C_c((0,T);H^1)\cap C^1_c((0,T);H^0)\cap
Z(0,T)$.
\end{enumerate}
By a solution  in $[0,T)$, $T\in
(0,\infty]$, we mean $u\in L^\infty_{\text{loc}}([0,T);H^1)$ which is
a solution  in $[0,T']$ for any $T'\in
(0,T)$. Such a solution is called maximal if has no proper
extensions in the same class.
\end{definition}
\begin{rem}\label{remark4.5}
For any weak solution of \eqref{C1} in $[0,T]$ we have
$F(u)\in L^\infty(0,T;H^0)$, hence as in
Remark~\ref{remark4.4} we see that weak solutions satisfy \eqref{4.9} as well as
the generalized versions of the energy and distribution identities in Remark \ref{remark4.4}. Moreover for
any couple $(u,v)$ of weak solutions the energy identity
\begin{equation}\label{4.14}
\tfrac 12 \|w'\|_{H^0}^2+\tfrac 12 \langle Aw,w\rangle_{H^1}\Big
|_s^t+\int_s^t\negquad \langle B(u')-B(v'),w'\rangle_W=\int_s^t
(F(u)-F(v),w')_{H^0}
\end{equation}
holds for $s,t\in \text{dom u}\cap\text{dom v}$, where $w$ denotes
the difference $u-v$. Finally also in this case \eqref{4.11bis} holds true for  $u_0\in
H^{1,\rho,\theta}_{\alpha,\beta}$, with  $(\rho,\theta)$
satisfying \eqref{special}.
\end{rem}
By previous remark  the following definition makes sense.
\begin{definition} By a strong, generalized or weak solution of
\eqref{C} we mean a solution of \eqref{C1} in  the corresponding
class verifying also the initial conditions.
\end{definition}
Our main result concerning \eqref{C} is the following one.
\begin{thm}\label{proposition4.1} Let (PQ1--3), \eqref{4.12} hold, $u_0\in H^1$ and $u_1\in
H^0$. Then problem \eqref{C} has a unique maximal weak solution
$u=u(u_0,u_1)$ in $[0,T_{\text{max}})$,
$T_{\text{max}}=T_{\text{max}}(u_0,u_1)>0$, which is also the unique
maximal generalized solution of it. If
\begin{equation}\label{4.15}
u_0\in H^1,\quad u_1\in X,\quad\text{and}\quad Au_0+B(u_1)\in H^0,
\end{equation}
then $u$ is actually the unique maximal strong solution of
\eqref{C}. Moreover
\begin{equation}\label{4.16}\lim\limits_{t\to T^-_{\text{max}}}
\|u(t)\|_{H^1}+\|u'(t)\|_{H^0}=\infty \end{equation}
 provided
$T_{\text{max}}<\infty$,
 and  $T_{\text{max}}=\infty$
when $F$ is globally Lipschitz.

\noindent Next, if $(u_{0n},u_{1n})\to
(u_0,u_1)$ in $H^1\times H^0$, denoting $u_n=u(u_{0n},u_{1n})$ and
$T_{\text{max}}^n=T_{\text{max}}(u_{0n},u_{1n})$, we have
\renewcommand{\labelenumi}{{(\roman{enumi})}}
\begin{enumerate}
\item $T_{\text{max}}\le \liminf_n T_{\text{max}}^n$, and
\item $u_n\to u$ in $C([0,T^*];H^1)\cap C^1([0,T^*];H^0)$ for
all $T^*\in(0,T_{\text{max}})$.
\end{enumerate}
Finally, if also (PQ4) holds and $(u_{0n},u_{1n})\to (u_0,u_1)$ in
$H^1\times H^0$ we also have $u_n'\to u'$ in $Z(0,T^*)$ for all
$T^*\in(0,T_{\text{max}})$. Consequently, if $(u_{0n},u_{1n})\to
(u_0,u_1)$ in $H^{1,\rho,\theta}_{\alpha,\beta}\times H^0$ for some
$\rho,\theta$ satisfying \eqref{special}, we also have $u_n\to u$ in
$C([0,T^*];H^{1,\rho,\theta}_{\alpha,\beta})$ for all
$T^*\in(0,T_{\text{max}})$.
\end{thm}
Theorem~\ref{proposition4.1} will be proved in the next section
by transforming \eqref{C} in a first order Cauchy problem, applying
nonlinear semigroup theory to it, and finally discussing the
relations between various type of solutions of \eqref{C}.

\section{\bf Well--posedness in $\boldsymbol{H^1\times H^0}$ and in $\boldsymbol{H^{1,\rho,\theta}_{\alpha,\beta}\times H^0}$: proofs}
\label{section4}

We introduce the phase space for problem \eqref{C}, that is the
Hilbert space
\begin{equation}\label{4.17}
\cal{H}=H^1\times H^0 ,\end{equation} endowed with the standard
scalar product $(\cdot,\cdot)_{\cal{H}}$ given by
\begin{equation}\label{4.18}
(U_1,U_2)_{\cal{H}}=(u_1,u_2)_{H^1}+(v_1,v_2)_{H^0}\qquad\text{for
all $U_i=(u_i,v_i)$, $i=1,2$.}
\end{equation}
Moreover, using  \eqref{3.17ripetuta}, we introduce the nonlinear operator $\cal{A}:
D(\cal{A})\subset \cal{H}\to\cal{H}$ by
\begin{gather}\label{4.19}
D(\cal{A})= \{(u,v)\in H^1\times X: Au+B(v)\in H^0\},\\
\label{4.20}\cal{A}\begin{pmatrix}u\\v\end{pmatrix} =
\begin{pmatrix}-v\\Au+B(v)\end{pmatrix},
\end{gather}
and the abstract Cauchy problem
\begin{equation}\label{C2}
\begin{cases} U'+\cal{A}U+\cal{F}(U)=0\qquad\text{in $\cal{H}$,}\\
U(0)=U_0\in\cal{H},
\end{cases}
\end{equation}
where $\cal{F}:\cal{H}\to\cal{H}$ is any locally Lipschitz map.

The meaning of strong and generalized solutions of \eqref{C2} in
$[0,T]$, $0<T<\infty$ is standard (see \cite[Theorem~4.1~and~
Definition,~pp.180--183]{showalter}),  while by solutions in $[0,T)$
we mean $U\in C([0,T);\cal{H})$ which are solutions in $[0,T']$ in
the corresponding sense for all $T'\in(0,T)$.
Our main result on problem \eqref{C2} is the
following one.
\begin{thm}\label{theorem4.2}
Let (PQ1--3) hold. Then the operator $\cal{A}+I$ is maximal monotone
in $\cal{H}$, $D(\cal{A})$  is dense in $\cal{H}$ and $A(0)=0$.
Consequently, given  any locally Lipschitz map
$\cal{F}:\cal{H}\to\cal{H}$, the following conclusions hold:
\renewcommand{\labelenumi}{{(\roman{enumi})}}
\begin{enumerate}
\item for any $U_0\in\cal{H}$ the problem \eqref{C2} has a unique maximal
generalized solution $U=U(U_0)$ in $[0,T_{\text{max}})$,
$T_{\text{max}}=T_{\text{max}}(U_0)$, which is the unique maximal
strong solution of it when $U_0\in D(\cal{A})$;
\item $\lim_{t\to T^-_{\text{max}}}
\|U(t)\|_{\cal{H}}=\infty$ provided $T_{\text{max}}<\infty$, and
$T_{\text{max}}=\infty$ provided $\cal{F}$ is globally Lipschitz;
\item if $U_{0n}\to U_0$ in $\cal{H}$ then
denoting $U_n=U(U_{0n})$ and
$T_{\text{max}}^n=T_{\text{max}}(U_{0n})$, we have
$T_{\text{max}}\le \liminf_n T_{\text{max}}^n$ and $U_n\to U$ in
$C([0,T^*];\cal{H})$ for all $T^*\in(0,T_{\text{max}})$.
\end{enumerate}
\end{thm}
\begin{proof}{\bf Step 1: $\boldsymbol{\cal{A}+I}$ is monotone
 in $\boldsymbol{\cal{H}}$.} Let $U_i=(u_i,v_i)\in
D(\cal{A})$ for $i=1,2$. By \eqref{4.18} and \eqref{4.20}
\begin{multline}\label{4.21}
(\cal{A}(U_1)-\cal{A}(U_2),U_1-U_2)_{\cal{H}}\\=(v_2-v_1,u_1-u_2)_{H^1}+(A(u_1-u_2)+B(v_1)-B(v_2),v_1-v_2)_{H^0}.
\end{multline}
Since $v_i\in X$ for $i=1,2$, by \eqref{3.17ripetuta} we have
\begin{multline}\label{4.22}
(A(u_1-u_2)+B(v_1)-B(v_2),v_1-v_2)_{H^0}\\=\langle
A(u_1-u_2),v_1-v_2\rangle_{H^1}+\langle
B(v_1)-B(v_2),v_1-v_2\rangle_{W}.
\end{multline}
By plugging \eqref{3.6}, \eqref{3.22} and \eqref{4.22} in
\eqref{4.21}  we  get
$$(\cal{A}(U_1)-\cal{A}(U_2),U_1-U_2)_{\cal{H}}=\int_{\Gamma_1}\!\!
(v_2-v_1)(u_1-u_2)+\langle B(v_1)-B(v_2),v_1-v_2\rangle_{W}.
$$
By Lemma~\ref{lemma4.1}--(iii), \eqref{3.6} and \eqref{4.18}  we then
get
\begin{multline*}
(\cal{A}(U_1)-\cal{A}(U_2),U_1-U_2)_{\cal{H}} \ge
\int_{\Gamma_1} (v_2-v_1)(u_1-u_2)\\  \ge -\frac 12
\|v_1-v_2\|_{2,\Gamma_1}^2-\frac 12
\|u_1-u_2\|_{2,\Gamma_1}^2
\ge - \|U_1-U_2\|_{\cal{H}}^2,
\end{multline*}
and then $\cal{A}+I$ is  monotone.

{\noindent \bf Step 2: $\boldsymbol{\cal{A}+I}$ is maximal monotone
in $\boldsymbol{\cal{H}}$.} By Step 1 and the nonlinear version of
Minty's theorem (see \cite[Lemma~1.3, p.~159]{showalter}) this fact
is equivalent to prove that $\text{Rg}(\cal{A}+2I)=\cal{H}$. Consequently, by
\eqref{4.19}--\eqref{4.20} we have to show that for all
$(h^0,h^1)\in H^0\times H^1$  the system
\begin{equation}\label{4.25}
\begin{cases}
2u-v=h^1 &\qquad\text{in $H^1$},\\
2v+Au+B(v)=h^0&\qquad\text{in $X'$},
\end{cases}
\end{equation}
has a solution $(u,v)\in H^1\times X$. Since $X\subset H^1$ we can
solve the first equation in $u$  and plug $u=\frac 12 (v+h^1)$ in
the second one. Hence to solve \eqref{4.25} reduces to prove that,
for $h^2=2h^0-Ah^1\in (H^1)'$, the single equation
\begin{equation}\label{4.26}
4v+Av+2B(v)=h^2\qquad\text{in $X'$}
\end{equation}
has a solution $v\in X$. Actually we claim that \eqref{4.26} has a
solution for any $h^2\in X'$  i.e. that the
operator $T:X\to X'$ given by $T=4I+A+2B$ is surjective.

We first consider, for the reader's convenience, the simplest linear case when $P(x,v)=\alpha(x)v$ and $Q(x,v)=\beta(x)v$.
In this case clearly $m=\overline{m}=\mu=\overline{\mu}=2$, so $X=H^1$ and for all $u,v\in H^1$ we have
$\langle T(u), v\rangle_X = a(u,v)$, where
$a$ is the continuous bilinear form in $H^1$ given by
$$a(u,v)=4(u,v)_{H^0}+\int_\Omega
\nabla u \nabla v +\int_{\Gamma_1}\nabla_\Gamma
u\nabla_\Gamma v+\int_\Omega \alpha uv+
\int_{\Gamma_1} \beta uv.$$
Since, by \eqref{3.6}, $a(u,u)\ge \|u\|^2$, $a$ is coercive, so $T$ is surjective by the Lax--Milgram Theorem.

In the general case $T$ is (possibly) nonlinear but,  by \eqref{3.22} and
Lemma~\ref{lemma4.1}--(iii), it is  monotone being the  sum of
monotone operators. Moreover by Lemma~\ref{lemma4.1}--(i) and
\eqref{3.17ripetuta} we have $T\in C(X,X')$. Next, by
\eqref{4.1},
\begin{alignat*}2
&\|[u]_\alpha\|_{\overline{m},\alpha}\le
&&\|[u]_\alpha\|_{2,\alpha}+\|[u]_\alpha\|_{m,\alpha}
\le\|\alpha\|_\infty \|u\|_2+\|[u]_\alpha\|_{m,\alpha},\\
&\|[v]_\beta\|_{\overline{\mu},\beta}\le
&&\|[v]_\beta\|_{2,\beta,\Gamma_1}+\|[v]_\beta\|_{\mu,\beta,\Gamma_1}
\le\|\beta\|_{\infty,\Gamma_1}\|v\|_{2,\Gamma_1}+\|[v]_\beta\|_{\mu,\beta,\Gamma_1}
\end{alignat*}
for all $u\in L^{2,\overline{m}}_\alpha(\Omega)$ and
$v\in L^{2,\overline{\mu}}_\beta(\Gamma_1)$. Consequently, by \eqref{H1plus} and \eqref{4.1}, there is
$c_1=c_1(\Omega, \|\alpha\|_\infty, \|\beta\|_{\infty,\Gamma_1})>0$
such that
\begin{equation}\label{4.27}
\|u\|_X\le c_1
(\|u\|_{H^1}+\|[u]_\alpha\|_{m,\alpha}+\|[u]_\beta\|_{\mu,\beta,\Gamma_1})\quad\text{for
all $u\in X$.}
\end{equation}
On the other hand, by \eqref{3.6} and (PQ3), for any $u\in X$ we
have
\begin{equation}\label{4.28}
\begin{aligned}
\langle T(u), u\rangle_X=&4 \|u\|_{H^0}^2+\int_\Omega
|\nabla u|^2+\int_{\Gamma_1}\!\!|\nabla_\Gamma
u|_\Gamma^2+2c_m'\int_\Omega
P(\cdot,u)u\\+&2c_\mu'\int_{\Gamma_1} Q(\cdot,u)u\,\,
\ge
\,\,c_2(\|u\|_{H^1}^2+\|[u]_\alpha\|_{m,\alpha}^m+\|[u]_\beta\|_{\mu,\beta,\Gamma_1}^\mu)
\end{aligned}
\end{equation}
where $c_2=\min\{4,2c_m' 2c_\mu'\}>0$. By the elementary inequality
$x^{s'}\le 1+x^s$ for all $0\le s'\le s$, $x\ge 0$, and discrete
H\"{o}lder inequality, from \eqref{4.28} we get
\begin{equation}\label{4.29}
\langle T(u), u\rangle_X\ge
3^{1-\mu_0}c_2\left(\|u\|_{H^1}+\|[u]_\alpha\|_{m,\alpha}+\|[u]_\beta\|_{\mu,\beta,\Gamma_1}\right)^{\mu_0}-3c_2
\end{equation}
where $\mu_0=\min\{2,m,\mu\}$. By combining \eqref{4.27} and
\eqref{4.29}, since $\mu_0>1$, we get that $T$ is coercive, i.e.
$\|u_n\|_X\to\infty$ implies $\langle
T(u_n),u_n\rangle_X/\|u_n\|_X\to\infty$. Then our claim follows
since monotone, hemicontinuous and coercive operators are surjective
(see \cite[Theorem~1.3~p.~40]{barbu} or
\cite[Corollary~2.3~p.~37]{barbu2010})).\begin{footnote}{  An alternative proof of this point
is given in Remark~\ref{piffero1357}, page~\pageref{piffero1357}.}\end{footnote}

{\noindent \bf Step 3: $\boldsymbol{\cal{A}0=0}$ and
$\boldsymbol{D(\cal{A})}$ is dense in  $\boldsymbol{\cal{H}}$.} The
first conclusion follows by \eqref{4.20} and Remark~\ref{remark4.1}.
To prove the second one we note that
$H^{1,2(\overline{m}-1),2(\overline{\mu}-1)}_{\alpha,\beta}\subseteq
X$ and by (PQ1) we have
$B(H^{1,2(\overline{m}-1),2(\overline{\mu}-1)}_{\alpha,\beta})\subseteq
H^0$. Consequently
\begin{equation}\label{inclusionj}
(A+I)^{-1}(H^0)\times
H^{1,2(\overline{m}-1),2(\overline{\mu}-1)}_{\alpha,\beta}\subseteq
D(\cal{A}).
\end{equation}
Now $H^{1,2(\overline{m}-1),2(\overline{\mu}-1)}_{\alpha,\beta}$ is
dense in $H^0$ by Lemma~\ref{lemma3.1}, while $(A+I)^{-1}(H^0)$ is
dense in $H^1$ since $H^0$ is dense in $(H^1)'$ by  \eqref{3.17ripetuta}
and $A+I:H^1\to(H^1)'$ is an isomorphism by \eqref{3.6}, \eqref{3.22} and Riesz--Fr\'{e}chet
theorem. Hence, by \eqref{inclusionj},  $D(\cal{A})$ is dense in
$\cal{H}$.

{\noindent \bf Step 4: conclusion.} Assertions (i--iii) follow at
once by applying Theorems~\ref{theoremA1}, \ref{theoremA2} and Remark \ref{theoremA1plus} in Appendix~\ref{appendixA}
to \eqref{C2},
which can trivially rewritten as
\begin{equation}\label{C3}
\begin{cases} U'+\cal{A_1}U+\cal{F_1}(U)=0\qquad\text{in $\cal{H}$,}\\
U(0)=U_0\in\cal{H}
\end{cases}
\end{equation}
where $\cal{A_1}=\cal{A}+I$ and $\cal{F_1}=\cal{F}+I$.
\end{proof}

\begin{rem}\label{piffero1357}
 The surjectivity of the operator $T$ introduced in {\bf Step 2} also follows by the Direct Method of the Calculus of Variations without invoking
the surjectivity theorem of V. Barbu quoted before, since $T$ has a variational nature. Indeed, setting
$$\cal{P}(x,v)=\int_0^v P(x,s)\,ds\qquad\text{and}\quad
\cal{Q}(y,v)=\int_0^v Q(y,s)\,ds$$
for a.a. $x\in\Omega$, $y\in \Gamma_1$ and all $v\in\R$, one easily sees that
$$\cal{B}(u)=2 \|u\|_{H^0}^2+\frac 12\int_\Omega
|\nabla u|^2+\frac 12 \int_{\Gamma_1}|\nabla_\Gamma
u|_\Gamma^2+\int_\Omega \cal{P}(\cdot, u)+\int_{\Gamma_1} \cal{Q}(\cdot,u)-\langle h^2,v\rangle_X
$$
defines a (possibly nonlinear) functional $\cal{B}\in C^1(X)$ and  that its Fr\'echet differential is nothing but $T-h^2$. Moreover, by \eqref{4.27} and \eqref{4.12}, since $\cal{B}(0)=0$,
one gets that for all $u\in X$
$$\cal{B}(u)=\int_0^1 \frac d{dt}J(tu)\,dt=\int_0^1 \langle T(tu)-h^2,u\rangle_X \,dt\ge c_3\|u\|^{\mu_0}_X-\|h^2\|_{X'}\|u\|_X-3c_2$$
where $c_3=3^{1-\mu_0}(\mu_0+1)^{-1}c_2c_1^{-\mu_0}$.
Hence, as $\mu_0>1$, $\cal{B}$ is coercive in $X$ (that is $\cal{B}(u_n)\to\infty$ when $\|u_n\|_X\to\infty$) so $\cal{B}$ has a minimum $v\in X$, which is then a critical point of it, so $T(v)=h^2$.
\end{rem}

To prove Theorem~\ref{proposition4.1} we have to prove that the
generalized solution found in Theorem~\ref{theorem4.2} is actually a
weak solution, which is unique. We start with the uniqueness.
\begin{lem}\label{lemma4.1bis}
The weak maximal solution of \eqref{C} is unique.
\end{lem}
\begin{proof}
Clearly the statement reduces to prove that, given two weak
solutions $u$ and $v$ in $[0,T]$, $0<T<\infty$, then $u=v$. We set
$M=\max\{\|u\|_{C([0,T];H^1)},\|v\|_{C([0,T];H^1)}\}$.
Using \eqref{4.14} and Lemma~\ref{lemma4.1} -- (iii) and \eqref{4.12}
we get the estimate
\begin{equation}\label{4.29bis}
\tfrac 12 \|w'\|_{H^0}^2+\tfrac 12 \langle Aw,w\rangle_{H^1}\Big
|_0^t\le L(M)\int_0^t\|w\|_{H^1}\|w'\|_{H^0}\qquad\text{for $t\in
[0,T]$}.
\end{equation}
By \eqref{3.6} and \eqref{3.22} we have $\langle
Au,u\rangle_{H^1}=\|u\|_{H^1}^2-\|u\|_{H^0}^2$ for all $u\in H^1$.
Moreover, as $w(0)=0$, by H\"{o}lder inequality we have
\begin{equation}\label{4.30}
\|w(t)\|_{H^0}^2=\left\|\int_0^t w'\right\|_{H^0}^2\le T\int_0^t
\|w'\|_{H^0}^2.
\end{equation}
Hence by \eqref{4.29bis} and Young inequality we get
\begin{equation}\label{4.31}
\tfrac 12 \|w'(t)\|_{H^0}^2+\tfrac 12 \|w(t)\|_{H^1}^2\le
\tfrac{T+L(M)}2\int_0^t \|w'\|_{H^0}^2+ \|w\|_{H^1}^2\qquad\text{for
$t\in [0,T]$.}
\end{equation}
The proof is completed by applying Gronwall inequality.
\end{proof}
\begin{lem}\label{lemma4.2}
Generalized solutions of \eqref{C} are also weak.
\end{lem}
\begin{proof}
Clearly we can prove the statement for
solutions  in $[0,T]$, $0<T<\infty$.

{\noindent\bf Step 1: strong solutions are also weak.} Let $u$ be a
strong solution. We first claim that
$u'\in Z(0,T)$. Since, by Definition~\ref{definition4.2}, $u'(t)\in
X$ for all $t\in [0,T]$, by \eqref{C1} we get
\begin{equation}\label{4.15bis}
(u'',u')_{H^0}+\langle Au,u'\rangle_{H^1}+\langle
B(u'),u'\rangle_W=(F(u),u')_{H^0}\qquad\text{a.e. in $(0,T)$}.
\end{equation}
 Since
$u\in W^{1,\infty}(0,T;H^1)\cap W^{2,\infty}(0,T;H^0)$, so $Au\in
W^{1,\infty}(0,T;(H^1)')$, by standard time--regularization we have
$\|u'\|_{H^0}^2, \langle Au,u\rangle_{H^1}\in W^{1,\infty}(0,T)$ and
$$\left(\|u'\|_{H^0}^2\right)'=2(u'',u')_{H^0},\quad \langle Au,u\rangle _{H^1}'=2\langle Au,u'\rangle_{H^1}
\qquad\text{a.e. in $(0,T)$,}$$ where the symmetry of $A$ is also
used. Since $(F(u),u')_{H^0}\in L^\infty(0,T)$, by \eqref{4.15bis} we
then get that $\langle B(u'),u'\rangle_W\in L^\infty(0,T)\subset
L^1(0,T)$ as $T<\infty$. Our claim then follows by (PQ3). Combining
it with Lemma~\ref{lemma4.1}--(ii) we get that $B(u')\in
L^{\overline{m}'}(0,T\, ;
[L^{\overline{m}}(\Omega,\lambda_\alpha)]')\times
L^{\overline{\mu}'}(0,T\, ;
[L^{\overline{\mu}}(\Gamma_1,\lambda_\beta)]')$. Since $F(u)\in
L^1(0,T;H^0)$ and trivially $u'\in W^{1,1}(0,T;X')$, by Riesz
theorem and Lemma~\ref{lemma3.2} we get \eqref{4.13}, concluding
Step 1.

{\noindent\bf Step 2: generalized solutions are also weak.} Let $u$ be a
generalized solution and $(u_n)_n$ a sequence of
strong solutions of \eqref{C1} converging  to $u$  in
$C([0,T];H^1)\cap C^1([0,T];H^0)$. By Step 1 and
Remark~\ref{remark4.5} the energy identity
$$\tfrac 12 \|u_n'\|_{H^0}^2+\tfrac 12
\langle Au_n,u_n\rangle_{H^1}\Big |_0^T+\int_0^T \langle
B(u_n'),u_n'\rangle_W=\int_0^T (F(u_n),u_n')_{H^0} $$ holds for all
$n\in\N$. Since by \eqref{4.12} we have $F(u_n)\to F(u)$ in
$C([0,T];H^0)$ we can pass to the limit in the last identity to get
\begin{equation}\label{4.17bis}
\tfrac 12 \|u'\|_{H^0}^2+\tfrac 12 \langle Au,u\rangle_{H^1}\Big
|_0^T+\lim_n\int_0^T \langle B(u_n'),u_n'\rangle_W=\int_0^T
(F(u),u')_{H^0}.
\end{equation}
Then, by (PQ3) and Lemma~\ref{lemma4.1}--(ii), it follows that
$u'_n$ and  $B(u_n')$ are (respectively) bounded in $Z(0,T)$  and
$L^{\overline{m}'}(0,T\, ;
[L^{\overline{m}}(\Omega,\lambda_\alpha)]')\times
L^{\overline{\mu}'}(0,T\, ;
[L^{\overline{\mu}}(\Gamma_1,\lambda_\beta)]')$. Hence, up to a
subsequence, $u_n'\to \psi$ and $B(u_n')\to \chi$ weakly in these
spaces. Since $u_n'\to u'$ in $L^2(0,T;H^0)$ and
$Z(0,T)\hookrightarrow L^2(0,T;H^0)$, it follows that $\psi=u'$, so
$u_n'\to u'$ weakly in $Z(0,T)$. We can pass to the limit in the
distribution identity \eqref{4.13} written, thanks to Step 1, for
$u_n$, and get
\begin{equation}\label{4.18bis}\int_0^T-(u',\phi')_{H^0}+\langle
Au,\phi\rangle_{H^1}+\langle \chi,\phi\rangle_W=\int_0^T
(F(u),\phi)_{H^0} \end{equation} for all $\phi\in C_c((0,T);H^1)\cap
C^1_c((0,T);H^0)\cap Z(0,T)$. By a further application of
Lemma~\ref{lemma3.2} we then get the energy identity
\begin{equation}\label{4.19bis}
\tfrac 12 \|u'\|_{H^0}^2+\tfrac 12 \langle Au,u\rangle_{H^1}\Big
|_0^T+\int_0^T \langle \chi,u'\rangle_W=\int_0^T (F(u),u')_{H^0}.
\end{equation}
Combining \eqref{4.17bis} and \eqref{4.19bis} we then get
$\lim_n\int_0^T \langle B(u_n'),u_n'\rangle_W=\int_0^T \langle
\chi,u'\rangle_W$. By Lemma~\ref{lemma4.1}--(ii--iii) and
\cite[Theorem~1.3~p.40]{barbu} $B$ is maximal monotone in $Z(0,T)$, so by the
classical monotonicity argument (see
\cite[Lemma~1.3~p.49]{barbu1993}
 we get $B(u')=\chi$ which, by \eqref{4.18bis}, concludes the proof.
\end{proof}

We not turn to the proofs of the results stated in
Section~\ref{section3}.

\begin{proof}[\bf Proof of Theorem~\ref{proposition4.1}]
We apply Theorem~\ref{theorem4.2} with $\cal{F}$ being given by
$$\cal{F}(U)=\begin{pmatrix}0\\-F(u)\end{pmatrix},\qquad\text{where}\quad U=\begin{pmatrix}u\\v\end{pmatrix},$$
which is trivially locally Lipschitz in $\cal{H}$ by \eqref{4.12}.
Consequently for any $(u_0,u_1)\in H^1\times H^0$ problem \eqref{C}
has a unique maximal generalized solution $u$ in
$[0,T_{\text{max}})$, which is the unique strong maximal solution of
it when (see \eqref{4.19}) also $u_1\in X$ and $Au_0+B(u_1)\in H^0$.
By Lemmas~\ref{lemma4.1bis} and \ref{lemma4.2} then $u$ is also the
unique weak solution of \eqref{C} in $[0,T_{\text{max}})$. By
Theorem~\ref{theorem4.2}--(ii) we then get \eqref{4.16} and the
maximality of $u$ among weak solutions of \eqref{C}. The continuous
dependence on the data in $H^1\times H^0$ then follows directly from
Theorem~\ref{theorem4.2}--(iii). Moreover when $F$ is globally Lipschitz also
$\cal{F}$ is globally Lipschitz, so all solutions are global in time.

To complete the proof we assume from now on that (PQ4) holds. By
\eqref{4.23bis}--\eqref{4.24bis}
there is $C=C(\|\alpha\|_\infty, \|\beta\|_{\infty,\Gamma_1},
c_m''', c_\mu''')\ge 0$ such that
\begin{equation}\label{4.25bis}
\widetilde{c_m}''\|v_\Omega-w_\Omega\|_{\overline{m},\alpha}^{\overline{m}}+
\widetilde{c_\mu}''\|v_\Gamma-w_\Gamma\|_{\overline{\mu},\beta,\Gamma_1}^{\overline{\mu}}
\le C |v-w|_{H^0}^2+\langle B(v)-B(w),v-w\rangle_W
\end{equation}
for any
$v=(v_\Omega,v_\Gamma),w=(w_\Omega,w_\Gamma)\in W$, where $\widetilde{c_m}''>0$ provided $m>\romega$ and
$\widetilde{c_\mu}''>0$ provided $\mu>\rgamma$. Then, using part (ii) of
the statement, the energy identity \eqref{4.14} and \eqref{4.25bis}
we get that $u_n'\to u'$ in
$L^{\overline{m}}(0,T^*;L^{2,\overline{m}}_\alpha(\Omega))$ provided
$m>\romega$ and $({u_n}_{|\Gamma})'\to (u_{|\Gamma})'$ in
$L^{\overline{\mu}}(0,T^*;L^{2,\overline{\mu}}_\beta(\Gamma_1))$
provided $\mu>\rgamma$. Since these conclusions are automatic when
$m\le \romega$ and $\mu\le\rgamma$ we get $u_n'\to u'$ in
$Z(0,T^*)$.

Finally, when $(u_{0n},u_{1n})\to (u_0,u_1)$ in
$H^{1,\rho,\theta}_{\alpha,\beta}\times H^0$ for some $\rho,\theta$
satisfying \eqref{special}, we recall \eqref{4.11bis} and we note
that  $u_n'\to u'$ in $Z(0,T^*)$ and $u_{0n}\to u_0$ in
$H^{1,\rho,\theta}_{\alpha,\beta}$ yields by a simple integration in
time that $u_n\to u$ in
$W^{1,1}(0,T^*;H^{1,\rho,\theta}_{\alpha,\beta})\hookrightarrow
C([0,T^*],H^{1,\rho,\theta}_{\alpha,\beta})$, concluding the proof.
\end{proof}

\begin{proof}[\bf Proof of Theorem~\ref{theorem4.1}]
It follows immediately by applying Theorem~\ref{proposition4.1}
with $F=(\widehat{f},\widehat{g})$, which satisfies \eqref{4.12} as a
consequence of  (FG1--2)  and \eqref{6}.
\end{proof}

\begin{proof}[\bf Proof of Theorems~\ref{theorem1} and \ref{theorem2}]
They are particular cases of
Theorem~\ref{theorem4.1}, by using
Remarks~\ref{remark4.1}, \ref{remark4.2bis},
\ref{remark4.3} and the fact that \eqref{8} is trivial when $m\le 2$
and $\mu\le 2$.
\end{proof}

\section{Regularity results}\label{section5}
This section is devoted to make explicit, when we are dealing with problem \eqref{1}, so $F=(\widehat{f},\widehat{g})$, the meaning of strong
solutions of problem \eqref{C} found in
Theorem~\ref{proposition4.1}.  In this way we shall get
our main regularity result for problem \eqref{1}. We shall from now
on assume that
\begin{equation}\label{regass}
\text{$ \Gamma$ is $C^2$ and $\overline{\Gamma_0}\cap
\overline{\Gamma_1}=\emptyset$.}
\end{equation}
Recalling the discussion made in Section~\ref{section 2} on Sobolev
spaces on compact $C^k$ manifolds, we remark that by the arguments
used in \cite[pp.~38-42]{lionsmagenes1}, when $M$ is a $C^2$ compact
manifold, then
\begin{alignat}2
\label{5.1}C^2(M)\quad\text{is dense in $W^{s,\rho}(M)$}
\qquad &\text{for $-2\le s\le2$}&&\quad \text{and $\rho\in (1,\infty)$;}\\
\label{5.2}[W^{s,\rho}(M]'\simeq W^{-s,\rho'}(M)\qquad&\text{for
$-2\le s\le2$ } && \quad\text{and $\rho\in (1,\infty)$.}
\end{alignat}

Since by \eqref{regass}, $\Gamma$, $\Gamma_0$ and $\Gamma_1$ are
compact $C^2$ manifolds, \eqref{5.1} and \eqref{5.2}  hold true when
$M=\Gamma$ and $M=\Gamma_i$, $i=0,1$.

 We now recall some fact on the Laplace-Beltrami
operator $\Delta_M$, which we shall use when  $M=\Gamma$ and
$M=\Gamma_i$, $i=0,1$, referring to \cite{taylor} for
more details and proofs, given there for smooth manifolds.  One
easily sees that the $C^2$ regularity of $M$ and the $C^1$ regularity of $(\cdot,\cdot)_M$ are enough. Then
$\Delta_M$ can be at first defined on $C^2(M)$ by the formula
\begin{equation}\label{26}
-\int_M \Delta_M u \,v=\int_M (\nabla_M u,\nabla_M v)_M
\end{equation}
for any $u,v\in C^2(M)$, and
$\Delta_M
u=g^{-1/2}\partial_i\left(g^{ij}g^{1/2}\partial_ju\right)$  in local coordinates.
Consequently $\Delta_M$ uniquely extends
\begin{footnote}{here we are implicitly considering $\Delta_M$ as the restriction to
real--valued distributions of the same operator acting on Sobolev
spaces of complex--valued distributions, which will be studied in
Appendix~\ref{appendixB}.
 }\end{footnote}
 to a bounded
linear operator from $W^{s+1,\rho}(M)$ to $W^{s-1,\rho}(M)$ for any
$s\in [-1,1]$ and $1<\rho<\infty$ (see
\cite[Lemma~1.4.1.3~pp.~21--24]{grisvard}).
Since $\Delta_M 1=0$ the operator is not injective.   The isomorphism properties of $-\Delta_M+I$ are given in  Lemma~\ref{propositionA1} in Appendix~\ref{appendixB}.

Since the characteristic functions $\chi_{\Gamma_0}$ ,
$\chi_{\Gamma_1}$ of $\Gamma_0, \Gamma_1$ are $C^2$ on $\Gamma$, by
identifying the elements of $W^{s,\rho}(\Gamma_i)$, $i=0,1$, with
their trivial extensions to $\Gamma$  we have the decomposition
\begin{equation}\label{5.4}
W^{s,\rho}(\Gamma)=W^{s,\rho}(\Gamma_0)\oplus W^{s,\rho}(\Gamma_1),
\quad\text{for $\rho\in (1,\infty)$, $-2\le s\le 2$,}
\end{equation}
so in particular $W^{s,\rho}(\Gamma_1)=\{u\in W^{s,\rho}(\Gamma):
u=0\,\,\text{in $\Gamma_0$}\}$, coherently with \eqref{ID}. By
\eqref{5.4} we also have
$\Delta_\Gamma=\Delta_{\Gamma_0}+\Delta_{\Gamma_1}$, hence
$\Delta_\Gamma u=\Delta_ {\Gamma_1}u$ for $u\in
W^{s,\rho}(\Gamma_1)$.

We recall here some classical facts on  the distributional normal
derivative. For any $u\in W^{1,\rho}(\Omega)$, $1<\rho<\infty$, such
that $-\Delta u=h\in L^\rho(\Omega)$ in the sense of distributions,
we set   $\partial_\nu u\in W^{-1/\rho,\rho}(\Gamma)$ by
\begin{footnote}{$\mathbb{D}$ was defined in
subsection~\ref{subsec2.2}}\end{footnote}
\begin{equation}\label{5.5}\langle \partial_\nu u
,\psi\rangle_{W^{1-1/\rho',\rho'}(\Gamma)}=-\int_\Omega
h\mathbb{D}\psi+\int_\Omega \nabla u\nabla
(\mathbb{D}\psi)\quad\text{for all $\psi\in
W^{1-1/\rho',\rho'}(\Gamma)$}.\end{equation} The operator
$u\mapsto\partial_\nu u$ is linear and bounded from
$D_\rho(\Delta)=\{u\in W^{1,\rho}(\Omega): \Delta u\in
L^\rho(\Omega)\}$, equipped with the graph norm, to
$W^{-1/\rho,\rho}(\Gamma)$. Moreover, since for any $\Psi\in
W^{1,\rho'}(\Omega)$ such that $\Psi_{|\Gamma}=\psi$ we have
$\Psi-\mathbb{D}\psi\in W^{1,\rho'}_0(\Omega)$,  \eqref{5.5} extends
to
\begin{equation}\label{5.5bis}\langle \partial_\nu u
,\psi\rangle_{W^{1-1/\rho',\rho'}(\Gamma)}=-\int_\Omega
h\Psi+\int_\Omega \nabla u\nabla\Psi\quad\text{for all $\psi\in
W^{1-1/\rho',\rho'}(\Gamma)$}.\end{equation} Moreover, by
\eqref{5.4}, we have $\partial_\nu u=\partial_\nu
u_{|\Gamma_0}+\partial_\nu u_{|\Gamma_1}$ and
$\psi=\psi_{|\Gamma_0}+\psi_{|\Gamma_1}$, where $\partial_\nu
u_{|\Gamma_i}\in W^{-1/\rho,\rho}(\Gamma_i)$, $\psi_{|\Gamma_i}\in
W^{1-1/\rho',\rho'}(\Gamma_i)$, $i=0,1$, and by \eqref{5.2},
\begin{equation}\label{5.6}
\langle \partial_\nu u
,\psi\rangle_{W^{1-1/\rho',\rho'}(\Gamma)}=\sum\nolimits_{i=0}^1\langle
\partial_\nu u_{|\Gamma_i}
,\psi_{|\Gamma_i}\rangle_{W^{1-1/\rho',\rho'}(\Gamma_i)}
\end{equation}
for all $\psi\in W^{1-1/\rho',\rho'}(\Gamma)$. Hence, in particular,
\begin{equation}\label{asteriscus}
\langle \partial_\nu u_{|\Gamma_1}
,\psi\rangle_{W^{1-1/\rho',\rho'}(\Gamma_1)}=-\int_\Omega
h\Psi+\int_\Omega \nabla u\nabla\Psi
\end{equation}
for all $\psi\in W^{1-1/\rho',\rho'}(\Gamma_1)$ and all $\Psi\in
W^{1,\rho'}(\Omega)$ such that $\Psi_{|\Gamma}=\psi$. Finally, when
$u\in W^{2,\rho}(\Gamma)$ the so--defined normal derivatives
coincide with the ones given by the already recalled trace theorem,
that is $\partial_\nu u\in  W^{2-1/\rho,\rho}(\Gamma)  $ and
$\partial_\nu u_{|\Gamma_i}\in  W^{2-1/\rho,\rho}(\Gamma_i)$ , $i=0,1$.

Our main regularity result is the following one.
\begin{thm} \label{theorem5.1}
Suppose that  (FG1--2) , (PQ1--3), \eqref{6} and \eqref{regass} hold true,
and
let $l,\lambda$ be the exponents defined in \eqref{ls}.
 Then, if
$$
(u_0,u_1)\in W^{2,l}\times X,\quad
 -\Delta u_0+\widehat{P}(u_1)\in
 L^2(\Omega),\quad\partial_\nu
{u_0}_{|\Gamma_1}-\Delta_{\Gamma}{u_0}+
\widehat{Q}({u_1}_{|\Gamma}\!)\in L^2(\Gamma_1), $$ the weak maximal
solution $u$ of problem \eqref{1} found in Theorem~\ref{theorem4.1}
enjoys the further regularity
\begin{gather}\label{6.4}
u\in L^\lambda([0,T_{\text{max}}); W^{2,l})\cap
C^1_w([0,T_{\text{max}});H^1)\cap
W^{2,\infty}_{\text{loc}}([0,T_{\text{max}});H^0),\\\label{6.5}
u'\in C_w([0,T_{\text{max}});X).
\end{gather}
Moreover
\begin{alignat}2\label{6.8}
 &u_{tt}-\Delta
u+\widehat{P}(u_t)=\widehat{f}(u)\quad &&\text{in $L^l(\Omega)$, a.e. in
$(0,T_{\text{max}})$,}\\
\label{6.9}&(u_{|\Gamma})_{tt}+\partial_\nu
u_{|\Gamma_1}\!\!-\Delta_{\Gamma}
u_{|\Gamma}+\widehat{Q}((u_{|\Gamma})_t)=\widehat{g}(u_{|\Gamma})\,\,
&&\text{in $L^l(\Gamma_1)$, a.e. on $(0,T_{\text{max}})$.}
\end{alignat}
\end{thm}

\begin{rem}\label{remark5.2}
By \eqref{6.5} one easily sees, integrating in time, that when
$u_0\in W^{2,l}\cap X$ then $u\in C^1_w([0,T_{\text{max}});X)$.
\end{rem}

\begin{rem}\label{remark5.1} By \eqref{6.4}--\eqref{6.5} $u$ and  all terms
in \eqref{6.8}  possess  a representative in
$L^1_\text{loc}((0,T_{\text{max}})\times\Omega)$ and all derivatives
in it are actually derivatives in the sense of distributions in
$(0,T_{\text{max}})\times\Omega$. The same remarks
\begin{footnote}{by the way a distribution in
$(0,T_{\text{max}})\times\Gamma_1$ is an elements of the dual of
$C^2_c((0,T_{\text{max}})\times\Gamma_1)$}\end{footnote} apply to
$u_{|\Gamma}$ and all terms in \eqref{6.9} in
$(0,T_{\text{max}})\times\Gamma_1$, and one easily proves that \eqref{6.8}--\eqref{6.9} are
equivalent to $u_{tt}-\Delta
u+P(\cdot,u_t)=f(\cdot,u)$ a.e. in $(0,T_{\text{max}})\times\Omega$ and $(u_{|\Gamma})_{tt}+\partial_\nu
u_{|\Gamma_1}-\Delta_{\Gamma}
u_{|\Gamma}+Q(\cdot,(u_{|\Gamma})_t)=g(\cdot,u_{|\Gamma})$,
a.e. on $(0,T_{\text{max}})\times\Gamma_1$.
\end{rem}
Before proving Theorem~\ref{theorem5.1} we  characterize the
domain of the operator $\cal{A}$ in \eqref{4.19}.
\begin{lem}\label{lemma5.1}
Let (PQ1--3) and \eqref{regass} hold. Then $D(\cal{A})=D_1$, where
$$
D_1:=\{(u,v)\in W^{2,l}\times X :
 -\Delta u+\widehat{P}(v)\in
 L^2(\Omega),\,\,\partial_\nu
u_{|\Gamma_1}-\Delta_{\Gamma}{u}+ \widehat{Q}(v)\in L^2(\Gamma_1)\},$$
and
\begin{equation}\label{rappresentazione}
Au+B(v)=(-\Delta u+\widehat{P}(v),\partial_\nu
u_{|\Gamma_1}-\Delta_{\Gamma}{u}+ \widehat{Q}(v))\quad\text{for all
$(u,v)\in D_1$.}
\end{equation}
\end{lem}
\begin{proof}
 By \eqref{3.22}, \eqref{4.3}, \eqref{4.4} and \eqref{4.19}
clearly $(u,v)\in D(\cal{A})$ if and only if $u\in H^1$, $v\in X$
and there are $h_1\in L^2(\Omega)$, $h_2\in L^2(\Gamma_1)$ such
that, for all $\phi\in X$,
\begin{equation}\label{6.11}
\int_\Omega \nabla u\nabla \phi+\int_{\Gamma_1}(\nabla_\Gamma
u,\nabla_\Gamma \phi)_\Gamma+\int_\Omega
\widehat{P}(u)\phi+\int_{\Gamma_1}\widehat{Q}(v)\phi=\int_\Omega
h_1\phi+\int_{\Gamma_1}h_2\phi.
\end{equation}
To prove that  $D(\cal{A})\subseteq D_1$ we fix $(u,v)\in
D(\cal{A})$. Taking $\phi\in C^\infty_c(\Omega)$ in \eqref{6.11} we
immediately get that
\begin{equation}\label{6.12}
-\Delta u+\widehat{P}(v)=h_1\qquad\text{in the sense of distributions in
$\Omega$.}
\end{equation}
We set $\widetilde{\romega}=\romega$ if $N \ge 3$, while $\widetilde{\romega}=2m$ if
$N=2$, so that $H^1(\Omega)\hookrightarrow L^{\widetilde{\romega}}(\Omega)$.
Hence, as $v\in X$, using  \eqref{Pgrowth} and Sobolev embedding we
have $\widehat{P}(v)\in L^{m_1}(\Omega)$, where
$m_1:=\max\{m,\widetilde{\romega}\}/(m-1)$. By \eqref{6.12} then $-\Delta u\in
L^{m_2}(\Omega)$ in the sense of distributions where, as $m_1>2$
when $N=2$,
\begin{equation}\label{6.13}
m_2:=\min\{2,m_1\}=\min\{2, \max\{m,\romega\}/(m-1)\}\le 2.
\end{equation}
 Since
$u\in H^1(\Omega)\subset W^{1,m_2}(\Omega)$ it has a distributional
derivative $\partial_\nu u_{|\Gamma_1}\in W^{-\frac 1{m_2},m_2}(\Gamma_1)$
and then, by \eqref{asteriscus}, we can rewrite \eqref{6.11} as
\begin{equation}\label{6.14}
\langle \partial_\nu u_{|\Gamma_1}
,\phi_{|\Gamma}\rangle_{W^{1-1/m_2',m_2'}(\Gamma_1)}+
\int_{\Gamma_1}(\nabla_\Gamma u,\nabla_\Gamma
\phi)_\Gamma+\int_{\Gamma_1}\widehat{Q}(v)\phi=\int_{\Gamma_1}h_2\phi.
\end{equation}
for all $\phi\in X$ such that $\phi_{|\Gamma}\in
W^{1-1/m_2',m_2'}(\Gamma_1)$. Since for any $\psi\in
C^2(\Gamma_1)\subset W^{1,2N}(\Gamma_1)$ we have $\mathbb{D}\psi\in
W^{1,2N}(\Omega)\subset C(\overline{\Omega})$ , by Morrey's theorem,
so $\mathbb{D}\psi\in X$, from \eqref{6.14} we can conclude that
\begin{equation}\label{6.15}
\partial_\nu
u_{|\Gamma_1}-\Delta_{\Gamma}{u}+ \widehat{Q}(v)=h_2\qquad \text{in
$[C^2(\Gamma_1)]'$.}
\end{equation}
We now set $\widetilde{\rgamma}=\rgamma$ if $N=2$ and $N \ge 4$, while
$\widetilde{\rgamma}=2\mu$ if $N=3$, so that $H^1(\Gamma_1)\hookrightarrow
L^{\widetilde{\rgamma}}(\Gamma_1)$. Hence, as $v\in X$, using  \eqref{Qgrowth}
and Sobolev embedding we have $\widehat{Q}(v)\in L^{\mu_1}(\Gamma_1)$,
where $\mu_1:=\max\{\mu,\widetilde{\rgamma}\}/(\mu-1)$. By \eqref{6.15} then
$\partial_\nu u_{|\Gamma_1}-\Delta_{\Gamma}{u}=h_3$ in the sense of
$[C^2(\Gamma_1)]'$ where   $h_3\in L^{\mu_2}(\Gamma_1)$ and, as
$\mu_1>2$ when $N=3$,
\begin{equation}\label{6.16}
\mu_2:=\min\{2,\mu_1\}=\min\{2, \max\{\mu,\rgamma\}/(\mu-1)\}\le 2.
\end{equation}
Since $\partial_\nu u_{|\Gamma_1}\in W^{-1/m_2,m_2}(\Gamma_1)$ and,
by \eqref{ls}, \eqref{6.13} and \eqref{6.16} we have
$l=\min\{m_2,\mu_2\}$, we get $-\Delta_{\Gamma_1}u\in
W^{-1/l,l}(\Gamma_1)$. By  Lemma~\ref{propositionA1}  then we get $u\in
W^{2-1/l,l}(\Gamma_1)$. Since $-\Delta u\in L^l(\Omega)$  as $l\le
m_2$ we then get by elliptic regularity (see
\cite[Theorem~2.4.2.5~p. 124]{grisvard}) that $u\in
W^{2,l}(\Omega)$, so \eqref{6.12} holds in $L^l(\Omega)$ and
$\partial_\nu u_{|\Gamma_1}\in W^{1-1/l,l}(\Gamma_1)$ by the trace
theorem. Plugging this information in \eqref{6.15} we then get, as
$l\le \mu_2$, that $-\Delta_{\Gamma_1}u\in L^l(\Gamma_1)$ so
\eqref{6.15} holds true in this space and, by a further application
of  Lemma~\ref{propositionA1}  then $u\in W^{2,l}(\Gamma_1)$, proving that
$D(\cal{A})\subseteq D_1$.

To prove that $D_1\subseteq D(\cal{A})$ let $(u,v)\in D_1$. We
denote  $h_1=-\Delta u+\widehat{P}(v)\in L^2(\Omega)$ and
$h_2=\partial_\nu u_{|\Gamma_1}-\Delta_{\Gamma}{u}+ \widehat{Q}(v)\in
L^2(\Gamma_1)$. Since $\widehat{P}(v)\in L^l(\Omega)$ as $l\le m_1$ we
have
$$\int_\Omega -\Delta u\phi+\int_\Omega \widehat{P}(v)\phi=\int_\Omega
h_1\phi\qquad\text{for all $\phi\in L^{l'}(\Omega)$}.$$ We now point
out that the classical integration by parts formula
\begin{equation}\label{intbyparts}
\int_\Omega \nabla h\nabla k+\int_\Omega \Delta h
k=\int_\Gamma\partial_\nu h \,k \end{equation} which is standard
when $h\in H^2(\Omega)$ and $k\in H^1(\Omega)$ (see \cite[Lemma
1.5.3.7 p.59]{grisvard}) extends to $h\in W^{2,l}(\Omega)\cap
H^1(\Omega)$ and $k\in H^{1,l',l'}_{1,1}$. Indeed, by using
\cite[Theorem 4.26, p. 84]{adams}, $h$ can be approximated in
$W^{2,l}(\Omega)\cap H^1(\Omega)$ by a sequence $(h_n)$ in
$C^2(\overline{\Omega})\subset H^2(\Omega)$, so we can pass to the
limit in \eqref{intbyparts} as $k\in L^{l'}(\Omega)$ and
$k_{|\Gamma}\in L^{l'}(\Gamma)$. Hence we get that \eqref{6.11}
holds for all $\phi\in H^{1,l',l'}_{1,1}$. By Lemma \ref{lemma3.1}
then \eqref{6.11} holds for all $\phi\in X$, so proving that
$D_1\subseteq D(\cal{A})$ and \eqref{rappresentazione} holds,
concluding the proof.
\end{proof}
\begin{proof}[Proof of Theorem \ref{theorem5.1}]
By Lemma \ref{lemma5.1} we have $(u_0,u_1)\in D(\cal{A})$, hence by
Theorem \ref{proposition4.1} the maximal solution of \eqref{1}
is actually a strong solution of \eqref{C} when
$F=(\widehat{f},\widehat{g})$, so $u\in
W^{1,\infty}_{\text{loc}}([0,T_{\text{max}});H^1)\cap
W^{2,\infty}_{\text{loc}}([0,T_{\text{max}});H^0)$, $(u(t),u'(t))\in
D_1$ for all $t\in [0,T_{\text{max}})$ and,  by
\eqref{rappresentazione}, \eqref{6.8}--\eqref{6.9} hold true.

To prove \eqref{6.4} we note that, since $\widehat{P}(u_t)\in
L^{\infty}_{\text{loc}}([0,T_{\text{max}});L^{\widetilde{\romega}/(m-1)}(\Omega))$
when $m\le \widetilde{\romega}$ and $\widehat{P}(u_t)\in
L^{m'}_{\text{loc}}([0,T_{\text{max}});L^{m'}(\Omega))$ when $m>
\widetilde{\romega}$, we have
\begin{equation}\label{Preg}\widehat{P}(u_t)\in
L^{\lambda_1}_{\text{loc}}([0,T_{\text{max}});L^{m_1}(\Omega)),
\end{equation}
where $\lambda_1=\infty$ when $m\le \romega$ and $\lambda_1=m'$
otherwise. Since moreover $u_{tt}, \widehat{f}\in
L^{\infty}_{\text{loc}}([0,T_{\text{max}});L^2(\Omega))$ we then get
from \eqref{6.8} that
\begin{equation}\label{6.22}
\Delta u\in
L^{\lambda_1}_{\text{loc}}([0,T_{\text{max}});L^{m_2}(\Omega))\subset
L^{\lambda}_{\text{loc}}([0,T_{\text{max}});L^l(\Omega)).
\end{equation}
Consequently, by the boundedness of the distributional normal
derivatives,  $\partial_\nu u|_{|\Gamma_1}\in
L^{\lambda}_{\text{loc}}([0,T_{\text{max}});W^{-1/l,l}(\Gamma_1))$.
Since $\widehat{Q}((u_{|\Gamma})_t\in
L^{\infty}_{\text{loc}}([0,T_{\text{max}});L^{\widetilde{\rgamma}/(\mu-1)}(\Gamma_1))$
when $\mu\le \widetilde{\rgamma}$ and $\widehat{Q}((u_{|\Gamma})_t\in
L^{\mu'}_{\text{loc}}([0,T_{\text{max}});L^{\mu'}(\Gamma_1))$ when
$\mu> \widetilde{\rgamma}$, we have
\begin{equation}\label{Qreg}
\widehat{Q}((u_{|\Gamma})_t\in
L^{\lambda_2}_{\text{loc}}([0,T_{\text{max}});L^{\mu_1}(\Gamma_1)),
\end{equation}
where $\lambda_2=\infty$ when $\mu\le \rgamma$ and $\lambda_2=\mu'$
otherwise.

Since moreover $(u_{|\Gamma})_{tt}, \widehat{g}\in
L^{\infty}_{\text{loc}}([0,T_{\text{max}});L^2(\Gamma))$ we then get
from \eqref{6.9} that $\Delta_{\Gamma_1} u_{|\Gamma_1}\in
L^{\lambda}_{\text{loc}}([0,T_{\text{max}});W^{-1/l,l}(\Gamma_1))$.
As $l\le 2$, from Lemma \ref{propositionA1} we get $u_{|\Gamma_1}\in
L^{\lambda}_{\text{loc}}([0,T_{\text{max}});W^{2-1/l,l}(\Gamma_1))$.
By \eqref{6.22} and the already quoted elliptic regularity result we
have $u\in
L^{\lambda}_{\text{loc}}([0,T_{\text{max}});W^{2,l}(\Omega))$, and
consequently we get $\partial_\nu u|_{|\Gamma_1}\in
L^{\lambda}_{\text{loc}}([0,T_{\text{max}});W^{1-1/l,l}(\Gamma_1))$.
Using \eqref{6.9} again then we get $$\Delta_{\Gamma_1}u \in
L^{\lambda}_{\text{loc}}([0,T_{\text{max}});L^l(\Gamma_1)),$$ hence
by Lemma \ref{propositionA1} we have $u_{|\Gamma_1} \in
L^{\lambda}_{\text{loc}}([0,T_{\text{max}});W^{2,l}(\Gamma_1))$.

Since $u_t\in C([0,T_{\text{max}});H^0)\cap
L^\infty_{\text{loc}}(0,T_{\text{max}};H^1)$, by  \cite[Theorem 2.1
p.544]{strauss} and Lemma \ref{lemma3.2} we then get $u_t\in
C_w([0,T_{\text{max}});H^1)$, completing the proof of \eqref{6.4}.

To prove \eqref{6.5} we remark that, as shown is the proof of Lemma
\ref{lemma4.2} -- Step 1, we have $\langle B(u'),u'\rangle_w\in
L^\infty_{\text{loc}}([0,T_{\text{max}}))$, hence by (PQ3) we have
$u'\in L^\infty_{\text{loc}}([0,T_{\text{max}});W)$ and consequently
using Lemma \ref{lemma3.2} and the result by W. Strauss already
recalled we get \eqref{6.5}, completing the proof.
\end{proof}

When the damping terms are not supersupercritical,  the time
regularity \eqref{6.4} can be improved as follows.
\begin{cor}\label{corollary5.1}
Suppose that all assumptions in Theorem \ref{theorem5.1} hold true,
and moreover suppose
\begin{equation}\label{6.23}
1<m\le \romega\qquad \text{and}\quad 1<\mu\le \rgamma.
\end{equation}
Then, if
$$
(u_0,u_1)\in W^{2,l}\times H^1,\quad
 -\Delta u_0+\widehat{P}(u_1)\in
 L^2(\Omega),\quad\partial_\nu
{u_0}_{|\Gamma_1}-\Delta_{\Gamma}{u_0}+
\widehat{Q}({u_1}_{|\Gamma}\!)\in L^2(\Gamma_1), $$ the regularity
\eqref{6.4} is improved to
\begin{equation}\label{6.24}
u\in C_w([0,T_{\text{max}}); W^{2,l})\cap
C^1_w([0,T_{\text{max}});H^1)\cap C^2_w([0,T_{\text{max}});H^0).
\end{equation}
Consequently, when
\begin{equation}\label{6.26}
1<m\le 1+\frac\romega 2\qquad \text{and}\quad 1<\mu\le \frac\rgamma
2,
\end{equation}
for initial data $(u_0,u_1)\in H^2\times H^1$ we have
\begin{equation}\label{6.27} u\in C_w([0,T_{\text{max}}); H^2)\cap
C^1_w([0,T_{\text{max}});H^1)\cap C^2_w([0,T_{\text{max}});H^0).
\end{equation}
\end{cor}
\begin{proof} When \eqref{6.23} holds, by
\eqref{ls} we have $\lambda=\infty$, hence by \eqref{6.4} we get
$u\in L^\infty_{\text{loc}}([0,T_{\text{max}}); W^{2,l})$. Since
$W^{2,l}(\Omega)$ and $W^{2,l}(\Gamma_1)$ are (respectively) dense
in $H^1(\Omega)$ and $H^1(\Gamma_1)$, by \cite[Theorem 2.1
p.544]{strauss} we get $u\in C_w([0,T_{\text{max}}); W^{2,l})$.
Hence $\Delta u\in C_w([0,T_{\text{max}});L^l(\Omega))$ and
$\Delta_{\Gamma_1} u-\partial_\nu u_{|\Gamma_1}\in
C_w([0,T_{\text{max}});L^l(\Gamma_1))$.

Moreover, by \eqref{Preg} and \eqref{Qreg}, we also have $P(u_t)\in
L^\infty_{\text{loc}}([0,T_{\text{max}}); L^l(\Omega))$ and
$Q((u_{\Gamma})_t)\in L^\infty_{\text{loc}}([0,T_{\text{max}});
L^l(\Gamma_1))$. Hence, as $\widehat{f}(u)\in
C_w([0,T_{\text{max}}];L^2(\Omega))$ and $\widehat{g}(u_{|\Gamma})\in
C_w([0,T_{\text{max}}];L^2(\Gamma_1))$, by \eqref{6.8}--\eqref{6.9}
we get $u''\in C_w([0,T_{\text{max}});L^l(\Omega)\times
L^l(\Gamma_1))$. Hence by \eqref{6.4}, the density of $H^0$ in
$L^l(\Omega)\times L^l(\Gamma_1)$ and \cite[Theorem 2.1
p.544]{strauss} again we get $u''\in C_w([0,T_{\text{max}});H^0)$,
concluding the proof of \eqref{6.24}.

When \eqref{6.26} holds we also have $l=2$ and for data
$(u_0,u_1)\in H^2\times H^1$ the conditions $-\Delta u_0+
\widehat{P}(u_1)\in
 L^2(\Omega)$ and  $\partial_\nu
{u_0}_{|\Gamma_1}-\Delta_{\Gamma}{u_0}+
\widehat{Q}({u_1}_{|\Gamma}\!)\in L^2(\Gamma_1)$ are automatic, so
\eqref{6.27} holds.
\end{proof}

\begin{proof}[\bf Proof of Theorems~\ref{theorem3} and \ref{theorem4}]
By  Remarks~\ref{remark4.1}, \ref{remark4.2bis} and \ref{remark4.3}
they are particular cases of Theorem~\ref{theorem5.1} and
Corollary~\ref{corollary5.1}.
\end{proof}

\section{Global existence versus blow--up}\label{section6}

This section is devoted to our global existence and blow--up results for problem \eqref{1}. Before giving them we need some preliminaries. We shall assume in the sequel  that assumptions (PQ1--3), (FG1--2) and \eqref{6} hold true.

\subsection{The energy function}\label{preliminarienergia}
To use energy methods it is fruitful to introduce an energy function involving the potential operator associated to
$F=(\widehat{f},\widehat{g})$, and to write the energy identity \eqref{enidd} in terms of this function.

We first need to point out the following abstract version of the classical chain rule, which easy proof is given for the reader's convenience.
\begin{lem}\label{chainrule} Let $V_0$ and $H_0$ be real Hilbert spaces, $I$ be a real interval
and denote by $(\cdot,\cdot)_{H_0}$ the scalar product of $H_0$. Suppose that $V_0\hookrightarrow H_0$ with dense embedding, so that
$V_0\hookrightarrow H_0\simeq H_0'\hookrightarrow V_0'$.

Then for any $J_1\in C^1(V_0)$ such that its Fr\`{e}chet derivative $J_1'$ is locally Lipschitz from $V_0$ to $H_0$ and
any $w\in C(I;V_0)\cap C^1(I;H_0)$ we have
 $J_1\cdot w\in C^1(I)$ and $(J_1\cdot w)'=(J_1'\cdot w,w')_{H_0}$ in $I$, where $\cdot$ denotes the composition product.
\end{lem}
\begin{proof} At first we note that $w$ can be trivially extended, with the same regularity, to the whole of $\R$,
so we assume without restriction that $I=\R$. Next we remark that our claim reduces to the well--known chain rule for the
Fr\`{e}chet derivative (see \cite[Proposition~1.4, p. 12]{ambrosettiprodi} when $w\in C^1(I;V_0)$.

In the general case, denoting by $(\rho)_n$ a standard sequence of mollifiers and  by $*$ the standard convolution product in $\R$,
we set $w_n=\rho_n *w$, so $w_n\in C^1(I;V_0)$ and, by previous remark,  $(J_1\cdot w_n)'=(J_1'\cdot w_n,w_n')_{H_0}$ in $\R$ for all $n\in\N$ .
Since the proof of  \cite[Proposition~4.21,~p. 108]{brezis2} trivially extends to vector--valued functions,
we also have  $w_n\to w$ in $V_0$ and $w_n'\to w'$ in $H_0$ locally uniformly in $\R$.  Since $J_1'$ is locally Lipschitz from $V_0$ to $H_0$
it then follows that  $J_1'\cdot w_n\to J_1'\cdot w$ in $H_0$ locally uniformly in $\R$ and consequently $(J_1\cdot w_n)'=(J_1'\cdot w_n,w_n')_{H_0}\to (J_1'\cdot w,w')_{H_0}$ locally uniformly in $\R$. Since $J_1\in C(V_0)$ we also have $J_1\cdot w_n\to J_1\cdot w$ in $\R$. Our assertion then follows
by standard results on uniformly convergent real sequences.
 \end{proof}

We introduce the primitives of the functions $f$ and $g$ by
\begin{equation}\label{fgprimitives}
\mathfrak{F}(x,u)=\int_0^u f(x,s)\,ds,\qquad\text{and}\quad \mathfrak{G}(y,u)=\int_0^u g(y,s)\,ds,
\end{equation}
for a.a. $x\in\Omega$, $y\in\Gamma_1$ and all $u\in\R$, and
we note that by (FG1) there are constants $c_p''',c_q'''\ge 0$ such that
\begin{equation}\label{nnn6.2}
|\mathfrak{F}(x,u)|\le c_p'''(1+|u|^p),\qquad \text{and}\quad|\mathfrak{G}(y,u)|\le c_q'''(1+|u|^q)
\end{equation}
for a.a. $x\in\Omega$, $y\in\Gamma_1$ and all $u\in\R$.

 By  \eqref{6}, \eqref{nnn6.2}  and  Sobolev embedding theorem we can set the potential operator
$J:H^1\to\R$ by
\begin{equation}\label{potentialJ}
J(v)=\int_\Omega \mathfrak{F}(\cdot,v)+\int_{\Gamma_1} \mathfrak{G}(\cdot,v_{|\Gamma})\qquad\text{for all $v\in H^1$.}
\end{equation}
By (FG1),  \eqref{6} and  Sobolev embedding theorem, using the  same arguments applied to prove \cite[Theorem~2.2, p. 16]{ambrosettiprodi}
one easily gets that $J\in C^1(H^1)$, its Fr\`echet derivative $J'$ being given by $F=(\widehat{f},\widehat{g})$, which is locally Lipschitz from
$H^1$ to $H^0$.

Hence, by \eqref{3.24} and Lemma~\ref{chainrule}, for any weak solution $u$ of \eqref{1}
we have $J\cdot u\in C^1(\text{dom }u)$ and
\begin{equation}\label{potentialderivative} (J\cdot u)'=\int_\Omega f(\cdot,u)u_t+\int_{\Gamma_1} g(\cdot,u_{|\Gamma})(u_{|\Gamma})_t.
\end{equation}

We also introduce the energy functional $\cal{E}\in C^1(\cal{H})$
defined by
 \begin{equation}\label{energyfunction}
\cal{E}(v,w)=\frac 12 \|w\|_{H^0}^2+\tfrac 12\int_\Omega |\nabla v|^2+\tfrac 12\int_{\Gamma_1}|\nabla_{\Gamma} v|_{\Gamma}^2-J(v),\quad\text{for all $(v,w)\in\cal{H}$,}
 \end{equation}
  and the energy function associated to a weak maximal solution $u$ of \eqref{1} by
 \begin{equation}\label{energyfunctionbis}
 \cal{E}_{u}(t)=\cal{E}(u(t),u'(t)),\qquad\text{for all $t\in [0,T_{\max})$}.
 \end{equation}
By \eqref{potentialderivative} and \eqref{energyfunction} the energy identity \eqref{enidd}
can be rewritten as
\begin{equation}\label{eniddd}
\cal{E}_u(t)-\cal{E}_u(s)+\int_s^t\langle B(u'),u'\rangle_W=0\quad\text{for all $s,t\in [0,T_{\max})$.}
\end{equation}
Consequently,  by Lemma~\ref{lemma4.1}, \eqref{3.6}, \eqref{energyfunction} and \eqref{energyfunctionbis}, $\cal{E}_u$ is decreasing and  we have
\begin{equation}\label{nnn6.6}
\begin{aligned}
\frac 12 \|u'(t)\|_{H^0}^2+\frac 12 \|u(t)\|_{H^1}^2 =&\cal{E}_u(0)-\int_0^t\langle B(u'),u'\rangle_W+\frac 12 \|u(t)\|_{H^0}^2+J(u(t))\\
\le &\cal{E}_u(0)+\frac 12 \|u(t)\|_{H^0}^2+J(u(t))\quad\text{for $t\in [0,T_{\max})$.}
\end{aligned}
\end{equation}
The alternative between global existence and blow--up depends on the specific structure of the nonlinearities involved.
We shall separately treat two different cases.
\subsection{Blow--up when damping terms are linear}
We shall consider in this subsection damping terms satisfying the following assumption
\renewcommand{\labelenumi}{{(PQ5)}}
\begin{enumerate}
\item there are $\alpha\in L^\infty(\Omega)$, $\beta\in L^\infty(\Gamma_1)$, $\alpha,\beta\ge 0$, such that
$$P(x,v)=\alpha(x) v,\qquad\text{and}\quad Q(y,v)=\beta(y)v$$
for a.a. $x\in\Omega$, $y\in\Gamma_1$ and all $v\in\R$.
\end{enumerate}
\begin{rem}\label{nuova6.1}
Trivially (PQ5) implies  assumptions (PQ1--3) with $m=\mu=2$, and in some sense override them.
Moreover in the case considered in problem \eqref{5} it reduces to assumption (I)$'$.
\end{rem}
Moreover we shall consider in this subsection source terms  satisfying the following specific assumption:
\renewcommand{\labelenumi}{{(FG4)}}
\begin{enumerate}
\item at least one between $f$ and $g$ is not a.e. vanishing and there are exponents $\overline{p},\overline{q}>2$ such that,
for a.a. $x\in\Omega$, $y\in\Gamma_1$ and all $u\in\R$,
\begin{equation}\label{blowupgrowth}
f(x,u)u\ge \overline{p}\,\mathfrak{F}(x,u)\ge 0, \qquad\text{and}\quad g(y,u)u\ge \overline{q}\,\mathfrak{G}(y,u)\ge 0.
\end{equation}
\end{enumerate}
\begin{rem}\label{nuova6.2}In the case considered in problem \eqref{5}, i.e. $f(x,u)=f_0(u)$ and $g(x,u)=g_0(u)$, assumption (FG4) reduces to
assumption (IV). Another specific example is is given by \eqref{f3g3} when $f_0$ and $g_0$ still satisfy (IV).
\end{rem}
We can now give our blow--up result
\begin{thm}\label{theorem6.1} Let (PQ4), (FG1--2), (FG4), \eqref{6} hold. Then
\renewcommand{\labelenumi}{{(\roman{enumi})}}
\begin{enumerate}
\item $N=\{(u_0,u_1)\in \cal{H}: \,\,\cal{E}(u_0,u_1)<0\}\not=\emptyset$, and
\item
for any $(u_0,u_1)\in N$ the unique maximal weak solution $u$ of \eqref{1} blows--up in finite time, that is $T_{\max}<\infty$, and \eqref{new6.4} holds.
\end{enumerate}
\end{thm}
\begin{proof} We first prove (ii). By (PQ5), we have $X=H^1$ and $W=H^0$.
Since we are going to apply \cite[Theorem~1]{lps} to
\eqref{1}, in the sequel we are going to check its assumptions.
 Referring to the notation of the quoted paper, adding a $_*$ subscript to it, we set
$$V_*=H^0,\qquad Y_*=H^0,\qquad W_*=H^1,\qquad X_*=L^p(\Omega)\times L^q(\Gamma_1)$$
where, according to \eqref{ID},   $L^q(\Gamma_1)$ is identified with $\{v\in L^q(\Gamma): v=0\,\,\text{on}\,\,\Gamma_0\}$.
We note that, by \eqref{6} and Sobolev embedding theorem we have $W_*\hookrightarrow X_*$.

We also set the operators
$$P_*:V_*\to V_*',\,\, A_*: W_*\to W_*', \,\, F_*: X_*\to X_*',\,\,Q_*:[0,\infty)\times Y_*\to Y_*'$$
by
$$P_*=\text{Id},\qquad A_*=A,\qquad F_*=F,\qquad Q_*=B,$$
nothing that in our case $A_*$, $F_*$ and $Q_*$ are autonomous so no explicit dependence on time is needed.
Trivially $P_*$ and  $Q_*$ are non--negative definite and symmetric. Moreover $A_*$ and $F_*$ are the Fr\`echet derivatives
of the $C^1$ potentials $\cal{A}_*: W_*\to \R$, $\cal{F}_*: X_*\to\R$
respectively given for all $v\in W_*$, $(v_1,v_2)\in X_*$ by
$$\cal{A}_*(v)=\frac 12 \|\nabla v\|_2^2+\frac 12 \|\nabla_\Gamma v\|_\Gamma^2\quad\text{and}\quad \cal{F}_*(v_1,v_2)= \int_\Omega \mathfrak{F}(\cdot,v_1)+\int_{\Gamma_1}\mathfrak{G}(\cdot,v_2).$$
With this setting the abstract evolution equation
\begin{equation}\label{abstractLPS}
P_*u''+Q_*u'+A_*(u)=F_*(u)
\end{equation}
considered in \cite[(2.1)]{lps} formally reduces to \eqref{C} and, taking
$G_*=W_*$ as the nontrivial subspace of $V_*$, $W_*$ and $Y_*$  we have $K_*=C([0,\infty);H^1)\cap C^1([0,\infty);H^0)$,
hence, by Definition~\ref{definition4.1} and \eqref{enidd}
strong solutions of \eqref{abstractLPS} in the sense of  \cite{lps} exactly reduce to weak solutions of \eqref{1} in $[0,\infty)$.
Moreover, to check the specific assumptions of \cite[Theorem~1]{lps} we note that, by (FG4) for all $v\in G_*=H^1$ we have
$$
\begin{aligned}
\langle A_*(v),v\rangle_{W_*}-\langle F_*(v),v\rangle_{X_*}=& \|\nabla v\|_2^2+\|\nabla_\Gamma v\|_\Gamma^2-\int_\Omega f(\cdot,v)-\int_{\Gamma_1}g(\cdot,v)\\
\le & 2\cal{A}_*(v)-\overline{p}\int_\Omega \mathfrak{F}(\cdot,v)-\overline{q}\int_{\Gamma_1}\mathfrak{G}(\cdot,v)\\
\le & q_*\left[\cal{A}_*(v)-\cal{F}_*(v)\right]
\end{aligned}
$$
where $q_*:=\min\{\overline{p},\overline{q}\}>2$, hence \cite[(2.3)]{lps} holds, while \cite[(2.4)]{lps} is trivially satisfied since $\cal{A}_*$ and
$\cal{F}_*$ does not depend on $t$.

By the quoted result we then get that \eqref{1} has no global weak solutions in $[0,\infty)$ with $(u_0,u_1)\in N$, and then $T_{\max}<\infty$. By Theorem~\ref{theorem4.1}
we  get the first limit in \eqref{new6.4}.
Consequently, by  \eqref{nnn6.2} and \eqref{nnn6.6}, since $p,q\ge 2$, we
also get  the second limit in \eqref{new6.4}.

We now prove (i), first considering the case in which $g$ does not vanish a.e. in $\Gamma_1\times\R$ (so $\Gamma_1\not=\emptyset$). Since $g(x,u)u\ge 0$ at least one between the two sets
\begin{equation}\label{sets}
E^\pm=\{(x,u)\in \Gamma_1\times\R: \pm g(x,u)>0,\,\pm u>0\}
\end{equation}
 has positive measure in $\Gamma_1\times\R$. In the sequel of the proof the symbol $\pm$ means $+$ is $E^+$ has positive measure and $-$ if $E^-$ has positive measure. Hence there are $C\subseteq\Gamma_1$, $\eps>0$ and
$\overline{u}\in\R$ such that $\pm \overline{u}>0$, $C\times(\overline{u}-\eps,\overline{u}+\eps)\subseteq E^\pm$ and $\sigma(C)>0$.
We denote
$$B_1=\{x'\in\R^{N-1}: |x'|<1\},\quad Q=B_1\times (-1,1),\quad Q^+=B_1\times (0,1),\quad Q^0=B_1\times\{0\}.$$
Since $\Gamma$ is $C^1$ and compact there is an open set $U_0$ in $\R^N$ and a coordinate map $\psi_1:Q\to U_0$, bijective and such that $\psi_1\in C^1(\overline{Q})$, $\psi_1^{-1}\in C^1(\overline{U_0})$, $\psi_1(Q^+)=U_0\cap\Omega$,
$\psi_1(Q^0)=U_0\cap\Gamma$ and $\sigma(U_0\cap C)>0$. We denote $\psi_2=\psi_1(\cdot,0):B_1\to\Gamma$, $\psi_2\in C^1(\overline{B_1})$ and
$D=\psi_2^{-1}(U_0\cap C)$. Since
$\sigma(U_0\cap C)=\int_D |\partial_{x_1}\psi_2\wedge\ldots\wedge\partial_{x_{N-1}}\psi_2|\,dx'$, where $x'=(x_1,\ldots,x_{N-1})$,
$D$ has positive measure in $R^{N-1}$ and hence it contains an open ball $B_2$ of radius $r>0$. We set $B=\psi_2(B_2)\subseteq U_0\cap C\subseteq \Gamma_1$
and $U_1=\psi_1(B_2\times (-1,1))$. Hence $U_1$ is open in $\R^N$ and $U_1\cap \Gamma_0=B\cap \Gamma_1=\emptyset$.

By \eqref{fgprimitives} and \eqref{sets}, since $B\times(\overline{u}-\eps,\overline{u}+\eps) \subseteq E^\pm$, we get
that $\phi_2:=\frak{G}(\cdot,\overline{u})>0$ a.e. in $B$. Integrating the differential inequality in \eqref{blowupgrowth} from $0$ to $\overline{u}$ and denoting $\phi_3=\phi_2|\overline{u}|^{-\overline{q}}$, we
get $\frak{G}(\cdot,u)\ge \phi_3|u|^{\overline{q}}$ a.e. in $B$ when $u-\overline{u}\in\R^\pm_0$, and consequently
\begin{equation}\label{primusinterpares}
\frak{G}(\cdot,u)\ge \phi_3|u|^{\overline{q}}-\phi_2,\qquad\text{a.e. in $B$, when  $u\in\R^\pm_0$.}
\end{equation}
Now (see \cite[p. 210]{brezis2}) we fix $\eta_0\in C^\infty(\R)$ such that $\eta_0(s)=1$ if $s<1/4$ and $\eta_0(s)=0$ if $s>3/4$.  Moreover we
denote $\overline{w_0}(x',x_N)=\eta_0(|x'|/r)\,\eta_0(|x_N|)$ for $(x',x_N)\in B_2\times (-1,1)$ and $w_0=\overline{w_0}\cdot \psi^{-1}$.
Hence $w_0\ge 0$, ${w_0}_{|\Gamma_1}\not\equiv 0$, $w_0\in C^1_c(U_1)$ and then, as $U_1\cap \Gamma_0=\emptyset$, $w_0\in H^1$.
Hence, by \eqref{energyfunction},  (FG4) and \eqref{primusinterpares} we have
$$
\cal{E}(sw_0,u_1)\le \tfrac 12 \|u_1\|_{H^0}^2+\tfrac 12\left(\|\nabla w_0\|_2^2+\|\nabla_\Gamma w_0\|_{2,\Gamma_1}^2\right)s^2-\left(\int_B\phi_3\,|w_0|^{\overline{q}} \right) |s|^{\overline{q}}+\|\phi_2\|_{1,\Gamma_1}
$$
for all $u_1\in H^0$ and $s\in\R^\pm_0$. Since $\overline{q}>2$ and $\int_B\phi_3\,|w_0|^{\overline{q}}>0$
it follows that $\cal{E}(sw_0,u_1)\to-\infty$ when $s\to\pm\infty$  and $u_1$ is fixed. Hence, choosing $u_0=sw_0$ for $s\in\R^\pm$ large enough, depending on $\|u_1\|_{H^0}$,  we get $(u_0,u_1)\in N$.

When $f$ does not vanish a.e. the proof repeats the arguments used in the previous case and hence it is omitted.
We just mention that we can directly
take $B\subseteq\Omega$ to be an open ball of radius $r>0$ and define $w_0\in C^\infty_c(B)$
by $w_0(x)=\eta_0(|x|/r)$.
\end{proof}

\begin{proof}[\bf Proof of Theorem~\ref{theorem1.5}]
By  Remarks~\ref{remark4.1}, \ref{remark4.3}, \ref{nuova6.1} and \ref{nuova6.2},
the statement is a particular case of Theorem~\ref{theorem6.1}.
\end{proof}

\subsection{Global existence}
In this subsection we shall deal with perturbation terms $f$ and $g$  which source part has at most linear growth at infinity, uniformly in the space
variable, or, roughly, it is dominated by the corresponding damping term. More precisely we shall make the following specific assumption:
\renewcommand{\labelenumi}{{(FGQP)}}
\begin{enumerate}
\item there are $p_1$ and $q_1$ verifying \eqref{V.2} and constants $C_{p_1},C_{q_1}\ge 0$ such that
$$
\mathfrak{F}(x,u)\le C_{p_1}\left[1+u^2+\gamma_0(x)|u|^{p_1}\right], \,\,\mathfrak{G}(y,u)\le  C_{q_1}\left[1+u^2+\delta_0(y)|u|^{q_1}\right]
$$
for a.a. $x\in\Omega$, $y\in\Gamma_1$ and all $u\in\R$.
\end{enumerate}
Since $\mathfrak{F}(\cdot,u)=\int_0^1 f(\cdot,su)u\,ds$ (and similarly $\frak{G}$), assumption
(FGQP) is a weak version of
 of the following one:
\renewcommand{\labelenumi}{{(FGQP)$'$}}
\begin{enumerate}
\item there are $p_1$ and $q_1$ verifying \eqref{V.2} and constants $C'_{p_1},C'_{q_1}\ge 0$ such that
$$f(x,u)u\le C'_{p_1}\left[|u|+u^2+\gamma_0(x)|u|^{p_1}\right], \quad g(y,u)u\le  C'_{q_1}\left[|u|+u^2+\delta_0(y)|u|^{q_1}\right]
$$
for a.a. $x\in\Omega$, $y\in\Gamma_1$ and all $u\in\R$.
\end{enumerate}
\begin{rem}\label{remark6.3ff}
Assumptions (FG1--2) and (FGQP)$'$ hold
 provided
 \begin{equation}\label{centerdotbis}
 f=f^0+f^1+f^2,\qquad g=g^0+g^1+g^2,
 \end{equation}
 where $f^i$, $g^i$ satisfy the following assumptions:
\renewcommand{\labelenumi}{{(\roman{enumi})}}
\begin{enumerate}
\item $f^0$ and $g^0$ are a.e. bounded and independent on $u$;
\item $f^1$ and $g^1$ satisfy (FG1--2) with exponents $p_1$ and $q_1$ satisfying \eqref{V.2}, and
\begin{enumerate}
\item when $p_1>2$ and $\essinf_\Omega\alpha=0$ there is a constant $\overline{c_{p_1}}\ge 0$ such
\begin{footnote}{that is $\varlimsup_{|u|\to\infty}|f_1(\cdot,u)|/|u|^{p_1-1}\le\overline{c_{p_1}}\alpha$ a.e. uniformly in $\Omega$}\end{footnote} that
$$|f^1(x,u)|\le \overline{c_{p_1}}\left[1+|u|+\alpha(x)|u|^{p_1-1}\right]$$
for a.a. $x\in\Omega$ and all $u\in\R$;
\item when $q_1>2$ and $\essinf_{\Gamma_1}\beta =0$ there is a constant $\overline{c_{q_1}}\ge 0$ such that
$$|g^1(y,u)|\le \overline{c_{q_1}}\left[1+|u|+\beta(y)|u|^{q_1-1}\right]$$
for a.a. $y\in\Gamma_1$ and all $u\in\R$;
\end{enumerate}
\item $f^2$ and  $g^2$ satisfy (FG1--2),
$f^2(x,u)u\le 0$ and $g^2(y,u)u\le 0$ for a.a. $x\in\Omega$, $y\in\Gamma_1$ and all $u\in\R$.
\end{enumerate}
Conversely any couple of functions $f$ and $g$ satisfying (FG1--2) and (FGQP)$'$ admits a decomposition of the form \eqref{centerdotbis}--(i--iii) with
$f^1$ and $g^1$ being source terms. Indeed one can set
$f^0=f(\cdot,0)$,
$$f^1(\cdot,u)=
\begin{cases}[f(\cdot,u)-f^0]^+ \,\,&\text{if $u>0$,}\\
\,\,\qquad 0\,\,&\text{if $u=0$,}\\
-[f(\cdot,u)-f^0]^- \,\,&\text{if $u<0$,}
\end{cases}\,\, \text{and}\,\,
f^2(\cdot,u)\begin{cases}-[f(\cdot,u)-f^0]^- \,\,&\text{if $u>0$,}\\
\,\,\qquad 0\,\,&\text{if $u=0$,}\\
[f(\cdot,u)-f^0]^+ \,\,&\text{if $u<0$,}
\end{cases}
$$
and define $g^0$, $g^1$, $g^2$ in the analogous way.
\end{rem}
\begin{rem}\label{remark6.4}When dealing with  problem \eqref{5} assumption
(FGQP) reduces to (V). The function $f\equiv f_2$ defined in \eqref{modellisupplementari} satisfies  (FGQP) provided
one among the following cases occurs:
\renewcommand{\labelenumi}{{(\roman{enumi})}}
\begin{enumerate}
\item $\gamma_1^+=\gamma_2^+\equiv 0$,
\item $\gamma_2^+\equiv0$, $\gamma_1^+\not\equiv0$,  $\widetilde{p}\le\max\{2,m\}$ and
$\gamma_1\le c'_1\alpha$ a.e. in $\Omega$ when $\widetilde{p}>2$
\item $\gamma_1^+\not=0$, $\gamma_2^+\not\equiv 0$, $p\le\max\{2,m\}$,
$\gamma_1\le c'_1\alpha$ when $\widetilde{p}>2$ and $\gamma_2\le c'_2\alpha$ when $p>2$, a.e. in $\Omega$,
\end{enumerate}
where $c_1', c_2'\ge 0$ denote suitable constants. The analogous cases (j--jjj) occurs when $g\equiv g_2$,
so that $(f_2,g_2)$ satisfies (FGQP) provided any combination between the cases (i--iii) and (j--jjj) occurs.
In particular then a damping term can be localized provided the corresponding source is equally localized.

Finally when $f\equiv f_3$ and $g\equiv g_3$ as in  \eqref{f3g3}, assumption (FGQP) holds provided
$f_0$ and $g_0$  satisfy assumption (V) (where we conventionally take $f_0\equiv 0$ when $\gamma\equiv 0$ and
$g_0\equiv 0$ when $\delta\equiv 0$), $\gamma\le \alpha$ when $p_1>2$ and
$\delta\le \beta$ when $q_1>2$.
\end{rem}

 Our global existence result is the following one.
 \begin{thm}\label{theorem6.2} Let (PQ1--3), (FG1--2), (FGQP) and \eqref{6} hold. Then, for any couple of data $(u_0,u_1)\in \cal{H}$  the unique maximal weak solution $u$ of \eqref{1} is global in time, that is $T_{\max}=\infty$. Consequently the semi--flow generated by problem \eqref{1} is a dynamical system in $H^1\times H^0$ and, when also (III) holds, in  $H^{1,\rho,\theta}_{\alpha,\beta}\times H^0$ for $(\rho,\theta)$ verifying \eqref{special}.
\end{thm}
\begin{proof} We suppose by contradiction that $T_{\max}<\infty$, so by Theorem~\ref{theorem4.1} we have
\begin{equation}\label{CCCP1234}
\lim\limits_{t\to T^-_{\text{max}}}
\|u(t)\|_{H^1}+\|u'(t)\|_{H^0}=\infty.
\end{equation}
We introduce the functional $\cal{I}:H^1\to\R^+_0$ given by
\begin{equation}\label{6.106bis}
\cal{I}(v)=C_{p_1}\int_\Omega \alpha|v|^{p_1}+C_{q_1}\int_{\Gamma_1}\beta|v|^{q_1}.
\end{equation}
Since the functions $p_1C_{p_1}(x)|u|^{p_1-2}u$ and $q_1C_{q_1}\beta(x)|u|^{q_1-2}u$ satisfy assumption (FG1), we see as in
subsection~\ref{preliminarienergia} that
$\cal{I}\in C^1(H^1)$, with Fr\`{e}chet derivative being given by the couple of Nemitskii operators associated with them. Hence by Lemma~\ref{chainrule} we have
$\cal{I}\cdot u\in C^1([0,T_{\max}))$  and, for all $t\in[0,T_{\max})$,
\begin{equation}\label{potentialderivative1}
\cal{I}(u(t))-\cal{I}(u_0)=\int_0^t\left[\int_\Omega p_1C_{p_1}\alpha|u|^{p_1-2}uu_t+\int_{\Gamma_1} q_1 C_{q_1}|u|^{q_1-2}u(u_{|\Gamma})_t\right].
\end{equation}
We introduce an auxiliary function associated to $u$ by
\begin{equation}\label{6.105}
\Upsilon(t)=\tfrac 12 \|u'(t)\|_{H^0}^2+\tfrac 12 \|u(t)\|_{H^1}^2+\cal{I}(u(t)),\qquad\text{for all $t\in [0,T_{\max})$}.
\end{equation}
By \eqref{nnn6.6} and \eqref{6.105} we have
\begin{equation}\label{6.106}
\Upsilon(t)=\cal{E}_u(0)+\tfrac 12 \|u(t)\|_{H^0}^2+J(u(t))+\cal{I}(u(t))-\int_0^t\langle B(u'),u'\rangle_W.
\end{equation}
By \eqref{6.106bis} and assumption (FGQP)  we get
\begin{equation}\label{new}
J(v)\le \left[C_{p_1}|\Omega|+C_{q_1}\sigma(\Gamma)\right]\, \left(1+\|v\|_{H^0}^2\right)+\cal{I}(v)\qquad\text{for all $v\in H^1$.}
\end{equation}
By \eqref{6.106}-- \eqref{new} we thus obtain
\begin{equation}\label{new2}
\Upsilon(t)\le \cal{E}_u(0)+k_1+k_1\|u(t)\|_{H^0}^2+2\cal{I}(u(t))-\int_0^t\langle B(u'),u'\rangle_W,
\end{equation}
where $k_1=C_{p_1}|\Omega|+C_{q_1}\sigma(\Gamma)+1/2$. Consequently, by \eqref{potentialderivative1},
\begin{equation}\label{6.107}
\begin{aligned}
\Upsilon(t)\le & k_2+\int_0^t \left[2k_1(u',u)_{H^0}-\langle B(u'),u'\rangle_W\phantom{\int}\right.\\+& \left.2p_1C_{p_1}\int_\Omega \alpha|u|^{p_1-2}uu_t+2q_1C_{q_1}\int_{\Gamma_1}\beta|u|^{q_1-2}u(u_{|\Gamma})_t\right],
\end{aligned}
\end{equation}
where $k_2=\cal{E}_u(0)+2\cal{I}(u_0)+k_1(1+\|u_0\|_{H^0}^2)$.
Consequently, by assumption (PQ3),  Cauchy--Schwartz  and Young inequalities, we get the preliminary estimate
\begin{equation}\label{6.108}
\begin{aligned}
\Upsilon(t)\le &k_2+\int_0^t\left[- c_m'\|[u_t]_\alpha\|_{m,\alpha}^m-c_\mu'\|[(u_\Gamma)_t]_\beta\|_{\mu,\beta}^\mu+k_1\|u'\|_{H^0}^2+k_1 \|u\|_{H^0}^2\right.\\+&\left.2p_1C_{p_1}\int_\Omega \alpha|u_t| |u|^{p-1}|u_t|+2q_1C_{q_1}\int_{\Gamma_1}\beta |u|^{q-1}|(u_{|\Gamma})_t|
|(u_{|\Gamma})_t|\right]
\end{aligned}
\end{equation}
for all $t\in[0,T_{\max})$.
We now estimate, a.e. in $[0,T_{\max})$, the last four integrands  in the right--hand side of \eqref{6.108}.
By \eqref{6.105} we get
\begin{equation}\label{6.109}
k_1\|u'\|_{H^0}^2\le 2k_1 \Upsilon.
\end{equation}
Moreover, by the embedding $H^1(\Omega;\Gamma)\hookrightarrow L^2(\Omega)\times L^2(\Gamma)$, there is a positive constant $k_3$,
depending only on $\Omega$, such that
\begin{equation}\label{6.110}
 \|u\|_{H^0}^2\le k_3 \|u\|_{H^1}^2.
\end{equation}
Consequently, by \eqref{6.105}, there is a positive constant $k_4$,  depending only on $\Omega$, such that
\begin{equation}\label{6.111}
k_1\|u\|_{H^0}^2\le  k_4\Upsilon.
\end{equation}
To estimate the addendum $2p_1C_{p_1}\int_\Omega \alpha|u|^{p_1-1}|u_t|$ we now distinguish between the cases $p_1=2$ and $p_1>2$.
When $p_1=2$,  by \eqref{6.105}, \eqref{6.111} and Young inequality,
\begin{equation}\label{6.112}
2p_1C_{p_1}\int_\Omega \alpha|u|^{p-1}|u_t|\le p_1C_{p_1}\|\alpha\|_\infty(\|u\|_{H^0}^2+\|u'\|_{H^0}^2)\le k_5\Upsilon,
\end{equation}
where $k_5=2p_1C_{p_1}\|\alpha\|_\infty (1+k_3)$.

When $p_1>2$, for any $\eps\in (0,1]$ to be fixed later, by  weighted Young inequality
\begin{equation}\label{6.114}
2p_1C_{p_1}\int_\Omega \alpha|u|^{p_1-1}|u_t|\le 2(p_1-1)C_{p_1}\eps^{1-p_1'}\int_\Omega \alpha|u|^{p_1}+2\eps C_{p_1}\int_\Omega \alpha |u_t|^{p_1}.
\end{equation}
By \eqref{6.105} we have
\begin{equation}\label{6.114bis}
2(p_1-1)C_{p_1}\eps^{1-p_1'}\int_\Omega \alpha|u|^{p_1}\le 2(p_1-1)\eps^{1-p_1'}\Upsilon.
\end{equation}
Moreover by \eqref{V.2} we have $p_1\le m=\overline{m}$ and consequently  $|u_t|^{p_1}\le 1+|u_t|^m$ a.e. in $\Omega$, which yields
\begin{equation}\label{6.115}
\int_\Omega \alpha |u_t|^{p_1}\le \int_\Omega\alpha+\int_\Omega \alpha |u_t|^m\le \|\alpha\|_\infty |\Omega|+\|[u_t]_\alpha\|_{m,\alpha}^m.
\end{equation}
Plugging \eqref{6.114bis} and \eqref{6.115} in \eqref{6.114} we get, as $\eps\le1$,
\begin{equation}\label{6.116}
2p_1C_{p_1}\int_\Omega \alpha|u|^{p_1-1}|u_t|\le k_6\left(\eps^{1-p_1'}\Upsilon+\eps \|[u_t]_\alpha\|_{m,\alpha}^m+1\right)
\end{equation}
where $k_6$ is a positive constant independent on $\eps$.

Comparing \eqref{6.110} and \eqref{6.116} we get that for $p\ge 2$ we have
\begin{equation}\label{6.117}
2p_1C_{p_1}\int_\Omega \alpha|u|^{p_1-1}|u_t|\le k_7\left[(1+\eps^{1-p_1'})\Upsilon+\eps \|[u_t]_\alpha\|_{m,\alpha}^m+1\right]
\end{equation}
where $k_7$ is a positive constant independent on $\eps$.

We estimate the last integrand in the right--hand side of
\eqref{6.108} by transposing from $\Omega$ to $\Gamma_1$ the arguments used to get \eqref{6.117}. At the end we get
\begin{equation}\label{6.118}
2q_1C_{q_1}\int_{\Gamma_1} \beta|u|^{q_1-1}|(u_{|\Gamma})_t|\le k_8\left[(1+\eps^{1-q_1'})\Upsilon+\eps \|[(u_{|\Gamma})_t]_\beta\|_{\mu,\beta,\Gamma_1}^\mu+1\right]
\end{equation}
where $k_8$ is a positive constant independent on $\eps$.

Plugging estimates \eqref{6.109}, \eqref{6.111}, \eqref{6.117} and \eqref{6.118} into \eqref{6.108} we get
 \begin{multline}\label{6.119}
   \Upsilon(t)\le k_2+\int_0^t \left[(k_7\eps-c_m')\|[u_t]_\alpha\|_{m,\alpha}^m+(k_8\eps-c_\mu')\|[(u_\Gamma)_t]_\beta\|_{\mu,\beta}^\mu\right]\\+
   k_9\int_0^t\left[(1+\eps^{1-p_1'}+\eps^{1-q_1'})\Upsilon+1\right]\qquad\text{for all $t\in [0,T_{\max})$.}
 \end{multline}
 where $k_9$ is a positive constant independent on $\eps$.  Fixing $\eps=\eps_1$, where $\eps_1=\min\{1, c_m'/k_7, c_\mu'/k_8\}$, and setting $k_{10}=k_9(1+\eps_1^{1-p'}+\eps_1^{1-q'})$, the estimate
 \eqref{6.119} reads as
 $$\Upsilon_u(t)\le \int_0^t k_{10}\left(1+\Upsilon\right)\qquad\text{for all $t\in [0,T_{\max})$.}$$
 Then, by Gronwall Lemma (see \cite[Lemma 4.2, p. 179]{showalter}),  $\Upsilon$ is bounded in $[0,T_{\max})$, hence
by \eqref{6.105} $\|u\|_{H^1}$ and $\|u'\|_{H^0}$ are bounded in $[0,T_{\max})$, contradicting \eqref{CCCP1234}.
\end{proof}

\begin{proof}[\bf Proof of Theorem~\ref{theorem1.6}]
It follows by Remarks~\ref{remark4.1}, \ref{remark4.3},  \ref{remark6.4}
and  Theorem~\ref{theorem6.2}.
\end{proof}

\appendix
\section{On the Cauchy problem for locally Lipschitz perturbations of maximal monotone
operators}\label{appendixA} The aim of this section is to complete
the statement of the local existence--uniqueness result in
\cite{chueshovellerlasiecka} concerning locally Lipschitz
perturbations of maximal monotone operators. We first recall it,
changing the notation to fit with \eqref{C3}.

Let $\cal{A_1}:D(\cal{A_1})\subset \cal{H}\to \cal{H}$ be  a maximal
monotone operator on the (real) Hilbert space $\cal{H}$,
$(\cdot,\cdot)_{\cal{H}}$ and $\|\cdot\|_{\cal{H}}$ respectively
denoting its scalar product and norm. Moreover let
$\cal{F_1}:\cal{H}\to\cal{H}$ be a locally Lipschitz map, i.e. for
any $R\ge0$ there is $L(R)\ge 0$ such that
\begin{equation}\label{A1}
\|\cal{F_1}(U)-\cal{F_1}(V)\|_{\cal{H}}\le L(R)\,
\|U-V\|_{\cal{H}}\qquad\text{provided
$\|U\|_{\cal{H}},\|V\|_{\cal{H}}\le R$
.}
\end{equation}
Given any  $h\in L^1_{\text{loc}}([0,\infty);\cal{H})$, we are concerned with the Cauchy problem
\begin{equation}\label{C4}
 U'+\cal{A_1}(U)+\cal{F_1}(U)\ni h\qquad\text{in $\cal{H}$,}\qquad
U(0)=U_0\in\cal{H},
\end{equation}

\begin{thm}[{\cite[Theorem~7.2]{chueshovellerlasiecka}}]\label{theoremA1}
Suppose that $\cal{A_1}$ is a maximal monotone operator in $\cal{H}$
with $0\in \cal{A_1}(0)$ and $\cal{F_1}$ satisfies \eqref{A1}. Then
for any $U_0\in D(\cal{A_1})$ and $h\in
W^{1,1}_{\text{loc}}([0,\infty);\cal{H})$ problem \eqref{C4} has a
unique maximal strong solution $U$ in the interval
$[0,T_{\text{max}})$. Moreover  for any $U_0\in
\overline{D(\cal{A_1})}$ and  $h\in
L^1_{\text{loc}}([0,\infty);\cal{H})$ problem \eqref{C4} has a
unique maximal generalized solution in $[0,T_{\text{max}})$. In both
cases we have $\lim_{t\to
T_{\text{max}}^-}\|U(t)\|_{\cal{H}}=\infty$ provided
$T_{\text{max}}<\infty$.
\end{thm}

\begin{rem}\label{theoremA1plus}It is well--known that $T_{\text{max}}=\infty$ for any datum
$U_0$ when $\cal{F_1}$ is globally Lipschitz, i.e. \eqref{A1} holds
with $R=\infty$, see \cite[Theorems 4.1 and
4.1A]{showalter}.\end{rem}

The aim of this section is to point out the continuous dependence of
$U$ from $U_0$ and $h$, which is a standard fact when $\cal{F}$ is
globally Lipschitz, since the author did not find a precise
reference for this fact when $\cal{F}$ is only locally Lipschitz.
We shall denote by
$U=U(U_0,h)$  the maximal generalized solution corresponding to
$U_0$ and $h$ and by $T_{\text{max}}=T_{\text{max}}(U_0,h)$ the
right--endpoint of its domain.

\begin{thm}\label{theoremA2}
Under the assumptions of Theorem \ref{theoremA1}, given $U_0,
(U_{0n})_n$ in $\overline{D(\cal{A_1})}$ such that $U_{0n}\to U_0$
in $\cal{H}$ and  $h_n\to h$ in
$L^1_{\text{loc}}([0,\infty);\cal{H})$, we have
\renewcommand{\labelenumi}{{\roman{enumi})}}
\begin{enumerate}
\item $T_{\text{max}}(U_0,h)\le\liminf_n
T_{\text{max}}(U_{0n},h_n)$, and
\item $U(U_{0n},h_n)\to U(U_0,h)$ in $C([0,T^*];\cal{H})$ for any
$T^*\in (0,T_{\text{max}}(U_0,h))$.
\end{enumerate}
\end{thm}
\begin{proof}
The proof is based on the arguments of the proof of Theorem
\ref{theoremA1}, so we are going to recall some details of it. The
solution $U$ is found as the solution of a modified version of
\eqref{C4}, where $\cal{F_1}$ is replaced by a globally Lipschitz
map $\cal{F_1^R}$ given by
$$\cal{F_1^R}(U)=\begin{cases}
\cal{F_1}(U),\qquad &\text{if $\|U\|_{\cal{H}}\le R$,}\\
\cal{F_1}\left(\frac {RU}{\|U\|_{\cal{H}}}\right),\qquad &\text{if
$\|U\|_{\cal{H}}\ge R$,}
\end{cases}$$
where $R$ is chosen so that $\|U_0\|_{\cal{H}}<R$. Then it is proved
that $\cal{F_1^R}$ is globally Lipschitz, with Lipschitz constant
$L(R)$, and that $\cal{A_1^R}=\cal{A_1}+L(R)I+\cal{F_1^R}$ is
maximal monotone, hence by \cite[Theorem 4.1]{showalter} the Cauchy
problem
\begin{equation}\label{CR}
\begin{cases} U'+\cal{A_1}(U)+\cal{F_1^R}(U)=U'+\cal{A_1^R}(U)-L(R)U\ni h\qquad\text{in $\cal{H}$,}\\
U(0)=U_0\in\cal{H},
\end{cases}
\end{equation}
has a unique generalized solution $U$ in $[0,\infty)$ provided $h\in
L^1_{\text{loc}}([0,\infty);\cal{H})$ and
$U_0\in\overline{D(\cal{A_1})}$, which is actually strong provided
$h\in W^{1,1}_{\text{loc}}([0,\infty);\cal{H})$ and $U_0\in
D(\cal{A_1})$. The existence of a solution of \eqref{C4} in some
interval $[0,t^*]$ then follows by choosing $t^*$ (small),
depending on $R$ and $h$, such that
\begin{equation}\label{A2}
\|U(t)\|_{\cal{H}}\le R\qquad\text{for all $t\in [0,t^*]$.}
\end{equation}

Our first claim is that, choosing $R=2(1+\|U_0\|_{\cal{H}})$, there
is $T_1:[0,\infty)^2\to (0,1]$, decreasing in both
variables, such that $t^*=T_1(\|U_0\|_{\cal H},
\|h\|_{L^1(0,1;\cal{H})})$ verifies \eqref{A2}, so
\begin{equation}\label{A3}
\|U(t)\|_{C([0,t^*];\cal{H})}\le 2(1+\|U_0\|_{\cal{H}}).
\end{equation}
To prove our claim we note, by the same arguments used in the proof
of Theorem \ref{theoremA1},
that when $U_0\in D(\cal{A_1})$, since $0\in\cal{A_1}(0)$,
$\cal{A_1^R}$ is monotone and $\cal{F_1^R}(0)=\cal{F_1}(0)$, we have
\begin{equation}\label{A4}
\frac d{dt}\left(\frac 12\|U(t)\|_{\cal{H}}^2\right)\le L(R)\,
\|U(t)\|_{\cal{H}}^2+
\left(\|h(t)\|_{\cal{H}}+\|\cal{F_1}(0)\|_{\cal{H}}\right)\,\|U(t)\|_{\cal{H}}
\end{equation}
for all  $t\in [0,\infty)$, hence by Gronwall Lemma (see \cite[Lemma
4.1, p. 179]{showalter})
\begin{equation}\label{A5}
\|U(t)\|_{\cal{H}}\le e^{L(R)t}\left[\|U_0\|_{\cal{H}}+\int_0^t
e^{-L(R)s} \Big(\|h(s)\|_{\cal{H}}+\|\cal{F_1}(0)\|\Big)\,ds \right]
\end{equation}
for all $t\in [0,\infty)$, which by \cite[(4.12), p. 183]{showalter}
holds for all $U_0\in \overline{D(\cal{A_1})}$. By \eqref{A5} then
\eqref{A2} holds provided
$$t^*\le 1,\quad e^{L(R)t^*}\le 2,\quad\text{and}\quad
e^{L(R)t^*}\left(\|\cal{F_1}(0)\|_{\cal{H}}+\|h\|_{L^1(0,1;\cal{H})}\right)\le
2.$$ Since $L(R)$ in \eqref{A1} can be assumed, without restriction,
to be increasing, our claim then follows by setting (where
$\log{(2/0)}$ and $(\log 2)/0$ stand for $\infty$)
$$
T_1=\min\left\{1, \frac {\log 2}{L(2+2\,\|U_0\|_{\cal{H}})},\frac
1{L(2+2\,\|U_0\|_{\cal{H}})}\log{\frac
2{\|h\|_{L^1(0,1;\cal{H})}+\|\cal{F_1}(0)\|_{\cal{H}}}}\right\} .$$

From our first claim then it follows the existence of a maximal
generalized solution, as well as its uniqueness, and clearly we have
\begin{equation}\label{A7}
T_1(\|U_0\|_{\cal H}, \|h\|_{L^1(0,1;\cal{H}})<T_{\text{max}}(U_0,h)
\end{equation}
for all $U_0\in\overline{D(\cal{A_1})}$ and $h\in
L^1_{\text{loc}}([0,\infty);\cal{H})$.

We now claim  that for any $U_0,
V_0\in\overline{D(\cal{A_1})}$, $h,k\in
L^1_{\text{loc}}([0,\infty);\cal{H})$,  $M,H$ such that
\begin{equation}\label{A8}
\max\{\|U_0\|_{\cal{H}},\|V_0\|_{\cal{H}}\}\le M,
\quad\text{and}\quad
\max\{\|h\|_{L^1(0,1;\cal{H})},\|k\|_{L^1(0,1;\cal{H})}\}\le H
\end{equation}
we have, denoting $U=U(U_0,h)$ and $V=U(V_0,k)$,
\begin{equation}\label{A9}
\|U(t)-V(t)\|_{\cal{H}}\le
e^{L(2M+2)t}\left(\|U_0-V_0\|_{\cal{H}}+\|h-k\|_{L^1(0,1;\cal{H})}\right)
\end{equation}
for all $t\in
[0,T_1(M,H)]$. To prove our claim we note that, being $T_1$ decreasing in both
variables, by \eqref{A8} we have
\begin{equation}\label{A9bis}
T_1(M,H)\le \min\{T_1(\|U_0\|_{\cal H}, \|h\|_{L^1(0,1;\cal{H}}),
T_1(\|V_0\|_{\cal H}, \|k\|_{L^1(0,1;\cal{H}})\}.
\end{equation}
Hence, by \eqref{A3} and \eqref{A8}, $U$ and $V$ solve in
$[0,T_1(M,H)]$ the equation in \eqref{CR} when $R=2(1+M)$. Then,
first considering data $U_0, V_0\in D(\cal{A_1})$ and using the
monotonicity of $\cal{A_1^R}$ we get
\begin{equation}\label{A10}
\frac d{dt}\left(\frac 12\|U-V\|_{\cal{H}}^2\right)\le L(2M+2)\,
\|U-V\|_{\cal{H}}^2+ \|h-k\|_{\cal{H}}\,\|U-V\|_{\cal{H}},
\end{equation}
in $[0,T_1(M,H)]$, hence, by using Gronwall Lemma again
$$
\|U(t)-V(t)\|_{\cal{H}}\le
e^{L(2M+2)t}\left(\|U_0-V_0\|_{\cal{H}}+\int_0^t e^{-L(2M+2)s}
\|h(s)-k(s)\|_{\cal{H}}\,ds \right),
$$
from which, as $T_1\le 1$, \eqref{A9} follows. By
\cite[(4.12)]{showalter} the estimate \eqref{A9} hold for $U_0,
V_0\in \overline{D(\cal{A_1})}$, concluding the proof of our second
claim.

Now let $U_0$, $(U_{0n})_n$, $h$, $h_n$ and $T^*$ as in the
statement, and denote for shortness $U=U(U_0,h)$,
$U_n=U(U_{0n},h_n)$, $T_{\text{max}}=T_{\text{max}}(U_0,h)$ and
$T_{\text{max}}^n=T_{\text{max}}(U_{0n},h_n)$. We set
$M(T^*)=\|U\|_{C([0,T^*];\cal{H})}$,
$H(T^*)=\|h\|_{L^1(0,T^*+1,\cal{H})}$,
$T_2(T^*)=T_1(1+M(T^*),1+H(T^*))$ and $\kappa(T^*)\in\N_0$ such that
\begin{equation}\label{A12}
\kappa(T^*) T_2(T^*)< T^*\le [\kappa(T^*)+1] T_2(T^*)
\end{equation}
i.e. $\kappa(T^*)=\min\left\{\kappa\in\N_0: T^*/T_2(T^*)\le
\kappa+1\right\}$.
 By \eqref{A8}, since $U_{0n }\to U_0$ in $\cal{H}$ and $h_n\to h$
in $L^1(0,1;\cal{H})$, there is $n_1(T^*)\in N$ such that
$\|U_{0n}\|_{\cal{H}}\le M(T^*)+1$ and
$\|h_n\|_{L^1(0,1;\cal{H})}\le H(T^*)+1$ for $n\ge n_1(T^*)$. By the
monotonicity of $T_1$ and \eqref{A8} then we have $$T_2(T^*)\le
T_1(\|U_0\|_{\cal{H}},\|h\|_{L^1(0,1;\cal{H})})\quad\text{and}\quad
T_2(T^*)\le T_1(\|U_{0n}\|_{\cal{H}},\|h_n\|_{L^1(0,1;\cal{H})})$$
for $n\ge n_1(T^*)$. By maximality it follows that
\begin{equation}\label{A13}
T_2(T^*)<T_{\text{max}},\quad\text{and}\quad
T_2(T^*)<T_{\text{max}}^n\quad\text{for $n\ge n_1(T^*)$.}
\end{equation}
By our second claim moreover  we have
$$\|U_n-U\|_{C([0,T_2(T^*)];\cal{H}})\le
e^{L(2M(T^*)+4)T_2(T^*)}\left(\|U_{0n}-U_0\|_{\cal{H}}+\|h_n-h\|_{L^1(0,1;\cal{H})}\right),
$$
from which  $U_n\to U$ in $C([0,T_2(T^*)];\cal{H})$, so that
\begin{equation}\label{A14}
U_n(T_2(T^*))\to U(T_2(T^*))\quad \text{and}\quad h_n\to h\quad
\text{in $L^1(T_2(T^*),T_2(T^*)+1;\cal{H})$.}
\end{equation}
When $T^*\le
T_2(T^*)$, or equivalently $\kappa(T^*)=0$, the proof of ii) is
complete, and by \eqref{A13} we have
\begin{equation}\label{A15}
T^*<T_{\text{max}}^n\qquad\text{for $n\ge n_1(T^*)$.}
\end{equation}
When $T^*> T_2(T^*)$, or equivalently $\kappa(T^*)\ge 1$, we simply
repeat previous arguments $\kappa(T^*)$ times, having \eqref{A14} as
the starting point. In this way we get that $U_n\to U$ in
$C([0,[\kappa(T^*)+1]T_2(T^*)];\cal{H})$ and $T^*<T_{\text{max}}^n$
for $n\ge n_{\kappa(T^*)+1}(T^*)$. By \eqref{A12} the proof of ii)
is then completed, while i) follows, since $T^*\in
(0,T_{\text{max}})$ is arbitrary, also using \eqref{A15}, concluding
the proof.
\end{proof}

\section{On the Laplace--Beltrami operator}
\label{appendixB} This section is devoted to prove  the following result
\begin{lem}\label{propositionA1}
Let $M$ be a $C^2$ compact manifold equipped with a $C^1$ Riemannian
metric $(\cdot,\cdot)_M$. Then  $-\Delta_M+I$ is a topological and
algebraic isomorphism between $W^{s+1,\rho}(M)$ and
$W^{s-1,\rho}(M)$ for any $s\in [-1,1]$ and $1<\rho<\infty$.
\end{lem}
This fact is well--known when $M$ is smooth (see for example
\cite{mitrea}, \cite{strichartz} and \cite[p.28]{taylor3}). A
proof  is given in  the sequel for the sake of completeness.
 Due to the linear nature of the problem
it is convenient to prove it for Sobolev spaces of complex--valued
distributions, since complex interpolation arguments are available.
The real case then trivially follows. All the preparatory material
in the main body of the paper still hold provided one proceeds as follows.
The tangent bundle $T(M)$ is complexified by setting $T(M)^{\C}:=\bigcup_{x\in
M} \{x\}\times T_x(M)^{\C}$, where $T_x(M)^{\C}\simeq
T_x(M)+iT_x(M)$ stands for the complexification of $T_x(M)$ (see \cite{roman}). By $\text{Re }v$ and $\text{Im }v$
we shall respectively  denote the real and imaginary part of $v\in
T(M)^{\C}$. Moreover $(\cdot,\cdot)_M$ is uniquely extended  as an hermitian form
on $T(M)^{\C}$. Finally  $v$ is replaced by $\overline{v}$ in the first integral in
\eqref{complessifica} and \eqref{26} and in the last one in
\eqref{3.5}--\eqref{3.6}.

 By repeating the arguments in \cite[pp.~38-42]{lionsmagenes1} and
using the well--known interpolations properties of  Sobolev
spaces in $\R^{n}$ (see \cite{triebel}) one easily proves that
\begin{equation}\label{5.3}W^{s,\rho}(\Gamma)=
\begin{cases}
(W^{s_0,\rho}(M),W^{s_1,\rho}(M))_{\theta,\rho}\qquad&\text{if
$s\not\in\Z$},\\
[W^{s_0,\rho}(M),W^{s_1,\rho}(M)]_\theta\qquad&\text{if $s\in\Z$}
\end{cases}
\end{equation}
where $s_0,s_1\in\Z$, $s=\theta s_0+(1-\theta)s_1$, $\theta\in
(0,1)$, $-2\le s_0\le s_1\le 2$, and $(\cdot,\cdot)_{\theta,\rho}$,
$[\cdot,\cdot]_\theta$ respectively denote the real and complex
interpolator functors (see \cite{berglofstrom}).
\begin{lem}\label{lemmaB1} Let $M$ be a $C^2$ compact manifold equipped with a $C^1$ Riemannian
metric $(\cdot,\cdot)_M$ and $1<\rho<\infty$. Then
\begin{enumerate}
\renewcommand{\labelenumi}{{\roman{enumi})}}
\item for any $u\in W^{2,1}(M)$ such that $\Delta_Mu\in L^\rho(M)$
we have $u\in W^{2,\rho}(M)$;
\item
there is $C=C(\rho, (\cdot,\cdot)_M)>0$ such that
\begin{equation}\label{ADN}
\|u\|_{W^{2,\rho}(M)}\le C\left( \|\Delta_M
u\|_{L^\rho(M)}+\|u\|_{L^\rho(M)}\right) \qquad\text{for all $u\in
W^{2,\rho}(M)$.}
\end{equation}
\end{enumerate}
\end{lem}
\begin{proof} We use the
standard localization technique. Since $M$ is compact it posses a
finite atlas $\cal{U}=\{(U_i,\phi_i),i=1,\ldots,r\}$, with
$\phi_i(U_i)=B_1$, where $B_R$ denotes the open ball in $\R^n$,
$n=\text{dim} M$, of
radius $R>0$ centered at the origin. By \cite[Theorem 4.1, p.
57]{sternberg} there is a $C^2$ partition of the unity
$\cal{T}=\{\theta_i,i=1,\ldots,r\}$ subordinate to it, i.e.
$\theta_i\in C^2(M)$, $0\le \theta_i\le 1$, $\text{supp }
\theta_i\subset\subset U_i$ for $i=1,\ldots,r$,
$\sum_{i=1}^r\theta_i=1$ on $M$. In the sequel we shall denote by
$C_1,C_2,\ldots$ positive constants depending on $\rho$,
$(\cdot,\cdot)_M$, $\cal{U}$ and $\cal{T}$.

We first claim that if $u\in W^{2,s}(M)$ for some $1<s<\rho$ such
that $\rho\le sn/(n-s)$ if $s<n$,  and  $\Delta_Mu\in L^\rho(M)$,
then $u\in W^{2,\rho}(M)$. To prove our claim we fix
$\bar{i}=1,\ldots,r$ and we denote $\theta=\theta_{\bar{i}}$, $v=u
\theta$,  $\widetilde{u}=u\cdot \phi_{\bar{i}}^{-1}\in
W^{2,s}(B_1)$, $\widetilde{\theta}=\theta\cdot
\phi_{\bar{i}}^{-1}\in C^2(B_1)$,
$\widetilde{v}=\widetilde{u}\widetilde{\theta}$. Now set $R\in
(0,1)$ such that $\text{supp }\widetilde{\theta}\subset\subset B_R$,
so that $\widetilde{v}\in W^{2,s}(B_R)$ and $\text{supp
}\widetilde{v}\subset\subset B_R$. By the expression of $\Delta_M$ in local coordinates we have
\begin{equation}\label{Leibnitz}
\Delta_M
v=\theta\Delta_M u+2(\nabla_M \theta,\nabla_Mu)_M+u\Delta_M\theta.
\end{equation}
Since, by Sobolev embedding theorem, we have $u, |\nabla_Mu|_M\in
L^\rho(M)$, we get that $\Delta_M v\in L^\rho(M)$. Using its expression  in local coordinates
again the operator $-\Delta_M$ is expressed, in local coordinates,
by $L^2+L^1$, where $L^2=-\partial_i(g^{ij}\partial_j)$ and
$L^1=-\frac 12 g^{-1}(\partial_j g) g^{ij}\partial_j$, hence by
\eqref{metricalimitata} and Sobolev embedding theorem we get $L^2
\widetilde{v}\in L^\rho(B_R)$. Since, also by
\eqref{metricalimitata}, $L^2$ is a linear uniformly elliptic
operator, in the divergence form, with coefficients in
$C^1(\overline{B_R})$, we can apply \cite[Lemma 2.4.1.4, p.
114]{grisvard} to the homogeneous Dirichlet problem in $B_R$
to conclude that $\widetilde{v}\in W^{2,\rho}(B_1)$, so that
$v\in W^{2\,\rho}(M)$. Summing up for $\bar{i}=1,\ldots,r$ we then get  $u\in W^{2,\rho}(M)$,
proving our claim. By a reiterate application of previous claim we get i).

To prove ii) we note that, by \cite[Theorem 15.2]{ADN}
\begin{footnote}{or, with a slight variant, \cite[Theorem 2.3.3.2, p. 106]{grisvard}}\end{footnote}
 we get
\begin{equation}\label{B3}
\|\tilde{v}\| _{W^{2,\rho}(B_1)}=\|\tilde{v}\| _{W^{2,\rho}(B_R)}\le
C_1\left(\|\Delta_M\tilde{v}\|_{L^{\rho}(B_R)}+\|\tilde{v}\|_{L^{\rho}(B_1)}\right),
\end{equation}
and then, by \eqref{Leibnitz},
$$ \|\tilde{v}\|
_{W^{2,\rho}(B_1)}\le
C_2\left(\|\Delta_M\tilde{u}\|_{L^{\rho}(B_1)}+\|\tilde{u}\|_{W^{1,\rho}(B_1)}\right),
$$
which, by \eqref{metricalimitata}, yields
$$ \|\tilde{v}\|
_{W^{2,\rho}(B_1)}\le
C_3\left(\|\Delta_Mu\|_{L^{\rho}(M)}+\|u\|_{W^{1,\rho}(M)}\right).
$$
Summing
up for $\bar{i}=1,\ldots,r$ we get
$$ \|u\|
_{W^{2,\rho}(M)}\le
C_4\left(\|\Delta_Mu\|_{L^{\rho}(M)}+\|u\|_{W^{1,\rho}(M)}\right).
$$
Since by \eqref{5.3} we have
$W^{1,\rho}(M)=[W^{2,\rho}(M),L^{\rho}(M)]_{1/2}$, by interpolation
\cite[p. 21]{triebel}) and weighted Young inequalities we get
\begin{equation}\label{B6}
\|u\|_{W^{2,\rho}(M)}\le C_5\left( \|\Delta_M
u\|_{L^\rho(M)}+\|u\|_{L^\rho(M)}\right) \qquad\text{for all $u\in
W^{2,\rho}(M)$}.
\end{equation}
We finally set
$C=\sup\left\{\frac{\|u\|_{W^{2,\rho}(M)}}{
\|\Delta_M u\|_{L^\rho(M)}+\|u\|_{L^\rho(M)}}, \, u\in
W^{2,\rho}(M)\setminus\{0\}\right\}$,
nothing that $C<\infty$ by \eqref{B6} and $C$ is trivially independent on
$\cal{U}$ and $\cal{T}$.
 \end{proof}
\begin{proof}[Proof of Lemma~\ref{propositionA1}] We denote $A_{s,\rho}=-\Delta_M+I:W^{s+1,\rho}(M)\to W^{s-1,\rho}(M)$.
By \eqref{complessifica}, \eqref{5.1} and  \eqref{5.4} we have
$\langle A_{0,2}\bar{u},v\rangle_{H^1(M)}=(v,u)_{H^1}$ for all $u,v\in H^1(M)$,
so by Riesz--Fr\'echet theorem,
$A_{0,2}$ is  an isomorphism.

We now consider the case $s=1$, starting with $\rho=2$. By previous remark, for all $h\in L^2(M)$ there
is a unique $u\in H^1(M)$ such that $-\Delta_Mu+u=h$. Since
\cite[Theorem 1.3, p. 304--306]{taylor} trivially extends to $C^2$ manifolds we get $u\in H^2(M)$,
hence also $A_{1,2}$ is an isomorphism.

We now consider $\rho\ge 2$. Given  $h\in L^\rho(M)$ there is a unique $u\in H^2(M)$ such that
$-\Delta_Mu+u=h$, and by Lemma \ref{lemmaB1} -- i) we have $u\in W^{2,\rho}(M)$, hence
$A_{1,\rho}$ is an isomorphism when $\rho\ge 2$.

We now take  $1<\rho<2$ and  we consider $A_{1,\rho}$ as an
unbounded linear operator in $L^\rho(M)$ with domain
$W^{2,\rho}(M)$. Being  bounded from  $W^{2,\rho}(M)$ to $L^\rho(M)$ by
Lemma~\ref{lemmaB1} --ii), it is a closed operator.
We now claim, as in\cite{grisvard,pazy}, that
$-\Delta_M$ is accretive in $L^\rho(M)$, i.e.
\begin{equation}\label{B9}
\text{Re } \int_M-\Delta_M u\,\, |u|^{\rho-2}\bar{u}\ge 0\qquad\text{for all $u\in W^{2,\rho}(M)$.}
\end{equation}
We first take $u\in C^2(M)$ and set $u_\eps^{\star}=(|u|^2+\eps)^{(\rho-2)/2}u$ for $\eps>0$.
Hence
\begin{multline*}
(\nabla_Mu,\nabla_M u_\eps^{\star})_M=(|u|^2+\eps)^{(\rho-2)/2}|\nabla_M u|_M^2+\tfrac {\rho-2}2 (|u|^2+\eps)^{(\rho-4)/2}
(\nabla_Mu, u^2\nabla_M \bar{u}\\+|u|^2\nabla_Mu)_M
=(|u|^2+\eps)^{(\rho-4)/2}\left[\eps |\nabla_M u|_M^2+\frac \rho 2 |\bar{u}\nabla_M u|_M^2+\tfrac{\rho-2}2(\bar{u}\nabla_M u, u\nabla_M\bar{u})_M\right]
\end{multline*}
and then, setting $v=\text{Re} (\bar{u}\nabla_M u)$ and $w=\text{Im} (\bar{u}\nabla_M u)$, we get
$$(\nabla_Mu,\nabla_M u_\eps^{\star})_M=(|u|^2+\eps)^{(\rho-4)/2}\left[ \eps |\nabla_M u|_M^2+(\rho-1)|v|_M^2+|w|^2+i(\rho-2)(v,w)_M\right].
$$
Consequently $\text{Re}(\nabla_Mu,\nabla_M u_\eps^{\star})_M\ge 0$. By \eqref{26} then
$\text{Re } \int_M-\Delta_M u\,\, \bar{u_\eps^{\star}}\ge 0$ for all $\eps>0$. Since $u_\eps^{\star}\to |u|^{\rho-2}u$
pointwise in $M$, being uniformly bounded, we can pass to the limit as $\eps\to0^+$ and get \eqref{B9} for all $u\in C^2(M)$.
By density  our claim is proved.
By \eqref{B9} we immediately get that
$\text{Re }\langle A_{1,\rho}u, |u|^{\rho-2}\bar{u}\rangle_{L^\rho(M)}\ge \|u\|_{L^\rho(M)}^\rho$ for all $u\in W^{2,\rho}(M)$, from which
$A_{1,\rho}$ is injective and, by H\"{o}lder inequality,
$$\|u\|_{L^\rho(M)}\le \|A_{1,\rho}u\|_{L^\rho(M)}\qquad\text{for all
$u\in W^{2,\rho}(M)$},$$
so $\text{Rg}(A_{1,\rho})$ is closed. But $L^2(M)=\text{Rg}(A_{1,2})\subset \text{Rg}(A_{1,\rho})$, and $L^2(M)$ is dense in $L^\rho(M)$, hence
$\text{Rg}(A_{1,\rho})$ is dense, so $\text{Rg}(A_{1,\rho})=L^\rho(M)$ and $A_{1,\rho}$ is an isomorphism also when $1<\rho<2$.

We now consider the case $s=-1$. By \eqref{5.1} and \eqref{26} we have
$$\langle A_{s,\rho} u,v\rangle_{W^{1-s,\rho'}(M)}=\langle A_{-s,\rho'}v,u\rangle_{W^{1+s,\rho}(M)}$$
for all $s\in [-1,1]$, $u\in W^{s+1,\rho}(M)$ and $v\in W^{1-s,\rho'}(M)$, hence $A_{-1,\rho}$ is the Banach adjoint of
 $A_{1,\rho'}$. It follows then that $A_{-1,\rho}$ is an isomorphism for $1<\rho<\infty$.
Finally the result holds for  $s\in [-1,1]$ by \eqref{5.3} and interpolation theory (see \cite{berglofstrom}).
\end{proof}

\bigskip
{\bf Conflict of Interest:} The author declares that he has  no
conflict of interest.
\bigskip

\def\cprime{$'$}
\providecommand{\bysame}{\leavevmode\hbox to3em{\hrulefill}\thinspace}
\providecommand{\MR}{\relax\ifhmode\unskip\space\fi MR }
\providecommand{\MRhref}[2]{%
  \href{http://www.ams.org/mathscinet-getitem?mr=#1}{#2}
}
\providecommand{\href}[2]{#2}

\end{document}